\documentclass[12pt,a4paper]{amsart}

\usepackage[utf8]{inputenc}
\usepackage[american]{babel}
\usepackage{amsmath,amssymb,amsthm}
\usepackage{enumerate}
\usepackage{geometry}
\usepackage{dsfont}
\usepackage[pdfencoding=auto]{hyperref}

\numberwithin{equation}{section}

\title[A Noncommutative Transport Metric]{A Noncommutative Transport Metric and Symmetric Quantum Markov Semigroups as Gradient Flows of the Entropy}
\author[Wirth]{Melchior Wirth}
\address{{Institute of Mathematics, Department of Mathematics and Computer Science, Friedrich Schiller University Jena, 07737 Jena, Germany\newline Current address: Institute of Science and Technology Austria (IST Austria), Am Campus 1, 3400 Klosterneuburg, Austria}}
\email{melchior.wirth@ist.ac.at}
\thanks{The author was financially supported by the German Academic Scholarship Foundation (Studienstiftung des deutschen Volkes) and the German Research Foundation (DFG) via RTG 1523.}

\newcommand{\A}{\mathcal{A}}
\newcommand{\B}{\mathcal{B}}
\def\C{\mathcal{C}}
\newcommand{\D}{\mathcal{D}}
\newcommand{\IC}{\mathbb{C}}
\newcommand{\E}{\mathcal{E}}
\newcommand{\F}{\mathcal{F}}
\newcommand{\Gam}{\boldsymbol{\Gamma}}
\renewcommand{\H}{\mathcal{H}}
\newcommand{\I}{\mathcal{I}}
\renewcommand{\L}{\mathcal{L}}
\newcommand{\M}{\mathcal{M}}

\newcommand{\IN}{\mathbb{N}}
\newcommand{\p}{\mathfrak{p}}

\newcommand{\IR}{\mathbb{R}}

\newcommand{\W}{\mathcal{W}}
\newcommand{\IZ}{\mathbb{Z}}

\renewcommand{\epsilon}{\varepsilon}
\renewcommand{\phi}{\varphi}

\newcommand{\AC}{\mathrm{AC}}
\newcommand{\AM}{\mathrm{AM}}

\renewcommand{\div}{\operatorname{div}}
\newcommand{\Ent}{\mathrm{Ent}}
\DeclareMathOperator*{\esssup}{ess\,sup}

\newcommand{\EVI}{\mathrm{EVI}}
\newcommand{\GE}{\mathrm{GE}}

\newcommand{\loc}{\mathrm{loc}}

\renewcommand{\Re}{\operatorname{Re}}
\newcommand{\tr}{\operatorname{tr}}

\newcommand{\vol}{\mathrm{vol}}

\newcommand{\lra}{\longrightarrow}
\newcommand{\1}{\mathds{1}}

\newcommand{\abs}[1]{\lvert#1\rvert}
\newcommand{\Abs}[1]{\left\lvert#1\right\rvert}
\newcommand{\norm}[1]{\lVert#1\rVert}

\theoremstyle{plain}
\newtheorem{lemma}{Lemma}[section]
\newtheorem{theorem}[lemma]{Theorem}
\newtheorem{proposition}[lemma]{Proposition}
\newtheorem{corollary}[lemma]{Corollary}

\theoremstyle{definition}
\newtheorem{definition}[lemma]{Definition}

\theoremstyle{remark}
\newtheorem{remark}[lemma]{Remark}
\newtheorem{example}[lemma]{Example}

\begin{document}

\begin{abstract}
We study quantum Dirichlet forms and the associated symmetric quantum Markov semigroups on noncommutative $L^2$ spaces. It is known from the work of Cipriani and Sauvageot that these semigroups induce a first order differential calculus, and we use this differential calculus to define a noncommutative transport metric on the set of density matrices. This construction generalizes both the $L^2$-Wasserstein distance on a large class of metric spaces as well as the discrete transport distance introduced by Maas, Mielke, and Chow--Huang--Li--Zhou. Assuming a Bakry--Émery-type gradient estimate, we show that the quantum Markov semigroup can be viewed as a metric gradient flow of the entropy with respect to this transport metric. Under the same assumption we also establish that the set of density matrices with finite entropy endowed with the noncommutative transport metric is a geodesic space and that the entropy is semi-convex along these geodesics.
\end{abstract}

\maketitle

\tableofcontents
\addtocontents{toc}{\protect\setcounter{tocdepth}{1}}

\section*{Introduction}

Since the seminal work of Jordan--Kinderlehrer--Otto \cite{JKO98} and Otto \cite{Ott01} it is known that the space of probability measures on $\IR^d$ endowed with the $L^2$-Wasserstein metric can, at least formally, be viewed as an infinite-dimensional Riemannian manifold and the heat flow as gradient flow of the Boltzmann entropy.

This insight spawned a lot of subsequent activity, extending the gradient flow characterization to various other geometric settings (see \cite{AGS14a,AS18,Erb10,GKO13,Jui14,OS09}) as well as to other evolution equations (see \cite{Erb16,Ott01}).

Not the least, the characterization of the heat flow as gradient flow of the entropy played a crucial role in the work of Ambrosio, Gigli and Savaré \cite{AGS14a,AGS14b,AGS15} that provided an understanding of the connection between synthetic lower bounded Ricci curvature  bounds in the sense of Lott--Sturm--Villani \cite{LV09,Stu06a,Stu06b} and Bakry--Émery \cite{BE85}.

In contrast, for discrete spaces respectively evolution equations with nonlocal generators, the Monge--Kantorovich formulation of transport distances has turned out not to be useful in this direction: While one can define the Wasserstein distances $W_p$ for an arbitrary metric $d$, absolutely continuous curves in the Wasserstein space are constant when $d$ is discrete and thus there are no non-trivial gradient flows.

However, Maas \cite{Maa11}, Mielke \cite{Mie11}, and Chow--Huang--Li--Zhou \cite{CHLZ12} independently defined a discrete transport metric $\W$ on the set of probability densities over a finite graph such that the heat flow for the graph Laplacian coincides with the gradient flow of the entropy with respect to $\W$. Instead of the Monge--Kantorovich optimal transport problem, their approach is based on a discrete version of the Benamou--Brenier formula \cite{BB00}, which gives an equivalent description of the $L^2$-Wasserstein metric on probability measures on Euclidean space.

This new metric has already proven to be very fertile. On the one hand, the gradient flow characterization has been generalized to the heat equation for generators of jump processes \cite{Erb14} as well as a variety of other evolution equations on graphs \cite{CLZ18,EM14,EFLS16,LM13}. On the other, (variants of) the metric $\W$ has been used (among other things) to define lower Ricci curvature bounds for graphs \cite{EM12} and to study a new discrete version of the nonlinear Schrödinger equation \cite{CLZ19}

Moreover, in recent years have seen new activity in the study of matrix-valued optimal transport with several groups studying a version of the metric $\W$ for matrix algebras (see \cite{CM14,CM17a,CGGT17,CGT18a,MM17}); and, independently, Brenier \cite{Bre17a,Bre18} discovered a surprising connection between matrix-valued optimal transport and fluid dynamics.  Notably, Carlen and Maas \cite{CM14,CM17a} showed that the metric $\W$ allows to view the flow determined by a finite-dimensional Markovian quantum master equation as gradient flow of the von Neumann entropy.

Both in the case of graphs and matrix algebras, all work so far has been limited to a finite-dimensional setting and the problem of extending it to the infinite-dimensional case has been raised in several of the aforementioned articles. This problem is solved in the present article. More precisely, we give a definition of $\W$ and a characterization of the flow defined by a Markovian quantum master equation as gradient flow of the entropy for \emph{quantum Dirichlet forms} -- a setting that generalizes many of the ones above -- based on the first order differential calculus developed by Cipriani and Sauvageot \cite{CS03}.

In particular, this article gives the first unified approach to the results in the local case (for example the heat equation on Euclidean space, manifolds, infinitesimally Riemannian metric measure spaces) on the one hand and non-local case (e.g. heat equation on graphs, for fractional powers of the Laplacian) on the other hand, which could only be treated by analogy until now.

On the noncommutative side, this setting does not only treat infinite-dimensional quantum systems, but also some classical examples of noncommutative geometry such as the noncommutative heat semigroup on the noncommutative torus. This could open the door to a theory of Ricci curvature for noncommutative spaces, a concept that has been notoriously elusive in noncommutative geometry until now.

Let us shortly comment on the differences to prior work. In contrast to the case of metric measure spaces, many powerful tools of optimal transport are not available here, and in the Benamou--Brenier formulation, the continuity equation depends linearly on the measure density in the local case, while in our setting, it is in general a nonlinear equation in the density.

These problems have already been tackled successfully in the non-local case of graphs and jump processes, however, the necessary analysis of monotonicity and convexity properties turns out more difficult in the noncommutative setting as operator monotonicity and operator convexity are decidedly more rigid notions than their commutative counterparts.

Compared to previous work on matrix-valued optimal transport, we deal not only with quantum Markov semigroups on an infinite-dimensional space, but in general with semigroups that are not uniformly bounded. This means that the Lindblad characterization of generators of quantum Markov semigroups, which is a central tool for the construction of $\W$ for matrix algebras, is no longer applicable. Moreover, the density matrices in our setting are not only operators on an infinite-dimensional space (as opposed to matrices), but in general unbounded ones. For that reason our analysis requires a careful adaptation of classical tools for operator monotonicity and convexity, which are usually only developed for bounded operators.  Furthermore, it is only in the infinite-dimensional case that the full power of the theory of gradient flows in metric spaces is needed, whereas in finite dimensions the elementary theory of gradient flows on Riemannian manifolds suffices.

Let us therefore stress that the infinite-dimensional setting does not only include new classes of examples, but that it is also necessary for unified treatment of the local and non-local case is possible since locality is a purely infinite-dimensional phenomenon (incidentally, it did not appear in the seminal work of Beurling--Deny \cite{BD58} on Dirichlet forms, as they only treated the finite-dimensional case).

Among other possible applications, we hope to lay the ground for a systematic study of displacement convexity of the entropy for infinite-dimensional quantum systems, a topic which has already proven useful for convergence results in the finite-dimensional case \cite{CM17a}.

Moreover, the theory developed here could provide a framework for approximation results of smooth spaces or infinite-dimensional systems by discrete spaces or finite-dimensional systems, which so far have only been treated in some particular cases \cite{GM13a,Tri17}.

Let us summarize the content of this article in some more detail. In Section \ref{Basics} we recall some basic facts about noncommutative integration and quantum Dirichlet forms, including the first order differential calculus of Cipriani and Sauvageot. One classical example of a (commutative) Dirichlet form is the Dirichlet energy on $\IR^n$, that is,
\begin{align*}
\E(u)=-\int u\Delta u\,dx.
\end{align*}
By partial integration, $\E$ can equivalently be expressed as
\begin{align*}
\E(u)=\int_{\IR^n}\abs{\nabla u}^2,
\end{align*}
and $\nabla$ is a derivation in the sense that it satisfies the product rule $\nabla(uv)=u\nabla v+v\nabla u$.

Now, if $\E$ is a Dirichlet form on the noncommutative $L^2$ space $L^2(\M,\tau)$ with generator $\L$, the first order differential calculus of Cipriani and Sauvageot (Theorem \ref{diff_calc}) asserts that it can be represented in the same way, that is, there exists a Hilbert bimodule $\H$ and an operator $\partial$ with values in $\H$ such that
\begin{align*}
\E(a)=\norm{\partial a}_\H^2
\end{align*}
and $\partial$ satisfies the product rule $\partial (ab)=a\partial b+(\partial a)b$. Notice that unlike in the case of the Dirichlet energy on $\IR^n$, the left and right multiplication on $\H$ may be different.

A priori, the left and right multiplication on $\H$ are only defined for elements in $D(\E)\cap \M$. In Section \ref{energy_dominance} we study when they can be extended to all of $\M$. It turns out that this question is closely related to the carré du champ
\begin{align*}
\Gam(a)(x)=\langle x\partial a,\partial a\rangle_\H.
\end{align*}
The carré du champ is $\sigma$-weakly continuous (in the commutative case, this holds if and only if the energy measure is absolutely continuous with respect to the reference measure) for all $a\in D(\E)$ if and only if the left and right multiplication have a $\sigma$-weakly continuous extension to all of $\M$ (Theorem \ref{char_energy_dom}). For the rest of the article we work under the standing assumption that this property holds.

In the classical Otto calculus on Euclidean space, the tangent space of $\mathcal{P}(\IR^n)$ at $\rho\, dx$ is identified with $H^1(\IR^n, \rho\,dx)$. In particular, the norm of the tangent vector $\psi$ is given by $\int \abs{\nabla \psi}^2\rho\,dx$. In our framework, the latter could be replaced by $\langle \rho\partial a,\partial a\rangle_\H$ or $\langle (\partial a)\rho,\partial a\rangle_\H$ or any ``mean'' of these two extreme cases (these two examples correspond to the left and right trivial mean).

We study a class of such means in Section \ref{A_E}. For a suitable mean $\theta$ we define $\hat\rho=\theta(L(\rho),R(\rho))$. A key technical role in this study plays the Lipschitz algebra $\A_\theta$ of bounded elements $a$ of $D(\E)$ such that
\begin{align*}
\norm{\partial a}_\rho^2=\langle \hat\rho\, \partial a, \partial a\rangle_\H
\end{align*}
is bounded on the space of density matrices. In this case, $\norm{\partial a}_\rho^2$ depends upper semicontinuously on $\rho$ (Theorem \ref{theta_norm_usc}). The elements of $\A_\theta$ can be tested uniformly against $\hat\rho$ for all density matrices $\rho$, which makes them a suitable choice of test ``function'' in the continuity equation discussed below.

While the discussion of Section \ref{A_E} is rather general, we will later focus on the logarithmic mean
\begin{align*}
\theta(s,t)=\frac{s-t}{\log s-\log t},
\end{align*}
which was already identified by Maas and Mielke as the correct mean to study gradient flows of the entropy $\Ent(\rho)=\tau(\rho\log \rho)$.

In Section \ref{Maas-Wasserstein} we introduce the noncommutative transport metric $\W$ via an analogue of the Benamou--Brenier formula (Definition \ref{def_W}). It is given as the length metric associated with the action functional 
\begin{align*}
(\rho_t)\mapsto \int \langle \hat\rho_t D\rho_t,D\rho_t\rangle_\H\,dt
\end{align*}
defined on a class of admissible curves, where the velocity vector field $(D\rho_t)$ is determined by the abstract continuity equation
\begin{align*}
\dot\rho_t=\partial^\ast (\hat\rho_t D\rho_t)
\end{align*}
in a suitable weak sense. Further we establish some basic properties such as the convexity of $\W$ (Lemma \ref{W_convex}) and lower semicontinuity of the action functional with respect to pointwise weak convergence in $L^1$ (Theorem \ref{energy_lsc}).

Section \ref{entropy_information}, which is quite technical in nature, deals with the entropy and the Fisher information. The latter is formally given by $\I(\rho)=\E(\rho,\log \rho)$, but this expression suffers from several regularity issues (not all density matrices are in the domain of $\E$, the logarithm is not a Lipschitz function), so we spend much of this section giving a rigorous definition via approximation and showing that several different approximations yield the same result. Then we go on to show that solutions of the quantum master equation are admissible curves in the definition of $\W$ (Proposition \ref{heat_flow_admissible}, Corollary \ref{heat_flow_L1}) and the entropy dissipation rate along these curves is given by the Fisher information (Proposition \ref{entropy_Fisher_decay}).

In Section \ref{heat_gradient_flow}, we come to the identification of the flow defined by the Markovian quantum master equation with the gradient flow of the entropy. First, we introduce the gradient estimate $\GE(K,\infty)$ in Definition \ref{def_GE}. If $\E$ is the Dirichlet energy on a Riemannian manifold, then $\GE(K,\infty)$ reduces to the well-known Bakry--Émery gradient estimate
\begin{align*}
\Gamma(P_t f)\leq e^{-2Kt} P_t\Gamma(f),
\end{align*}
which is equivalent to $\mathrm{Ric}\geq K$.

After a technical subsection singling out a suitable class of regular curves, we can then show in Theorem \ref{EVI_gradient_flow} that $\GE(K,\infty)$ implies that the flow defined by the Markovian quantum Master equation
\begin{align*}
\dot \rho_t =-\L\rho_t
\end{align*}
is an $\EVI_K$ gradient flow of the entropy, that is,
\begin{align*}
\frac 1 2\frac{d^+}{dt}\W(P_t\rho,\sigma)^2+\frac K 2\W(P_t\rho,\sigma)^2+\Ent(P_t\rho)\leq \Ent(\sigma).
\end{align*}
In Section \ref{geodesic_convex} we study consequences of the gradient flow characterization with regard to semi-convexity of the entropy along geodesics. The $\EVI_K$ implies that the distance $\W$ between two density matrices with finite entropy can be realized by a sequence of curves with uniformly bounded entropy. Combining this fact with the weak $L^1$-compactness of sublevel sets of the entropy, we conclude that density matrices with finite entropy are joined by a minimizing geodesic (Theorem \ref{D(Ent)_geodesic}). Finally, by an abstract result on gradient flows, the entropy is $K$-convex along geodesics, that is,
\begin{align*}
\Ent(\rho_t)\leq (1-t)\Ent(\rho_0)+t\Ent(\rho_1)-\frac K 2t(1-t)\W(\rho_0,\rho_1)^2.
\end{align*}
The relation between $\GE(K,\infty)$, $\EVI_K$ and geodesic $K$-convexity are summarized in Theorem \ref{main_theorem_summary}.

The content of this article was included in the author's PhD thesis at the University of Jena.

\subsection*{Acknowledgments}
The author would like to express his gratitude to Daniel Lenz for the support and helpful advice as his supervisor. He wants to thank Matthias Erbar and Jan Maas for fruitful discussions on the topic of this article, Aljosha Sukeylo for the translation of the article \cite{Tik87} and Simon Puchert for the proof of Lemma \ref{increasing_Lip}.

After the project had grown, we learned that ideas for a related, but slightly different definition of the metric $\mathcal{W}$ had also been developed independently by David Hornshaw \cite{Hor18}. The author is grateful for the exchange of draft versions and discussions.

\section{Quantum Dirichlet Forms}\label{Basics}

In this section we give a short overview over the theory of noncommutative integration and noncommutative Dirichlet forms, and show how some examples encountered later fit into that framework. In particular, we review the first order differential calculus developed by Cipriani and Sauvageot, which will be a central tool later on.

To begin, let us recall some basics of noncommutative integration theory as developed in \cite{Seg53a,Seg53b}; a good overview is given in \cite{PX03}.

An algebra $\M$ of bounded operators on a Hilbert space $H$ is called \emph{von Neumann algebra} if it is closed under taking adjoints, contains $1$ and is closed in the weak operator topology. The set $\M'=\{x\in B(H)\mid xy=yx\text{ for all }y\in\M\}$ is called the \emph{commutant of $\M$}. By the bicommutant theorem, a unital $\ast$-subalgebra of $B(H)$ is a von Neumann algebra if and only if $\M=\M^{\prime\prime}$. The set of all positive operators in $\M$ is denoted by $\M_+$.

Let $\M$ be a von Neumann algebra. A map $\tau\colon \M_+\lra[0,\infty]$ is called a \emph{weight} if $\tau(\lambda x)=\lambda \tau(x)$ and $\tau(x+y)=\tau(x)+\tau(y)$ for all $\lambda\geq 0$, $x,y\in\M_+$ (with the convention $0\cdot\infty=0$).

The weight $\tau$ is called 
\begin{itemize}
\item \emph{normal} if $\tau(\sup_i x_i)=\sup_i \tau(x_i)$ for every increasing net $(x_i)$ in $\M_+$,
\item \emph{semi-finite} if $\{x\in \M_+\mid \tau(x)<\infty\}$ generates $\M$ as von Neumann algebra,
\item \emph{faithful} if $\tau(x^\ast x)=0$ implies $x=0$, 
\item \emph{tracial} or a \emph{trace} if $\tau(x^\ast x)=\tau(xx^\ast)$ for all $x\in \M$.
\end{itemize}

We say that $\tau$ is an \emph{n.s.f. trace} if it is an normal, semi-finite, faithful, tracial weight, and call the pair $(\M,\tau)$ a \emph{tracial von Neumann algebra}. Every n.s.f. trace $\tau$ induces a faithful normal representation $\pi_\tau$ on a Hilbert space $H_\tau$. We will routinely identify $\M$ with $\pi_\tau(\M)$.

A closed, densely defined operator $x$ is said to be \emph{affiliated} with $\M$ if $xu=ux$ for every unitary $u\in \M'$. A self-adjoint operator $x$ is affiliated with $\M$ if and only if $\phi(x)\in \M$ for every bounded Borel function $\phi\colon \IR\lra \IR$. An affiliated operator $x$ is called \emph{$\tau$-measurable} if $\tau(\1_{(\lambda,\infty)}(\abs{x}))<\infty$ for some $\lambda\geq 0$. The set of all $\tau$-measurable operators is denoted by $L^0(\M,\tau)$.

The trace $\tau$ can be extended to the set of positive $\tau$-measurable operators via
\begin{align*}
\tau(x)=\int_{[0,\infty)}\lambda\,d(\tau\circ e)(\lambda),
\end{align*}
where $e$ denotes the spectral measure of $x$. Equivalently, $\tau(x)=\sup_{n\in\IN}\tau(x\wedge n)$.

The noncommutative $L^p$ spaces are defined as 
\begin{align*}
L^p(\M,\tau)=\{x\in L^0(\M,\tau)\mid \tau(\abs{x}^p)<\infty\}
\end{align*}
for $p\in[1,\infty)$ and endowed with the norm $\norm{\cdot}_p=\tau(\abs{\cdot}^p)^{1/p}$. For $p=\infty$ one sets $L^\infty(\M,\tau)=\M$.

The space $L^p(\M,\tau)$ coincides with the completion of $\{x\in\M\mid\tau(\abs{x}^p)<\infty\}$ with respect to $\norm\cdot_p$. In particular, it is a Banach space. Moreover, $L^2(\M,\tau)\cong H_\tau$ with the action of $\M$ on $L^2(\M,\tau)$ given by left multiplication. In the following, we will always identify $\M$ with its image in $B(L^2(\M,\tau))$.

\begin{example}
If $(X,\B,m)$ is a localizable (for example $\sigma$-finite) measure space, then $L^\infty(X,m)$ is a von Neumann algebra (acting on $L^2(X,m)$ by multiplication) and the functional
\begin{align*}
\tau_m\colon L^\infty_+(X,m)\lra[0,\infty],\,\tau_m(f)=\int_X f\,dm
\end{align*}
is an n.s.f. trace. Moreover, every commutative tracial von Neumann algebra arises in this way.

The space $L^p(L^\infty(X,m),\tau_m)$ is isometrically isomorphic to $L^p(X,m)$, and the isomorphism can be chosen in such a way that it is consistent for all $p\in[1,\infty]$.
\end{example}

\begin{remark}
If $A$ is a $C^\ast$-algebra and $\tau\colon A_+\lra[0,\infty]$ a lower semicontinuous, semi-finite, faithful trace, one also gets a GNS representation $\pi_\tau$ on a Hilbert space $H_\tau$. In this case, $L^\infty(A,\tau)=\pi_\tau(A)''$ is a von Neumann algebra and $\tau$ extends to an n.s.f. weight on $L^\infty(A,\tau)$. One sets $L^p(A,\tau):=L^p(L^\infty(A,\tau),\tau)$.

From a theoretical point of view one can therefore always assume to be given a tracial von Neumann algebra $(\M,\tau)$, but in the examples it will sometimes be more convenient to give a description in terms of $(A,\tau)$.
\end{remark}

Denote by $L^p_h(\M,\tau)$ the set of self-adjoint (as operators on $H_\tau$) elements of $L^p(\M,\tau)$. One advantage of the definition via affiliated operators over that as abstract completion is the fact that we can use functional calculus for elements of $L^p_h(\M,\tau)$.

For the next lemma recall that for a nonempty, closed, convex subset $C$ of a Hilbert space $H$ and $x\in H$ there is a unique element $y\in C$ with $\norm{x-y}=\inf_{z\in C}\norm{x-z}$. The map $P_C\colon x\mapsto y$ is called (metric) projection onto $C$. The element $P_C(x)$ can alternatively be characterized as the unique $y\in C$ such that
\begin{align*}
\Re\langle x-y,z-y\rangle\leq 0
\end{align*}
for all $z\in C$.

We write $\alpha\wedge \beta=\min\{\alpha,\beta\}$ and $\alpha\vee\beta=\max\{\alpha,\beta\}$ for $\alpha,\beta\in\IR$. If $x$ is a self-adjoint operator, $x\wedge\alpha$ stands for the application of the function $\min\{\,\cdot\,,\alpha\}$ to $x$, which is the infimum of $x$ and $\alpha1$ in the (commutative) unital $C^\ast$-algebra generated by $x$.

\begin{lemma}
Let $(\M,\tau)$ be a tracial von Neumann algebra and let $C$ be the closure of $\{x\in L^2_h(\M,\tau)\cap\M\mid x\leq 1\}$ in $L^2(\M,\tau)$. Then $C$ is convex and the projection $P_C$ onto $C$ is given by $P_C(a)=a\wedge 1$ for all $a\in L^2_h(\M,\tau)$.
\end{lemma}
\begin{proof}
It is easy to see that $C$ is convex. For $a\in L^2_h(\M,\tau)$ let $a_n=(a\wedge 1)\vee(-n)$. Then $a_n\in L^2_h(\M,\tau)\cap \M$, $a_n\leq 1$ and $a_n\to a\wedge 1$ in $L^2(\M,\tau)$, hence $a\wedge 1\in C$. If $b\in \M\cap L^2_h(\M,\tau)$ with $b\leq 1$, then
\begin{align*}
\tau((a-a\wedge 1)(b-a\wedge 1))&=\tau((a-1)_+^{1/2}(b-a\wedge 1)(a-1)_+^{1/2})\\
&\leq \tau((a-1)_+^{1/2}(1-a\wedge 1)(a-1)_+^{1/2})\\
&=\tau((a-1)_+(a-1)_-)\\
&=0.
\end{align*}
For arbitrary $b\in C$, the inequality above follows by continuity. Thus $P_C(a)=a\wedge 1$.
\end{proof}
Now we can turn to the theory of Dirichlet forms and Markovian semigroups in the noncommutative setting. For some basic references see \cite{AH77,DL92}, for the first order differential calculus described below see \cite{CS03} and the expository article \cite{Cip08}.

A quadratic form $\E\colon L^2(\M,\tau)\lra[0,\infty]$ is \emph{real} if $\E(a^\ast)=\E(a)$ for all $a\in L^2(\M,\tau)$ and \emph{Markovian} if $\E(a\wedge 1)\leq \E(a)$ for all $a\in L^2_h(\M,\tau)$. The lemma above shows that the cut-off $a\wedge 1$ can be understood either as an application of functional calculus or as projection in $L^2(\M,\tau)$.

By the next lemma (see \cite[Proposition 2.12]{DL92} and \cite[Theorem 10.2]{CS03}), Markovian forms automatically satisfy a stronger contraction property with respect to Lipschitz functional calculus.

\begin{lemma}
A closed real quadratic form $\E\colon L^2(\M,\tau)\lra[0,\infty]$ is Markovian if and only if $\E(f(a))\leq \E(a)$ for all $a\in L^2_h(\M,\tau)$ and all $1$-Lipschitz functions $f\colon \IR\lra\IR$ with $f(0)=0$.
\end{lemma}

For $n\in\IN$ denote by $\tr_n$ the normalized trace on $M_n(\IC)$ and let $\tau_n=\tau\otimes \tr_n$ on $(\M\otimes M_n(\IC))_+\cong M_n(\M)_+$, that is,
\begin{align*}
\tau_n\colon M_n(\M)_+\lra[0,\infty],\,\tau_n((a_{ij}))=\frac 1 n\sum_{i=1}^n \tau(a_{ii}).
\end{align*}
The quadratic form $\E$ can be extended to $L^2(M_n(\M),\tau_n)$ via
\begin{align*}
\E_n\colon L^2(M_n(\M),\tau_n)\lra[0,\infty],\,\E_n((a_{ij}))=\sum_{i,j=1}^n \E(a_{ij}).
\end{align*}
We say that $\E$ is \emph{completely Markovian} if $\E_n$ is Markovian for all $n\in\IN$.

A lower semicontinuous, densely defined, real, completely Markovian quadratic form $\E$ on $L^2(\M,\tau)$ is called \emph{completely Dirichlet form} on $(\M,\tau)$.

\begin{remark}
For every quadratic form $q$ on a Hilbert space $H$ there is an associated sesquilinear form $\tilde q$ defined as
\begin{align*}
\tilde q\colon D(q)\times D(q)\lra \IC,\,\tilde q(u,v)=\frac 1 4\sum_{k=0}^3 i^k q(u+i^k v),
\end{align*}
where $D(q)=\{u\in H\mid q(u)<\infty\}$. We will use these two points of view interchangeably and write $q$ for both of these maps.
\end{remark}

\begin{remark}
In \cite{CS03}, an additional condition called \emph{regularity} is imposed in most results. This property depends not only on the form $\E$, but also on the choice of some $C^\ast$-subalgebra of $\M$. Every completely Dirichlet form $\E$ is regular with respect to the norm closure of $D(\E)\cap \M$.
\end{remark}

\begin{example}
Let $(X,\B,m)$ be a localizable measure space. Every Markovian form on $L^\infty(X,m)$ is completely Markovian so that Dirichlet forms on $L^2(X,m)$ in the sense of Beurling--Deny \cite{BD58,BD59} can be identified with completely Dirichlet forms on $L^2(L^\infty(X,m),\tau_m)$.
\end{example}

There is a bijective correspondence between quantum Dirichlet forms and quantum sub-Markov semigroups on $(\M,\tau)$ analogous to the commutative case: The semigroup $(P_t)$ generated by a positive self-adjoint operator $\L$ on $L^2(\M,\tau)$ is sub-Markovian, that is, $0\leq P_t(a)\leq 1$ for $0\leq a\leq 1$, if and only if the quadratic form generated by $\L$ is a Markovian form (see \cite[Theorems 2.7, 2.8]{AH77} in the finite case and \cite[Theorems 2.13, 3.3]{DL92} in the semi-finite case).

Moreover, $(P_t)$ extends uniquely to strongly continuous semigroups on $L^p(\M,\tau)$ for $p\in[1,\infty)$ and to a quantum sub-Markov semigroup on $\M$. We will usually denote these extensions by the same symbol, occasionally also writing $(P_t^{(p)})$ when the space on which the semigroup acts is important. Similarly, $L^{(p)}$ denotes the generator of $(P_t)$ on $L^p(\M,\tau)$

The curve $(P_t^{(p)} a)_{t\geq 0}$ is the unique (mild) solution of the initial value problem for the Markovian quantum master equation
\begin{align*}
\begin{cases}\dot x_t=-\L^{(p)} x_t,\\
x_0=a
\end{cases}
\end{align*}
in $L^p(\M,\tau)$.

The semigroup $(P_t)_{t\geq 0}$ is called \emph{conservative} if $P_t 1=1$. As we want to study the evolution on density matrices, conservativeness is a natural assumption, so we reserve a special name for the associated Dirichlet forms (motivated by the term quantum Markov semigroup for the corresponding semigroup on $\M$).

\begin{definition}[Quantum Dirichlet form]
A completely Dirichlet form $\E$ is called \emph{quantum Dirichlet form} if the associated semigroup $(P_t)$ is conservative.
\end{definition}

The following representation theorem for completely Dirichlet forms by Cipriani and Sauvageot (see \cite[Theorems 4.7, 8.2, 8.3]{CS03}) is central to our investigations.

\begin{theorem}[First order differential calculus]\label{diff_calc}
Let $\E$ be a quantum Dirichlet form on the tracial von Neumann algebra $(\M,\tau)$ and $\C=D(\E)\cap \M$.

Then $\C$ is a $\ast$-algebra and there exist a Hilbert space $\H$, commuting non-degenerate $\ast$-representations $L$ of $\C$ and $R$ of $\C^\circ$ on $\H$, an anti-linear isometric involution $J\colon\H\lra \H$, and a closed operator $\partial \colon D(\E)\lra \H$ such that 
\begin{itemize}
\item $\{L(a)\partial b\mid a,b\in\C\}$ is dense in $\H$,
\item $J$ intertwines $L$ and $R$: $L(a)=J R(a)^\ast J$ for all $a\in\C$,
\item $\partial$ is $J$-real: $J\partial a=\partial(a^\ast)$ for all $a\in\C$,
\item $\partial$ satisfies the Leibniz rule: $\partial (ab)=L(a)\partial b+R(b)\partial a$ for all $a,b\in\C$,
\item $\E$ can be represented by $\partial$: $\E(a)=\norm{\partial a}_{\H}^2$ for all $a\in D(\E)$.
\end{itemize}
If $(\tilde\partial, \tilde \H,\tilde L,\tilde R,\tilde J)$ is another quintuple with the same properties, then there exists a unitary map $U\colon \H\lra\tilde \H$ such that
\begin{itemize}
\item $U\partial =\tilde \partial$,
\item $UL=\tilde L$, $UR=\tilde R$,
\item $UJ=\tilde J U$.
\end{itemize}
\end{theorem}
In the sense of this theorem, we can speak of \emph{the} first order differential calculus associated with $\E$. The $\ast$-representations $L,R$ are to be understood as left and right multiplication of $\C$ on $\H$. Accordingly, we will write $a\cdot\xi$ and $\xi\cdot b$ for $L(a)\xi$ and $R(b)\xi$, respectively.

\begin{remark}
In the theorem, $\C^\circ$ denotes the opposite algebra of $\C$, that is, the $\ast$-algebra with same underlying vector space and involution, but with multiplication given by $a\circ b= ba$ for $a,b\in \C$.
\end{remark}

\begin{remark}
Instead of conservativeness, it suffices to assume that the killing term of $\E$ vanishes in the sense of \cite[Theorem 8.1]{CS03}.
\end{remark}

An important consequence of the product rule for the first order differential calculus is a (two-variable) chain rule. For that purpose, let
\begin{align*}
\tilde f\colon I\times I\lra \IR,\,\tilde f(s,t)=\begin{cases}\frac{f(s)-f(t)}{s-t}&\text{if } s\neq t,\\f'(s)&\text{if } s=t\end{cases}
\end{align*}
for $f\in C^1(I)$. The function $\tilde f$ is sometimes called the quantum derivative of $f$. With this notation, the chain rule reads as follows (\cite[Lemma 7.2]{CS03}).
\begin{lemma}[Chain rule]
If $f\in C^1(\IR)$ has bounded derivative and $f(0)=0$, then
\begin{align*}
\partial f(a)=\tilde f(L(a),R(a))\partial a.
\end{align*}
for all $a\in D(\E)_h$.
\end{lemma}
If $\E$ is strongly local, then $L=R$ and one recovers the usual chain rule $\partial f(a)=f'(a)\partial a$.

\begin{example}[Weighted graphs]\label{ex_graphs}
Let $X$ be a countable set, $m\colon X\lra (0,\infty)$ and $b\colon X\times X\lra[0,\infty)$ such that
\begin{itemize}
\item $b(x,x)=0$ for all $x\in X$,
\item $b(x,y)=b(y,x)$ for all $x,y\in X$,
\item $\sum_y b(x,y)<\infty$ for all $x\in X$.
\end{itemize}
The triple $(X,b,m)$ is called a \emph{weighted} graph (compare \cite{KL10,KL12}). Often one allows for an additional killing weight $c\colon X\lra[0,\infty)$, but the associated Dirichlet form will never be conservative if $c\neq 0$, so we drop it from the beginning.

The associated Dirichlet form with Neumann boundary conditions is
\begin{align*}
\E^{(N)}\colon \ell^2(X,m)\lra [0,\infty],\,\E^{(N)}(u)=\frac 1 2\sum_{x,y}b(x,y)\abs{u(x)-u(y)}^2.
\end{align*}
The associated Dirichlet form with Dirichlet boundary conditions $\E^{(D)}$ is the closure of the restriction of $\E^{(D)}$ to $C_c(X)$.

The first order differential calculus associated with $\E^{(N)}$ is given by $\H=\ell^2(X\times X,\frac 1 2 b)$, $(u\cdot\xi)(x,y)=u(x)\xi(x,y)$, $(\xi\cdot v)(x,y)=\xi(x,y) v(y)$, $\partial u(x,y)=u(x)-u(y)$ and $(J\xi)(x,y)=-\overline{\xi(y,x)}$.

The first order differential calculus associated with $\E^{(D)}$ is obtained by suitable restriction.
\end{example}

\begin{example}[Riemannian manifolds]\label{ex_manifolds}
Let $(M,g)$ be a complete Riemannian manifold and $\E$ the standard Dirichlet integral
\begin{align*}
\E\colon L^2(M)\lra [0,\infty],\,\E(u)=\begin{cases}\int_M \abs{\nabla u}^2\,d\vol_g&\text{if } \nabla u\in L^2(M),\\\infty& \text{otherwise}.\end{cases}
\end{align*}
The first order differential calculus for $\E$ is given by $\H=L^2(M;T M)$, $(u\xi)(x)=(\xi u)(x)=u(x)\xi(x)$, $\partial=\nabla$ and $J\xi=\overline{\xi}$.
\end{example}

\begin{example}[Metric measure spaces]\label{ex_RCD-space}
If $(X,d,m)$ is an infinitesimally Hilbertian metric measure space (see \cite{AGS14b}) and $\E$ the associated Dirichlet form (twice the Cheeger energy), then the first order differential calculus described above coincides with first order differential calculus developed in \cite{Gig14a}.
\end{example}

Notice that the crucial difference between Example \ref{ex_graphs} on the one hand and Examples \ref{ex_manifolds}, \ref{ex_RCD-space} on the other hand is that left and right multiplication on $\H$ coincide for the Dirichlet forms on Riemannian manifolds and metric measure spaces while they differ for graphs. More generally, left and right multiplication coincide in the commutative setting whenever $\E$ is a strongly local regular Dirichlet form (see \cite[Theorem 2.7]{IRT12}).

\begin{example}[Noncommutative torus]\label{ex_nc_heat_semigroup}
Let $\theta\in (0,1)$ be irrational and let $U,V\in B(H)$ be unitaries with $VU=e^{2\pi i \theta}UV$. The unital $C^\ast$-algebra $A_\theta$ generated by $U,V$ is called \emph{noncommutative torus} (and, up to $\ast$-isomorphism, it is indeed independent of the choice of $U$, $V$). Let $\A_\theta$  be the linear hull of $\{U^m V^n\mid m,n\in\IZ\}$, which is clearly a dense $\ast$-subalgebra of $A_\theta$.

The map
\begin{align*}
\tau\colon \A_\theta\lra\IC,\,\tau(U^mV^n)=\delta_{m,0}\delta_{n,0}
\end{align*}
extends to a tracial state on $A_\theta$. Furthermore, the map 
\begin{align*}
P_t\colon \A_\theta\lra \A_\theta,\,P_t(U^m V^n)=e^{-t(m^2+n^2)}U^m V^n
\end{align*}
extends to a bounded linear operator on $L^2(A_\theta,\tau)$ and $(P_t)_{t\geq 0}$ is a $\tau$-symmetric quantum Markov semigroup, called \emph{noncommutative heat semigroup}. The associated Dirichlet form $\E$ acts on $\A_\theta$ as
\begin{align*}
\E\left(\sum_{m,n\in\IZ}\alpha_{m,n}U^m V^n\right)=\sum_{m,n\in\IZ}(m^2+n^2)\abs{\alpha_{m,n}}^2.
\end{align*}

Let $\partial_1,\partial_2\colon \A_\theta\lra \A_\theta$ be defined by $\partial_1(U^m V^n)=im U^m V^n$ and $\partial_2 (U^m V^n)=inU^m V$, and let $\H$ be the the closed linear hull of $\{(a\partial_1(b),a\partial_2(b))\mid a,b\in \A_\theta\}$ in $L^2(A_\theta,\tau)\oplus L^2(A_\theta,\tau)$.

The first order differential calculus associated with $\E$ is given by $\partial=\partial_1\oplus \partial_2$, $L(a)(u,v)=(au,av)$, $R(b)(u,v)=(ub,vb)$, $J(u,v)=-(v^\ast,u^\ast)$.
\end{example}

\begin{example}[Fermionic Clifford algebra]
Let $H$ be an infinite-dimensional, se\-pa\-rable real Hilbert space and $\IC\ell(H)$ the Clifford $C^\ast$-algebra over $H$ (see \cite{SS64}). It is well-known that $\IC\ell(H)$ is a simple $C^\ast$-algebra with a unique tracial state $\tau$. The von Neumann algebra $L^\infty(\IC\ell(H),\tau)$ is the hyperfinite type II$_1$ factor.

Let $(e_i)_{i\in\IN}$ be an orthonormal basis of $H$. The linear hull of all products of the form $e_{i_1}\dots e_{i_k}$ with $i_1<\dots<i_k$ and $k\in\{0,1,\dots\}$ is a dense $\ast$-subalgebra of $\IC\ell(H)$, which we denote by $\A$. 

Let $\F_-(H)$ be the fermionic Fock space over $H$, that is, $\F_-(H)=\bigoplus_{k\geq 0}\bigwedge^k H$. The map
\begin{align*}
\A\lra \F_-(H),\,\sum_{i_1<\dots<i_k} \alpha_{i_1\dots i_k} e_{i_1}\dots e_{i_k}\mapsto \sum_{i_1<\dots<i_k} \alpha_{i_1\dots i_k} e_{i_1}\wedge \dots\wedge e_{i_k}
\end{align*}
extends to an isometric isomorphism $\Phi\colon L^2(\IC\ell(H),\tau)\lra\F_-(H)$, the Chevalley--Segal isomorphism.

The number operator on $\F_-(H)$ is defined by
\begin{align*}
D(N)=\{(\psi_k)\in \F_-(H)\mid \sum_{k\geq 0}k^2\norm{\psi_k}_{\Lambda^k H}^2<\infty\},\,N(\psi_k)=(k\psi_k),
\end{align*}
and $\Phi^{-1} N\Phi$ generates a conservative quantum Dirichlet form $\E_N$ on $L^2(\IC\ell(H),\tau)$.

Let $a_i$ be the annihilation operator on $\F_-(H)$ characterized by
\begin{align*}
a_i(e_{j_1}\wedge\dots\wedge e_{j_k})=\frac 1{\sqrt k}\sum_{l=1}^k (-1)^l \langle e_i,e_{j_l}\rangle e_{j_1}\wedge\dots\wedge \widehat{e_{j_l}}\wedge\dots\wedge e_{j_k}
\end{align*}
and by $\gamma\colon L^\infty(\IC\ell(H),\tau)\lra L^\infty(\IC\ell(H),\tau)$ the grading operator.

The first order differential calculus for $\E_N$ is given by $\H=\sum_{i\geq 0}L^2(\IC\ell(H),\tau)$, $L(x)(\xi_i)=(x\xi_i)$, $R(x)(\xi_i)=(\gamma(x)\xi_i)$, $J(\xi_i)=-(\xi_i^\ast)$ and $\partial=\bigoplus_{i\geq 0}\Phi^{-1}a_i\Phi$.
\end{example}

\section{Carré du champ}\label{energy_dominance}

In this section we study the question of when the first order differential calculus introduced in the last section can be extended to $\M$. It turns out that this question is closely related to the so-called carré du champ operator defined below. More precisely we show in \ref{char_energy_dom} that the carré du champ $\Gam(a)$ has a density with respect to $\tau$ if and only if the left and right action of $D(\E)\cap \M$ have normal extensions to $\M$. This provides a characterization of the noncommutative analogue of energy dominant measures.

Throughout the section let $(\M,\tau)$ be a tracial von Neumann algebra, $\E$ a quantum Dirichlet form on $L^2(\M,\tau)$, $\C=D(\E)\cap \M$, and $(\partial,\H,L,R,J)$ the associated first order differential calculus.

The \emph{carré du champ} $\Gam$ of $\E$  is defined as
\begin{align*}
\Gam\colon \C\times\C\lra \C^\ast,\,\Gam(a,b)(x)=\langle x\partial a,\partial b\rangle_\H.
\end{align*}
We write $\Gam(a)$ for $\Gam(a,a)$. It is easy to see that $\Gam$ is sesquilinear and $\norm{\Gam(a,b)}_{\C^\ast}\leq \E(a)^{1/2}\E(b)^{1/2}$ for all $a,b\in\C$.

\begin{remark}
In terms of $\E$, the carré du champ can be expressed as 
\begin{align*}
\Gam(a)(x)=\frac 12(\E(a,ax^\ast)+\E(ax,a)-\E(a^\ast a,x^\ast))
\end{align*}
for all $a,x\in \C$.
\end{remark}

For $\xi=\sum_i a_i\partial b_i$ we define (compare \cite{HRT13} in the commutative case)
\begin{align*}
\Gam_\H(\xi)=\sum_{i,k} a_i\Gam(b_i,b_k)a_k^\ast.
\end{align*}
Then
\begin{align*}
\abs{\Gam_\H(\xi)(x)}&=\Abs{\sum_{i,k} \Gam(b_i,b_k)(a_k^\ast x a_i)}\\
&=\Abs{\sum_{i,k} \langle x a_i\partial b_i,a_k \partial b_k\rangle_\H}\\
&=\abs{\langle x\xi,\xi\rangle_\H}\\
&\leq \norm{x}_\M \norm{\xi}_\H^2
\end{align*}
for all $x\in \C$. Hence the map $\Gamma_\H\colon\mathrm{lin}\{a\partial b\mid a,b\in\C\}\lra \C^\ast$ is $\norm\cdot_\H$-$\norm\cdot_{\C^\ast}$ continuous. Since $L$ is non-degenerate, we can extend $\Gam_\H$ continuously to $\H$.

For the following two results, we use the $\sigma$-weak topology, which in our situation can be described as follows. Every von Neumann algebra $\M$ is isometrically isomorphic to the dual space of a Banach space, and if $\tau$ is an n.s.f. trace on $\M$, this isomorphism can be realized as
\begin{align*}
\M\lra L^1(\M,\tau)^\ast,\,x\mapsto\tau(x\,\cdot\,).
\end{align*}
The weak$^\ast$ topology under this identification is called \emph{$\sigma$-weak topology}. For $\M\subset B(H)$ the $\sigma$-weak topology can equivalently be characterized as the topology generated by the seminorms
\begin{align*}
p_{(\xi_n),(\eta_n)}\colon \M\lra [0,\infty],\,x\mapsto \sum_{n=1}^\infty\abs{\langle x\xi_n,\eta_n\rangle}
\end{align*}
for sequences $(\xi_n)$, $(\eta_n)$ in $H$ with $\sum_n (\norm{\xi_n}^2+\norm{\eta_n}^2)<\infty$. The space of all $\sigma$-weakly continuous linear functionals on $\M$ is denoted by $\M_\ast$. If $\tau$ is an n.s.f. trace on $\M$, then $L^1(\M,\tau)\cong \M_\ast$ via
\begin{align*}
L^1(\M,\tau)\lra \M_\ast,x\mapsto \tau(x\,\cdot\,).
\end{align*}

\begin{lemma}\label{C_ultraweakly_dense}
If $\E$ is a quantum Dirichlet form on $(\M,\tau)$, then $\C$ is $\sigma$-weakly dense in $\M$.
\end{lemma}
\begin{proof}
Let $(P_t)_{t\geq 0}$ be the quantum Markov semigroup associated with $\E$. If $a\in L^2(\M,\tau)\cap \M$, then $P_t(a)\in D(\E)\cap \M$ for all $t>0$ and $P_t (a)\to a$ $\sigma$-weakly as $t\searrow 0$. Now the assertions follows from the fact that $L^2(\M,\tau)\cap \M$ is $\sigma$-weakly dense in $\M$.
\end{proof}
\begin{remark}
Since $\C$ is a $\ast$-algebra, the Kaplansky density theorem (\cite[Theorem II.4.8]{Tak02}) asserts that $D(\E)\cap \M_1$ is even strongly dense in $\M_1$, where $\M_1$ denotes the unit ball in $\M$.
\end{remark}

\begin{theorem}[Characterization energy dominant trace]\label{char_energy_dom}
Let $\E$ be a quantum Di\-rich\-let form on the tracial von Neumann algebra $(\M,\tau)$. The following assertions are equivalent:
\begin{enumerate}[(i)]
\item $L$ is $\sigma$-weakly continuous
\item $R$ is $\sigma$-weakly continuous
\item $\Gam(a)$ is $\sigma$-weakly continuous for all $a\in \C$
\item $\Gam_\H(\xi)$ is $\sigma$-weakly continuous for all $\xi\in \H$
\end{enumerate}
\end{theorem}
\begin{proof}
(i)$\iff$(ii): Since multiplication by a fixed bounded operator and taking adjoints are $\sigma$-weakly continuous, the equivalence of (i) and (ii) follows from $L(\cdot)=JR(\cdot)^\ast J$. 

(iii)$\implies$(iv): It is easy to see that $\Gam_\H(\xi)$ is $\sigma$-weakly continuous for $\xi\in\mathrm{lin}\{a\partial b\mid a,b\in\C\}$. Combined with the fact that the norm limit of $\sigma$-weakly continuous functionals is $\sigma$-weakly continuous, (iv) follows.

(iv)$\implies$ (iii): obvious.

(i)$\implies$(iii): This is a consequence of the fact that $\sigma$-weak convergence implies weak operator convergence.

(iv)$\implies$(i): By Lemma \ref{C_ultraweakly_dense} and the subsequent remark, the set $D(\E)\cap\M_1$ is $\sigma$-weakly dense in $\M_1$. Moreover, since $\Gamma_\H(\xi)$ is linear and $\sigma$-weakly continuous, it is uniformly continuous with respect to the $\sigma$-weak topology (see \cite[Theorem 1.17]{Rud91}). Thus, by \cite[Theorem II.2]{Bou89}, for every $\xi\in \H$ there is a unique $\sigma$-weakly continuous extension of $\Gam_\H(\xi)$ to $\M$ with the same norm. We continue to write $\Gam_\H(\xi)$ for this extension.

For $L$ to be $\sigma$-weakly continuous it suffices to show that $\phi\circ L$ is $\sigma$-weakly continuous for all $\phi\in B(\H)_\ast$. Every $\phi\in B(\H)_\ast$ is of the form $\phi=\sum_n \langle\,\cdot\,\xi_n,\eta_n\rangle_\H$ for sequences $(\xi_n)$, $(\eta_n)$ in $\H$ such that $\sum_n (\norm{\xi_n}_\H^2+\norm{\eta_n}_\H^2)<\infty$. Then
\begin{align*}
\sum_{n=1}^\infty \norm{\Gam_\H(\xi_n,\eta_n)}_{\M^\ast}\leq\sum_{n=1}^\infty \norm{\xi_n}_\H\norm{\eta_n}_\H\leq \frac 1 2 \sum_{n=1}^\infty(\norm{\xi_n}_\H^2+\norm{\eta_n}_\H^2).
\end{align*}
Hence $\sum_n \Gam_H(\xi_n,\eta_n)$ converges absolutely with respect to $\norm{\cdot}_{\M^\ast}$ to some $\omega\in\M^\ast$. Since the space $\M_\ast$ of $\sigma$-weakly continuous linear functionals is closed in $\M^\ast$, we have $\omega\in \M_\ast$.

Now let $(a_i)$ be a sequence in $\M$ such that $a_i\to 0$ $\sigma$-weakly. Then
\begin{align*}
\sum_n \langle a_i\xi_n,\eta_n\rangle_\H=\sum_n \Gam_\H(\xi_n,\eta_n)(a_i)=\omega(a_i)\overset{i}\to 0.
\end{align*}
Hence $L$ is $\sigma$-weakly continuous.
\end{proof}

\begin{definition}[Energy dominant trace]
Let $\E$ be a quantum Dirichlet form on the tracial von Neumann algebra $(\M,\tau)$. We say that $\tau$ is \emph{energy dominant} if one of the equivalent assertions of Theorem \ref{char_energy_dom} holds.
\end{definition}
As already seen in the proof of Theorem \ref{char_energy_dom}, if the trace $\tau$ is energy dominant, the functional $\Gam_\H(\xi)$ has a unique $\sigma$-weakly continuous extension to $\M$ for all $\xi\in\H$ of the same norm. Since $L$ is non-degenerate, this extension is still a positive functional.

We denote by $\Gamma_\H(\xi)$ the preimage of $\Gam_\H(\xi)$ under the isomorphism $L^1(\M,\tau)\lra \M_\ast,\,x\mapsto \tau(x\,\cdot\,)$, that is, $\Gamma_\H(\xi)$ is the unique element in $L^1(\M,\tau)$ such that $\Gam_\H(\xi)(x)=\tau(x\Gamma_\H(\xi))$ for all $x\in\C$. Similarly we define $\Gamma(a)\in L^1(\M,\tau)$ for $a\in D(\E)$.

On the other hand, if $\tau$ is energy dominant, also the left and right action $L$ and $R$ have unique $\sigma$-weakly continuous extensions $\tilde L$ and $\tilde R$ to $\M$ and $\M^\circ$, respectively. These extensions are characterized by
\begin{align*}
\langle \tilde L(a)\xi,\eta\rangle_\H=\tau(a\Gamma_\H(\xi,\eta))
\end{align*}
for $a\in \M$, $\xi,\eta\in \H$, and similarly for $\tilde R$.

Since the vector space operations as well as the multiplication and the involution on $\M$ are all (separately) $\sigma$-weakly continuous, the extensions $\tilde L$ and $\tilde R$ are again $\ast$-homomorphisms. From now on we will denote these extensions simply by $L$, $R$.

\begin{remark}
If $\E$ is a Dirichlet form on $L^2(X,m)$, then $\Gam$ is twice the (linear functional induced by the) energy measure as defined in \cite[Section 3.2]{FOT94}. In this case, the measure $m$ is energy dominant if and only if $\Gam(u)$ is absolutely continuous with respect to $m$ for all $u\in D(\E)$. This concept was introduced by Kusuoka \cite{Kus89,Kus93} in the study of Dirichlet forms on fractals.
\end{remark}

\begin{remark}
In the noncommutative setting, energy dominant traces were studied for example in \cite{JZ15}, where the corresponding semigroups are called \emph{noncommutative diffusion semigroups}. We do not adopt this terminology as it conflicts with the well-established definition of diffusion semigroups in the commutative case.
\end{remark}

\begin{remark}
For an irreducible local Dirichlet form $\E$ it is always possible to construct an energy dominant measure $\mu$ such that $\E$ is closable in $L^2(X,\mu)$, see \cite[Theorem 5.1]{HRT13}.

In the noncommutative setting, it is not clear why an analogously constructed weight should be tracial.
\end{remark}

\section{\texorpdfstring{Operator means and the algebra $\A_\theta$}{Operator means and the algebra Aθ}}\label{A_E}

In this section we study means of the left and right action on $\H$, which will later appear both in the action functional and the constraint in the definition of the metric $\W$. As an important tool we introduce the space $\A_\theta$, which will take on the role of a space of test ``functions''. We prove several continuity properties of these means, which are important technical tools for the remainder of the thesis, especially the semicontinuity property established in Theorem \ref{theta_norm_usc}.

In the later sections we will focus on the logarithmic mean as it gives the connection to gradient flows of the entropy. In this section however we keep the discussion more general since means other than the logarithmic one have also proven useful in the commutative case (see for example \cite{CLLZ17} for an application to evolutionary games).

Throughout this section let $(\M,\tau)$ be a tracial von Neumann algebra and $\E$ a quantum Dirichlet form on $L^2(\M,\tau)$ with associated first-order differential calculus $(\partial,\H,L,R,J)$. We further assume that $\tau$ is energy dominant.

Since we assume $\tau$ to be energy dominant, the left and right action $L$ and $R$ extend to $\M$ by Theorem \ref{char_energy_dom}. Using the spectral theorem, we can even extend them to operators affiliated with $\M$ in the following way.

For self-adjoint $a\in \M$ let
\begin{align*}
a=\int_{\IR}\lambda\,de(\lambda)
\end{align*}
be the spectral decomposition. Since $L,R$ are normal $\ast$-homomorphisms, the maps $L\circ e$ and $R\circ e$ are spectral measures on $\H$ and
\begin{align*}
L(a)=\int_{\IR}\lambda\,d(L\circ e)(\lambda),
\end{align*}
and analogously for $R(a)$. This formula obviously extends to self-adjoint operators affiliated with $\M$. We continue to denote also these extensions by $L$ and $R$. For arbitrary $a$ affiliated with $\M$ with polar decomposition $a=u\abs{a}$ we define $L(a)=L(u)L(\abs{a})$ and $R(a)=R(\abs{a})R(u)$. Again this definition is clearly consistent for $a\in \M$, which justifies the use of the same symbol both for the maps on $\M$ and their extensions to operators affiliated with $\M$.

 It is easy to see that for self-adjoint $a,b$ affiliated with $\M$ the operators $L(a)$ and $R(b)$ commute strongly, that is, the spectral measures of $L(a)$ and $R(b)$ commute. Hence we can make sense of expressions of the form $\theta(L(\rho),R(\rho))$ via functional calculus (see \cite[Section 5.5]{Sch12}) for positive self-adjoint $\rho$ affiliated with $\M$.

\begin{definition}\label{def_rho_hat}
Let $\rho$ be a positive self-adjoint operator affiliated with $\M$ and let $\theta\colon [0,\infty)^2\lra [0,\infty)$ be measurable. Let $e$ denote the joint spectral measure of $L(\rho)$ and $R(\rho)$. The multiplication operator $\hat\rho$ is defined by
\begin{align*}
D(\hat\rho)&=\left\lbrace \xi\in \H\,\left|\, \int_{[0,\infty)^2}\theta(s,t)^2\,d\langle e(s,t)\xi,\xi\rangle_\H<\infty\right.\right\rbrace,\\
\langle\hat\rho \xi,\eta\rangle_\H&=\int_{[0,\infty)^2}\theta(s,t)\,d\langle e(s,t)\xi,\eta\rangle_\H.
\end{align*}
\end{definition}

\begin{remark}\label{rep_Beurling_Deny}
If $\M$ is commutative, one could alternatively define $\hat\rho$ separately for the strongly local and the jump part of $\E$ (recall that one can always regularize $\E$, even if at the cost of a huge state space). Indeed, in the light of the discussion in \cite[Section 10.1]{CS03} it is not hard to see that
\begin{align*}
\norm{\hat\rho^{1/2} u\partial v}_\H^2&=\int \theta(\rho(x),\rho(x))\abs{u(x)}^2\,d\Gamma^{(c)}(v)(x)\\
&\quad+\frac 1 2\int \theta(\rho(x),\rho(y))\abs{u(x)}^2 \abs{v(x)-v(y)}^2\,dJ(x,y).
\end{align*}
However, such a definition would be against the spirit of the present thesis to give a unified treatment of the local and non-local case. Moreover, there is no obvious way to extend this kind of definition to noncommutative Dirichlet forms.

Note that in the strongly local case, $\hat\rho$ only depends on the diagonal values of $\theta$.
\end{remark}

\begin{lemma}\label{theta_increasing}
Assume that $\theta\colon [0,\infty)^2\lra [0,\infty)$ is measurable and increasing in both arguments. For positive self-adjoint $\rho$ affiliated with $\M$ let $\rho_n=\rho\wedge n$. Then $\xi\in D(\hat \rho^{1/2})$ if and only if $\sup_n \langle\hat\rho_n \xi,\xi\rangle_\H<\infty$, and
in this case
\begin{align*}
\norm{\hat \rho^{1/2}\xi}_\H^2=\sup_{n\in \IN}\langle \hat\rho_n\xi,\xi\rangle_\H.
\end{align*}
\end{lemma}
\begin{proof}
Let $\xi\in\H$ and let $e$ be a joint spectral measure for $L(\rho)$ and $R(\rho)$. Then
\begin{align*}
\langle\hat\rho_n\xi,\xi\rangle_\H=\int_{[0,\infty)^2}\theta(s\wedge n,t\wedge n)\,d\langle e(s,t)\xi,\xi\rangle.
\end{align*}
By assumption, $\theta(s\wedge n,t\wedge n)\nearrow \theta(s,t)$ for all $s,t\geq 0$. The monotone convergence theorem gives
\begin{align*}
\int_{[0,\infty)^2}\theta(s,t)\,d\langle e(s,t)\xi,\xi\rangle=\sup_{n\in\IN}\langle\hat\rho_n\xi,\xi\rangle_\H.
\end{align*}
Thus $\xi \in D(\hat\rho^{1/2})=D(\theta(L(\rho),R(\rho))^{1/2})$ if and only if $\sup_n \langle\hat\rho_n\xi,\xi\rangle_\H<\infty$, and in this case $\norm{\hat\rho^{1/2}\xi}_\H^2=\sup_n\langle\hat\rho_n\xi,\xi\rangle_\H$.
\end{proof}

\begin{definition}\label{def_norm_rho}
For a positive self-adjoint operator $\rho$ affiliated with $\M$ and a measurable function $\theta\colon [0,\infty)^2\lra [0,\infty)$ we define
\begin{align*}
\norm\cdot_\rho\colon\H\lra [0,\infty],\,\norm{\xi}_\rho=\begin{cases}\norm{\hat\rho^{1/2}\xi}_\H&\text{if }\xi\in D(\hat\rho^{1/2}),\\\infty&\text{otherwise.}\end{cases}
\end{align*}
\end{definition}
In other words, $\norm\cdot_\rho^2$ is the quadratic form generated by $\rho$. Note that this definition implicitly depends on the choice of $\theta$. Lemma \ref{theta_increasing} shows that if $\theta$ is increasing in both arguments, this norm can alternatively be computed as $\norm{\xi}_\rho^2=\sup_n\langle\hat\rho_n\xi,\xi\rangle_\H$ for $\xi\in \H$.

\begin{lemma}\label{rho_norm_lsc}
Assume that $\theta\colon [0,\infty)^2\lra[0,\infty)$ is continuous, increasing in both arguments and $\theta(s,t)>0$ for $s,t>0$. If $\rho$ is an invertible positive self-adjoint operator affiliated with $\M$, then the map
\begin{align*}
L^2(\M,\tau)\lra[0,\infty],\,a\mapsto \begin{cases}\norm{\partial a}_\rho^2&\text{if } a\in D(\E)\\
\infty&\text{otherwise}\end{cases}
\end{align*}
is lower semicontinuous.
\end{lemma}
\begin{proof}
First assume that $\rho$ is bounded. Since $\rho$ is invertible and $\theta(s,t)>0$ for $s,t>0$, the operator $\hat \rho$ is also invertible. Thus $\hat\rho^{1/2}\partial$ is closed and the lower semicontinuity follows from a standard Hilbert space argument. If $\rho$ is not necessarily bounded, the lower semicontinuity follows from Lemma \ref{theta_increasing} and the first part.
\end{proof}

\begin{definition}\label{def_A_theta}
Let $\theta\colon[0,\infty)^2\lra [0,\infty)$ be measurable. For $a\in D(\E)$ let
\begin{align*}
\norm{a}_\theta^2=\sup_{\rho\in L^1_+(\M,\tau)}\frac{\norm{\partial a}_\rho^2}{\norm{\rho}_1}.
\end{align*}
The test space $\A_\theta$ is the set of all $a\in D(\E)\cap\M$ with $\norm{a}_\theta<\infty$.
\end{definition}

\begin{example}
If $\E$ is a strongly local commutative Dirichlet form on $L^2(X,m)$, then
\begin{align*}
\norm{\partial f}_\rho^2=\int_X \theta(\rho(x),\rho(x))\Gamma(f)\,dm.
\end{align*}
In particular, if $\theta(s,s)=s$, then $\norm{f}_\theta^2=\norm{\Gamma(f)}_\infty$ and
\begin{align*}
\A_\theta=\{f\in D(\E)\cap\M\mid \Gamma(f)\in L^\infty(X,m)\}.
\end{align*}
This is the space of test functions used in \cite{AES16}.
\end{example}

\begin{example}\label{ex_AM}
If $\theta=\mathrm{AM}$, the arithmetic mean, then
\begin{align*}
\norm{\partial a}_\rho^2=\frac 1 2\tau((\Gamma(a)+\Gamma(a^\ast))\rho)
\end{align*}
and
\begin{align*}
\A_{\mathrm{AM}}=\{a\in D(\E)\cap\M\mid \Gamma(a),\Gamma(a^\ast)\in \M\}.
\end{align*}
A variant of this algebra (without the assumption $\Gamma(a^\ast)\in \M$) was introduced in \cite[Definition 10.7]{Cip16} under the name \emph{Lipschitz algebra}. By \cite[Proposition 10.6]{Cip16} the boundedness of $\Gamma(a)$ is equivalent to the boundedness of the commutator $[D,a^\ast]$, where 
\begin{align*}
D=\begin{pmatrix}0&\partial^\ast\\\partial&0\end{pmatrix}
\end{align*}
is the Dirac operator acting on $L^2(\M,\tau)\oplus \H$. Hence the space $\A_{\mathrm{AM}}$ is closely related to spectral triples in Connes' noncommutative geometry \cite{Con94} (compare also Remark \ref{L1-Wasserstein}).
\end{example}

\begin{remark}
In general, it does not seem feasible to give a more explicit description of $\A_\theta$. Note however that if $\theta$ is concave, there exist $\alpha,\beta>0$ such that $\theta\leq \alpha \mathrm{AM}+\beta$ and thus $\A_{\mathrm{AM}}\subset \A_\theta$.
\end{remark}

\begin{lemma}\label{A_theta_symmetric}
If $\theta\colon [0,\infty)^2\lra[0,\infty)$ is a symmetric measurable function, then $\norm{\partial a^\ast}_\rho^2=\norm{\partial a}_\rho^2$ for all $a\in D(\E)\cap \M$ and positive self-adjoint operators $\rho$ affiliated with $\M$. In particular, $\A_\theta$ is self-adjoint.
\end{lemma}
\begin{proof}
It follows from the properties of the first-order differential calculus that $J\1_A(L(\rho))\1_B(R(\rho))=\1_B(L(\rho))\1_A(R(\rho))J$ for all Borel sets $A,B\subset [0,\infty)$. Thus, if $e$ denotes the joint spectral measure of $L(\rho)$ and $R(\rho)$, then
\begin{align*}
\norm{\partial a^\ast}_\rho^2&=\int_{[0,\infty)^2}\theta(s,t)\,d\langle e(s,t)J\partial a,J\partial a\rangle_\H\\
&=\int_{[0,\infty)^2}\theta(s,t)\,d\langle e(t,s)\partial a,\partial a\rangle\\
&=\norm{\partial a}_\rho^2,
\end{align*}
since $\theta$ is symmetric.
\end{proof}

\begin{lemma}
If $\tau$ is a state and $\theta\colon [0,\infty)^2\lra [0,\infty)$ is continuous, increasing in both arguments and $\theta(s,t)>0$ for $s,t>0$, then $\norm{\cdot}_\theta$ is lower semicontinuous on $L^2(\M,\tau)$.
\end{lemma}
\begin{proof}
By Lemma \ref{rho_norm_lsc} it suffices to show that the supremum in the definition of $\norm\cdot_\theta$ can be taken over all invertible $\rho\in L^1_+(\M,\tau)$.

For $\rho\in L^1_+(\M,\tau)$ and $\epsilon>0$ let $\rho^\epsilon=\rho+\epsilon$. Evidently, $\rho^\epsilon\in L^1_+(\M,\tau)$ is invertible and $\norm{\rho^\epsilon}_1=\norm{\rho}_1+\epsilon$. Since $\theta$ is increasing in both arguments, one sees as in the proof of Lemma \ref{theta_increasing} that
\begin{align*}
\norm{\partial a}_{\rho^\epsilon}^2= \norm{\partial a}_{\rho+\epsilon}^2\geq \norm{\partial a}_\rho^2.
\end{align*}
Thus 
\begin{equation*}
\sup_{\epsilon> 0}\frac{\norm{\partial a}_{\rho^\epsilon}^2}{\norm{\rho^\epsilon}_1}\geq\sup_{\epsilon>0}\frac{\norm{\partial a}_\rho^2}{\norm{\rho}_1+\epsilon}= \frac {\norm{\partial a}_\rho^2}{\norm{\rho}_1}.\qedhere
\end{equation*}
\end{proof}

\begin{corollary}
If $\tau$ is a state and $\theta$ is continuous, increasing in both arguments  and $\theta(s,t)>0$ for $s,t>0$, then $\A_\theta$ is complete in the norm $\norm{\cdot}_\M+\norm\cdot_\theta$.
\end{corollary}

The following technical lemma is certainly known to experts, but because we could not find a reference we include its proof for the convenience of the reader.

\begin{lemma}\label{Lp_conv_implies_src}
If $(x_n)$ is a sequence in $L^p_h(\M,\tau)$ and $x\in L^p_h(\M,\tau)$ such that $x_n\to x$ in $L^p$, then $(x_n)$ converges in the strong resolvent sense to $x$, that is, $(x_n+z)^{-1}\to (x+z)^{-1}$ strongly for every $z\in \IC\setminus \IR$.
\end{lemma}
\begin{proof}
Since $((x_n+z)^{-1})_n$ is bounded in $B(L^2(\M,\tau))$, it is  enough to prove that
\begin{align*}
\tau(a ((x_n+z)^{-1}-(x+z)^{-1})b)\to 0
\end{align*}
for $a,b\in L^2(\M,\tau)\cap L^{2q}(\M,\tau)$, where $q$ is the dual exponent of $p$.

Using the resolvent formula, we see
\begin{align*}
\abs{\tau(a((x_n+z)^{-1}-(x+z)^{-1})b)}&=\abs{\tau(a(x+z)^{-1}(x-x_n)(x_n+z)^{-1}b)}\\
&\leq \norm{x_n-x}_p\norm{a(x+z)^{-1}}_{2q}\norm{(x_n+i)^{-1}b}_{2q}\\
&\leq \frac 1{(\operatorname{Im} z)^2}\norm{x-x_n}_p\norm{a}_{2q}\norm{b}_{2q}\\
&\to 0.\qedhere
\end{align*}
\end{proof}

Next we study continuity properties of the map $\rho\mapsto \norm{\hat\rho^{1/2}\partial a}_\H^2$. We start with an auxiliary result for bounded $\theta$.

\begin{lemma}\label{theta_bounded_cont}
Denote by $\C_h(\H)$ the set of all self-adjoint operators on $\H$. If $\theta\colon [0,\infty)^2\lra [0,\infty)$ is continuous, then the map
\begin{align*}
L^1_+(\M,\tau)\lra \C_h(\H),\,\rho\mapsto \theta(L(\rho),R(\rho))
\end{align*}
is continuous with respect to the norm topology on $L^1$ and the strong resolvent topology on $\C_h(\H)$.

If $\theta$ is additionally bounded, then the map
\begin{align*}
L^1_+(\M,\tau)\lra [0,\infty),\,\rho\mapsto \norm{\xi}_\rho^2
\end{align*}
is continuous for all $\xi\in \H$.
\end{lemma}
\begin{proof}
Let $(\rho_n)$ be a sequence in $L^1_+(\M,\tau)$ and $\rho\in L^1_+(\M,\tau)$ such that $\rho_n\to \rho$ in the strong $L^1$ topology. By Lemma \ref{Lp_conv_implies_src} the sequence $(\rho_n)$ also converges to $\rho$ in the strong resolvent sense. Since $L$ and $R$ are normal $\ast$-homomorphisms, $L(\rho_n)\to L(\rho)$ and $R(\rho_n)\to R(\rho)$ in the strong resolvent sense as well.

Now $\theta(L(\rho_n),R(\rho_n))\to \theta (L(\rho),R(\rho))$ follows from \cite[Theorem VIII.20]{RS78}. As the strong resolvent topology coincides with the strong topology on norm bounded subsets of $B(\H)$, the last part is clear.
\end{proof}

\begin{theorem}\label{theta_norm_usc}
If $\theta\colon [0,\infty)^2\lra[0,\infty)$ is continuous, then
\begin{align*}
\Lambda\colon L^1_+(\M,\tau)\lra [0,\infty),\,\rho\mapsto \norm{\xi}_\rho^2
\end{align*}
is lower semicontinuous with respect to $\norm\cdot_1$ for all $\xi\in \H$. If $\theta$ is additionally concave, then $\Lambda$ is continuous for $\xi=\partial a$ with $a\in\A_\AM$.

In particular, if $\Lambda$ is concave, then it is weakly upper semicontinuous for $\xi=\partial a$ with $a\in\A_\AM$.
\end{theorem}
\begin{proof}
For $k\in\IN$ let $\theta_k=\theta\wedge k$. By Lemma \ref{theta_bounded_cont} the map
\begin{align*}
\Lambda_k\colon L^1_+(\M,\tau)\lra [0,\infty),\,\rho\mapsto \langle\theta_k(L(\rho),R(\rho))\xi,\xi\rangle_\H
\end{align*}
is continuous with respect to $\norm\cdot_1$ for all $\xi\in\H$. Since $\Lambda_k\nearrow \Lambda$ by functional calculus, the map $\Lambda$ is lower semicontinuous as supremum of continuous maps.

To prove the continuity when $\xi=\partial a$ with $a\in \A_\AM$, it only remains to show upper semicontinuity. Since $\theta$ is concave, there exist $\alpha,\beta\geq 0$ such that $\theta \leq \alpha\AM+\beta$. Since $a\in \A_\AM$, we have
\begin{align*}
\norm{(\alpha \AM(L(\rho),R(\rho))+\beta)^{1/2}\partial a}_\H^2=\frac\alpha 2\tau((\Gamma(a)+\Gamma(a^\ast))\rho)+\beta \E(a),
\end{align*}
which clearly depends continuously on $\rho$.

Moreover, the map
\begin{align*}
L^1_+(\M,\tau)\lra [0,\infty),\,\rho\mapsto \norm{(\alpha \AM+\beta-\theta)^{1/2}(L(\rho),R(\rho))\partial a}_\H^2
\end{align*}
is lower semicontinuous by the first part. Thus $-\Lambda$ is lower semicontinuous as the sum of two lower semicontinuous maps in this case.

Finally, the weak upper semicontinuity for concave $\Lambda$ follows from the Hahn-Banach theorem.
\end{proof}

\begin{remark}
In general, concavity of $\theta$ is not sufficient for concavity of $\Lambda$. However, in the following we will study a class of functions $\theta$ for which $\Lambda$ is concave.
\end{remark}

\begin{remark}
Note that for the upper semicontinuity part we need $a\in \A_\AM$ instead of $\A_\theta$. Whether upper semicontinuity still holds for $a$ in the bigger space $\A_\theta$ is unclear.
\end{remark}

In the first part of this section we saw that a crucial property of the multiplication operator is the concavity of the map $\rho\mapsto\hat\rho$. An important class of functions $\theta$ for which this property holds are those that can be represented as an operator mean in the sense of Kubo--Ando \cite{KA80}. We will now review the definition and a representation theorem for operator means before we turn to the application to the multiplication operator.

\begin{definition}\label{def_op_mean}
Let $H$ be an infinite-dimensional Hilbert space. An \emph{operator mean} is a map $\#\colon B(H)_+\times B(H)_+\lra B(H)_+$ such that
\begin{itemize}
\item $x_1\leq x_2$ and $y_1\leq y_2$ imply $x_1\# y_1\leq x_2\# y_2$,
\item $z(x\#y)z\leq (zxz)\#(zyz)$ for $x,y,z\in B(H)_+$,
\item $x_n\searrow x$ and $y_n\searrow y$ imply $x_n\#y_n\searrow x\#y$,
\item $1\#1=1$.
\end{itemize}
If $H$ is finite-dimensional, a map $\#\colon B(H)_+\times B(H)_+\lra B(H)_+$ is called operator mean if $H$ embeds into an infinite-dimensional Hilbert space $K$ such that $\#$ extends to an operator mean on $K$.

An operator mean $\#$ is called \emph{symmetric} if $x\#y=y\#x$ for all $x,y\in B(H)_+$.
\end{definition}

We say that a continuous function $\theta\colon [0,\infty)^2\lra [0,\infty)$ can be \emph{represented by a (symmetric) operator mean} if there exists a (symmetric) operator mean such that $\theta(x,y)=x\#y$ for all commuting $x,y\in B(\H)_+$.

\begin{example}\label{ex_operator_means}
Examples of symmetric operator means include
\begin{itemize}
\item the arithmetic operator mean $(x,y)\mapsto\frac 1  2(x+y)$,
\item the logarithmic operator mean, given by the generating function $f(t)=(t-1)/\log t$ (see Proposition \ref{prop_generating_fct}),
\item the harmonic operator mean $(x,y)\mapsto 2(x^{-1}+y^{-1})^{-1}$,
\item the geometric operator mean $(x,y)\mapsto x^{1/2}(x^{-1/2} y x^{-1/2})^{1/2} x^{1/2}$.
\end{itemize}
Two examples of non-symmetric operator means are 
\begin{itemize}
\item the left trivial mean $(x,y)\mapsto x$, and
\item the right trivial mean $(x,y)\mapsto y$.
\end{itemize}
\end{example}

There is a close relation between operator means and operator monotone functions. A continuous function $f\colon I\lra\IR$ is called \emph{operator monotone} if $x\leq y$ implies $f(x)\leq f(y)$ for all bounded self-adjoint operators $x,y$ with spectrum in $I$.

\begin{proposition}[{\cite[Theorem 3.2]{KA80}}]\label{prop_generating_fct}
Let $H$ be a Hilbert space. For every operator monotone function $f\colon (0,\infty)\lra (0,\infty)$ with $f(1)=1$ there exists a unique operator mean $\#$ such that
\begin{align*}
x\# y=x^{1/2}f(x^{-1/2} y x^{-1/2})x^{1/2}
\end{align*}
for all invertible $x,y\in B_+(H)$, and every operator mean arises this way.
\end{proposition}
In the situation of the proposition above, the operator monotone function $f$ is called the \emph{generating function} of $\#$. An important result of Löwner's seminal work on operator monotone functions (see \cite{Low34}) is that every operator monotone function admits an integral representation. A variant of this theorem reads as follows.

\begin{proposition}[{\cite[Theorem 4.9]{Han80}}]\label{Loewner}
A function $f\colon (0,\infty)\lra (0,\infty)$ is operator monotone if and only if there exists a finite Borel measure $\mu$ on $[0,1]$ such that
\begin{align*}
f(t)=\int_0^1\frac{t}{\lambda +(1-\lambda)t}\,d\mu(t)
\end{align*}
for $t>0$.
\end{proposition}

\begin{corollary}
A function $\theta\colon [0,\infty)^2\lra[0,\infty)$ can be represented by an operator mean if and only if there exists a Borel probability measure $\mu$ on $[0,1]$ such that
\begin{align*}
\theta(s,t)=\int_0^1\frac{st}{\lambda s+(1-\lambda)t}\,d\mu(\lambda)
\end{align*}
for $s,t>0$.

In this case $\theta$ is increasing in both arguments, positively homogeneous and satisfies $\theta(s,s)=s$ for $s\geq 0$. The resulting mean is symmetric if and only if $\mu(A)=\mu(1-A)$ for all Borel sets $A\subset [0,1]$.

Conversely, if $\#$ is an operator mean with generating function $f$ and one defines $\theta$ by $\theta(s,t)=s f(s/t)$ for $s,t>0$, then $x\# y=\theta(x,y)$ for all commuting $x,y\in B(\H)_+$.
\end{corollary}

After we introduced the multiplication operator $\hat\rho=\theta(L(\rho),R(\rho))$ and operator means, we will now discuss some additional properties of the multiplication operator in the case when $\theta$ can be represented by an operator mean.

\begin{proposition}\label{A_theta_algebra}
If $\theta$ can be represented by a symmetric operator mean, then $\A_\theta$ is a $\ast$-algebra.
\end{proposition}
\begin{proof}
We have already proven that $\A_\theta$ is self-adjoint in Lemma \ref{A_theta_symmetric}. It remains to show that $\A_\theta$ is an algebra.

Let $a,b\in \A_\theta$. By Lemma \ref{theta_increasing} it suffices to show that there exists a constant $C>0$ such that $\norm{\partial(ab)}_\rho^2\leq C\norm{\rho}_1$ for all $\rho\in L^1_+(\M,\tau)\cap \M$. We can assume without loss of generality $\norm{a}_\M,\norm{b}_\M\leq 1$.

By the product rule we have
\begin{align}
\begin{split}
\norm{\partial(ab)}_\rho^2&=\norm{\hat\rho^{1/2}(L(a)\partial b+R(b)\partial a)}_\H^2\\
&\leq 2\langle L(a^\ast)\hat\rho L(a)\partial b,\partial b\rangle_\H^2+2\langle R(b^\ast)\hat\rho R(b)\partial a,\partial a\rangle_\H.
\label{bound_rho_norm_prod}
\end{split}
\end{align}
Let $f$ be the generating function of the operator mean $\#$ that is represented by $\theta$. If $x\in \M_+$ is invertible, then
\begin{align*}
L(a^\ast)(L(x)\# R(x))L(a)&=R(x)^{1/2}L(a^\ast)f(R(x)^{-1/2}L(x)R(x)^{-1/2})L(a)R(x)^{1/2}\\
&\leq R(x)^{1/2}f(R(x)^{-1/2}L(a^\ast x a)R(x)^{-1/2})R(x)^{1/2}\\
&=L(a^\ast x a)\#R(x),
\end{align*}
where the inequality in the second line follows from the operator monotonicity of $f$ (see \cite{Han80}). If $x$ is not necessarily invertible, the same inequality still holds by the continuity property of $\#$.

Thus
\begin{align}
\begin{split}
L(a^\ast)\hat \rho L(a)&\leq L(a^\ast \rho a)\# R(\rho)\\
&\leq L(a^\ast \rho a)\# R(\rho)+ L(\rho)\# R(a^\ast \rho a)\\
&\leq L(a^\ast \rho a+\rho)\# R(\rho+ a^\ast \rho a),
\label{bound_left_mult}
\end{split}
\end{align}
where we used the concavity of operator means (\cite[Theorem 3.5]{KA80}) for the last inequality.

Since $\#$ is assumed to be symmetric, the inequality
\begin{align}
R(b^\ast)\hat \rho R(b)\leq L(b^\ast \rho b+\rho )\# R(\rho+b^\ast \rho b)\label{bound_right_mult}
\end{align}
follows analogously.

If we combine (\ref{bound_rho_norm_prod}), (\ref{bound_left_mult}) and (\ref{bound_right_mult}), we obtain
\begin{align*}
\norm{\partial(ab)}_\rho^2&\leq 2\norm{\partial a}_{\rho+b^\ast \rho b}^2+2\norm{\partial b}_{\rho+a^\ast \rho a}^2\\
&\leq 2\norm{a}_\theta^2\norm{\rho+b^\ast \rho b}_1+2\norm{b}_\theta^2\norm{\rho+a^\ast \rho a}_1\\
&\leq 4(\norm{a}_\theta^2+\norm{b}_\theta^2)\norm{\rho}_1.
\end{align*}
Hence $ab\in \A_\theta$.
\end{proof}

\begin{lemma}\label{theta_leq_AM}
If $\theta$ can be represented by a symmetric operator mean, then
\begin{align*}
\norm{\partial a}_\rho^2\leq \frac 1 2\tau((\Gamma(a)+\Gamma(a^\ast))\rho)
\end{align*}
for $a\in D(\E)$ and $\rho\in L^1_+(\M,\tau)$.
\end{lemma}
\begin{proof}
By \cite[Theorem 4.5]{KA80} we have $\theta\leq \mathrm{AM}$. Now it suffices to notice that
\begin{equation*}
\langle L(\rho)\partial a,\partial a\rangle_\H=\tau(\Gamma(a)\rho)
\end{equation*}
and
\begin{equation*}
\langle R(\rho)\partial a,\partial a\rangle_\H=\tau(\Gamma(a^\ast)\rho).\qedhere
\end{equation*}
\end{proof}

\begin{lemma}\label{norm_log_monotone}
Assume that $\theta$ can be represented by an operator mean. If $\rho_0,\rho_1\in L^1_+(\M,\tau)$ with $\rho_0\leq \rho_1$, then $\norm{\xi}_{\rho_0}\leq \norm{\xi}_{\rho_1}$ for all $\xi\in\H$.
\end{lemma}
\begin{proof}
If $\rho_0,\rho_1$ are bounded, then the claim is immediate from the definition of operator means. In the general case let $f_\epsilon(r)=r(1+\epsilon r)^{-1}$. This function is operator monotone, hence $f_\epsilon(\rho_0)\leq f_\epsilon(\rho_1)$. Moreover, $f_\epsilon(r)\nearrow r$ as $\epsilon\to 0$ implies $\norm{\xi}_{f_\epsilon(\rho_i)}\to \norm{\xi}_{\rho_i}$ as $\epsilon\to 0$ for $i\in\{0,1\}$. Combining this convergence with the monotonicity in the bounded case, we obtain
\begin{equation*}
\norm{\xi}_{\rho_0}=\lim_{\epsilon\to 0}\norm{\xi}_{f_\epsilon(\rho_0)}\leq \lim_{\epsilon\to 0}\norm{\xi}_{f_\epsilon(\rho_1)}=\norm{\xi}_{\rho_1}.\qedhere
\end{equation*}
\end{proof}

\begin{corollary}\label{monotone_conv_rho_norm}
Assume that $\theta$ can be represented by an operator mean. If $\rho_n\to \rho$ in $L^1_+(\M,\tau)$ and $\rho_n\leq \rho$, then $\norm{\xi}_{\rho_n}\to\norm{\xi}_\rho$ for $\xi\in \H$. If moreover $\xi\in D(\hat\rho^{1/2})$, then $\xi\in D(\hat\rho_n^{1/2})$ for all $n\in\IN$ and $\hat\rho_n^{1/2}\xi\to \hat\rho^{1/2}\xi$ in $\H$.
\end{corollary}
\begin{proof}
The first part is an immediate consequence of Theorem \ref{theta_norm_usc} and Lemma \ref{norm_log_monotone}. For the second part first note that $\widehat{\rho_n\wedge N}^{1/2}\to \widehat{\rho\wedge N}^{1/2}$ strongly as $n\to\infty$ by Lemma \ref{theta_bounded_cont}.

Let $e_n$ denote the joint spectral measure of $L(\rho_n)$ and $R(\rho_n)$. Then
\begin{align*}
\norm{(\hat\rho_n^{1/2}-\widehat{\rho_n\wedge N}^{1/2})\xi}_\H^2&=\int_{[0,\infty)^2}(\theta(s,t)^{1/2}-\theta(s\wedge N,t\wedge N)^{1/2})^2\,d\langle e_n(s,t)\xi,\xi\rangle_\H\\
&\leq \int_{[0,\infty)^2}(\theta(s,t)-\theta(s\wedge N,t\wedge N))\,d\langle e_n(s,t)\xi,\xi\rangle_\H\\
&=\norm{\xi}_{\rho_n}^2-\norm{\xi}_{\rho_n\wedge N}^2.
\end{align*}
The same holds for $\rho_n$ replaced by $\rho$. Thus
\begin{align*}
\norm{(\hat\rho_n^{1/2}-\hat\rho^{1/2})\xi}_\H&\leq \norm{(\hat\rho_n^{1/2}-\widehat{\rho_n\wedge N}^{1/2})\xi}_\H+\norm{(\widehat{\rho_n\wedge N}^{1/2}-\widehat{\rho\wedge N}^{1/2})\xi}_\H\\
&\quad + \norm{(\widehat{\rho\wedge N}^{1/2}-\hat\rho^{1/2})\xi}_\H\\
&\leq (\norm{\xi}_{\rho_n}^2-\norm{\xi}_{\rho_n\wedge N}^2)^{1/2}+\norm{(\widehat{\rho_n\wedge N}^{1/2}-\widehat{\rho\wedge N}^{1/2})\xi}_\H\\
&\quad+(\norm{\xi}_{\rho}^2-\norm{\xi}_{\rho\wedge N}^2)^{1/2}.
\end{align*}
Hence
\begin{align*}
\limsup_{n\to\infty}\norm{(\hat\rho_n^{1/2}-\hat\rho^{1/2})\xi}_\H\leq 2(\norm{\xi}_{\rho}^2-\norm{\xi}_{\rho\wedge N}^2)^{1/2},
\end{align*}
which goes to zero as $N\to\infty$.
\end{proof}

\begin{lemma}\label{logarithmic_mean_concave}
If $\theta$ can be represented by an operator mean, then
\begin{align*}
L^1_+(\M,\tau)\lra [0,\infty),\,\rho\mapsto \norm{\xi}_\rho^2
\end{align*}
is concave for all $\xi\in\H$.
\end{lemma}
\begin{proof}
Since operator means are jointly concave by \cite[Theorem 3.5]{KA80}, the map $\rho\mapsto\norm{\xi}_\rho^2$ is concave on $L^1_+(\M,\tau)\cap\M$. Hence, if $\rho_0,\rho_1\in L^1_+(\M,\tau)$ and $\lambda>0$, then
\begin{align*}
\norm{\xi}_{(1-\lambda)(\rho_0\wedge n)+\lambda(\rho_1\wedge n)}^2\geq (1-\lambda)\norm{\xi}_{\rho_0\wedge n}^2+\lambda\norm{\xi}_{\rho_1\wedge n}^2.
\end{align*}
By Lemma \ref{theta_increasing}, the right-hand side converges to $(1-\lambda)\norm{\xi}_{\rho_0}^2+\lambda\norm{\xi}_{\rho_1}^2$ as $n\to\infty$. On the other hand, Lemma \ref{norm_log_monotone} gives 
\begin{align*}
\norm{\xi}_{(1-\lambda)\rho_0+\lambda\rho_1}^2\geq \norm{\xi}_{(1-\lambda)\rho_0\wedge n+\lambda \rho_1\wedge n}^2.
\end{align*}
Thus $\rho\mapsto \norm{\xi}_\rho^2$ is concave.
\end{proof}

As mentioned before, we will later focus on the case when $\theta$ is the logarithmic mean
\begin{align*}
\mathrm{LM}\colon [0,\infty)^2\lra [0,\infty),\,(s,t)\mapsto \begin{cases}\frac{s-t}{\log s-\log t}&\text{if }s\neq t,\\s&\text{otherwise}.\end{cases}
\end{align*}
Alternatively, it can be represented as
\begin{align*}
\mathrm{LM}(s,t)=\int_0^1 s^\alpha t^{1-\alpha}\,d\alpha.
\end{align*}
A direct calculation shows that $\mathrm{LM}$ can be represented by a symmetric operator mean, namely the logarithmic operator mean from Example \ref{ex_operator_means}. Thus, all the results from this section are applicable in this case.

It is the following identity that sets the logarithmic mean apart from other possible choices of operator means in our context:
\begin{align*}
\mathrm{LM}(L(a),R(a))\partial \log(a)=\mathrm{LM}(L(a),R(a))\widetilde{\log}(L(a),R(a))\partial a=\partial a.
\end{align*}
This cancellation effect relies only on the chain rule for the first-order differential calculus. It would therefore be natural to consider more general functions $\theta$ of the form
\begin{align*}
\theta(s,t)=\frac{s-t}{\psi(s)-\psi(t)}.
\end{align*}
However, if we additionally require that $\theta$ can be represented by an operator mean, then it is not hard to see that $\psi$ is already forced to be the logarithm (up to an additive constant). Thus the choice of the logarithmic mean (and the von Neumann entropy later) is not arbitrary, but a consequence of these two simple structural assumptions.

\section{\texorpdfstring{The noncommutative transport metric $\W$}{The noncommutative transport metric W}}\label{Maas-Wasserstein}

In this section we define a transport metric on the space of density operators that generalizes both the discrete transport metric $\W$ from \cite{Maa11,Mie11,CHLZ12} and the Wasserstein metric $W_2$ on Riemannian manifolds.

The study of the optimal transport problem
\begin{align*}
&\int_{X\times X}d(x,y)^2\,d\pi(x,y)\to \min\\
&(\mathrm{pr}_1)_\#\pi=\mu,(\mathrm{pr}_2)_\#\pi=\nu
\end{align*}
defining the $L^2$-Wasserstein metric goes back to the work of Monge \cite{Mon81} and Kantorovich \cite{Kan42,Kan48}, who formulated the relaxed problem in the modern form. Especially for the quadratic case metric, the name ``Wasserstein metric'' is misleading, and some authors prefer to call it \emph{Monge-Kantorovich metric} or some variations of that. More information on the history of the Wasserstein metric as well as optimal transport in general can be found in the bibliographical notes in Villani's book \cite{Vil09}.

The Benamou--Brenier formula
\begin{align*}
W_2(\mu,\nu)^2=\inf\left\lbrace\int_0^1 \int_{\IR^n}\abs{v_t}^2\,d\mu_t\,dt\,\bigg|\, \dot\mu_t+\nabla\cdot(\mu_t v_t)=0,\mu_0=\mu,\mu_1=\nu\right\rbrace
\end{align*}
gives an equivalent definition of the Wasserstein metric on Borel probability measures over $\mathbb{R}^n$ as dynamical optimization problem. It was found by Benamou and Brenier \cite{BB00} in relation to numerical algorithms for the Wasserstein distance and later generalized to considerably more general settings (see for example \cite{AES16}).

Our definition of the transport metric $\W$ relies on a modification of the Benamou--Brenier formula. As already observed in the articles mentioned above in the case of finite graphs and matrix algebras, the crucial step is to not only replace the action functional in the classical Benamou--Brenier formula, but also the constraint by a suitable noncommutative version of the continuity equation.

While the form of this continuity equation is easily adapted from the previous work on the finite-dimensional case, finding a good weak formulation is still challenging. As it turns out, especially in view of the results in Section \ref{heat_gradient_flow}, the algebra $\A_\mathrm{AM}$ introduced in the last section is a good choice of test ``functions''. Among several other useful properties of the metric $\W$, we will use the continuity properties from the last section to prove lower semicontinuity of the energy functional defining $\W$ (Theorem \ref{energy_lsc}).

As usual, $(\M,\tau)$ is a tracial von Neumann algebra, $\E$ a quantum Dirichlet form on $L^2(\M,\tau)$ such that $\tau$ is energy dominant, and $(\partial,\H,L,R,J)$ the associated first-order differential calculus. We further assume that $\theta\colon [0,\infty)^2\lra [0,\infty)$ is a continuous function that can be represented by a symmetric operator mean. In particular, all results from Section \ref{A_E} are applicable. All expressions like $\hat \rho$, $\norm\cdot_\rho$ etc. are to be understood with respect to this particular choice of $\theta$.

\begin{definition}[Density operator]\label{def_density_op}
A \emph{density operator} is an element $\rho$ of $L^1_+(\M,\tau)$ with $\tau(\rho)=1$. The space of all density operators over $(\M,\tau)$ is denoted by $\D(\M,\tau)$.
\end{definition}

Under the map $\rho\mapsto\tau(\,\cdot\,\rho)$, the density matrices correspond exactly to the normal states on $\M$. Of course, if $\M$ is commutative, the density operators over $(\M,\tau)$ are just the classical probability densities.

\begin{definition}[Hilbert space $\H_\rho$]
For $\rho\in \D(\M,\tau)$ let $\tilde\H_\rho$ be the Hilbert space obtained from $D(\hat\rho^{1/2})$ after separation and completion with respect to $\norm\cdot_\rho$. Let $\H_\rho$ be the closure of $\partial \A_\AM$ in $\tilde\H_\rho$.

If $(\rho_t)_{t\in I}$ is a curve in $\D(\M,\tau)$, we say that a curve $(\xi_t)_{t\in I}$ with $\xi_t\in \H_{\rho_t}$ is measurable if $t\mapsto\norm{\xi_t}_{\rho_t}$ is measurable and $t\mapsto \langle \xi_t,\partial a\rangle_{\rho_t}$ is measurable for all $a\in\A_\AM$. The space of all a.e.-equivalence classes of measurable curves $(\xi_t)$ such that $\int_I \norm{\xi_t}_{\rho_t}^2\,dt<\infty$ is denoted by $L^2(I;(\H_{\rho_t})_{t\in I})$. The space $L^2_\loc(I;(\H_{\rho_t})_{t\in I})$ is defined accordingly.
\end{definition}

\begin{remark}
If there exists a countable subset $E$ of $\A$ such that $\partial E$ is dense in $\H_{\rho_t}$ for all $t\in I$, then $(\H_{\rho_t})_{t\in I}$ is a measurable field of Hilbert spaces in the sense of \cite[Definition 8.9]{Tak02} and $L^2(I;(\H_{\rho_t})_{t\in I})$ is just a different notation for the direct integral $\int_I^\oplus \H_{\rho_t}\,dt$.
\end{remark}

\begin{definition}[Admissible curves]\label{def_admissible_curve}
A curve $(\rho_t)_{t\in I}$ in $\D(\M,\tau)$ is \emph{admissible} if $t\mapsto \tau(\rho_t a)$ is locally absolutely continuous for all $a\in\A_\AM$ and there exists $\xi\in L^2_\loc(I;(\H_{\rho_t})_{t\in I})$ such that for all $a\in\A_\AM$ the continuity equation
\begin{align*}
\frac{d}{dt}\tau(a \rho_t)=\langle \partial a,\xi_t\rangle_{\rho_t}\tag{CE}\label{CE}
\end{align*}
holds for a.e. $t\in I$.

If it exists, such an element $\xi\in L^2_\loc(I;(\H_{\rho_t})_{t\in I})$ is necessarily unique since $\partial\A_\AM$ is dense in $\H_{\rho_t}$ for all $t\in I$, and we write $D\rho=\xi$ in this case.
\end{definition}

A couple of remarks are in order. First, the definition of absolutely continuous functions allows for an integral characterization of admissible curves that will be useful later on.

\begin{remark}
Let $(\rho_t)_{t\in I}$ be a curve in $\D(\M,\tau)$. It is easy to see that $(\rho_t)$ is admissible if and only if there exists a $c\in L^2_\loc(I)$ such that
\begin{align*}
\abs{\tau(a\rho_t)-\tau(a\rho_s)}\leq\int_s^t c(r)\norm{\partial a}_{\rho_r}\,dr
\end{align*}
for all $s,t\in I$ and $a\in \A_\AM$, and in this case, $r\mapsto \norm{D\rho_r}_{\rho_r}$ is the minimal function $c$ with this property.
\end{remark}

\begin{remark}
First rudiments of a solution theory of equations of similar type based on the noncommutative differential calculus have been developed in \cite{Zae16}.
\end{remark}

\begin{remark}
If $\E$ is the standard Dirichlet energy on a complete Riemannian mani\-fold $(M,g)$, then (\ref{CE}) reduces to the classical continuity equation
\begin{align*}
\dot\rho_t+\div (\rho_t\xi_t)=0
\end{align*}
(weakly in duality with the bounded Lipschitz functions).

Accordingly, if $(M,g)$ has lower bounded Ricci curvature, then a curve $(\rho_t)_{t\in I}$ of probability densities is admissible if and only if it is in $\AC^2_\loc(I;(P_2(M),W_2))$ by \cite[Proposition 2.5]{Erb10}. Compare also Example \ref{ex_2_Wasserstein} and Proposition \ref{prop_regular_sublevels}.
\end{remark}

Finally, let us also discuss two possible variants of the definition of admissible curves.

\begin{remark}
Instead of restricting to $\xi_t\in \H_{\rho_t}$ in (\ref{CE}), one might want to take $\xi_t\in \tilde \H_{\rho_t}$. This is no longer unique, but if it exists, the orthogonal projection $\eta_t$ of $\xi_t$ onto $\H_{\rho_t}$ still satisfies (\ref{CE}) and $\norm{\eta_t}_{\rho_t}\leq \norm{\xi_t}_{\rho_t}$. Instead of minimizing over all admissible curves $(\rho_t)$ with unique ``velocity vector field'' $(D\rho_t)$ in the definition $\W$ below, one can therefore equivalently minimize over all pairs of curves $(\rho_t,\xi_t)$ satisfying (\ref{CE}), where we only assume $\xi_t\in \tilde \H_{\rho_t}$.
\end{remark}

\begin{remark}\label{rmk_admissible_curves}
Since the definition of the multiplication operator $\hat \rho$ uses the mean $\theta$, it might appear more natural to replace $\A_\AM$ by the bigger space $\A_\theta$ both in the definition of $\H_\rho$ and the weak continuity equation (\ref{CE}). The crucial point is that the upper semicontinuity property from Theorem \ref{theta_norm_usc} is only guaranteed for $\A_\AM$.

However, under suitable conditions on the Dirichlet form $\E$ we introduce in Chapter \ref{subsec_GE}, the closure of $\partial \A_\theta$ in $\tilde \H_\rho$ coincides with $\H_\rho$ and the duality in the continuity equation can be extended to $a\in \A_\theta$ so that both of these possible definitions finally yield the same result.
\end{remark}

Under strong conditions on the curve $(\rho_t)$, the duality in (\ref{CE}) can be extended beyond to $D(\E)$.

\begin{lemma}\label{CE_bounded}
Assume that $\A_\AM\subset D(\E)$ is dense. If $(\rho_t)_{t\in I}$ is an admissible curve in $\D(\M,\tau)$ such that
\begin{equation*}
\sup_{J\subset I}\norm{\rho_t}_\M<\infty
\end{equation*}
for all compact $J\subset I$, then $t\mapsto \tau(\rho_t a)$ is locally absolutely continuous for all $a\in D(\E)$ and
\begin{align*}
\frac {d}{dt}\tau(a\rho_t)=\langle \partial a,D\rho_t\rangle_{\rho_t}
\end{align*}
for a.e. $t\in I$.
\end{lemma}
\begin{proof}
Let $(a_k)$ be a sequence in $\A_\AM$ such that $a_k\to a$ w.r.t. $\norm\cdot_\E$. Since $\rho_t\in \D(\M,\tau)\cap\M\subset L^2(\M,\tau)$, we have $\tau(a_k\rho_t)\to\tau(a\rho_t)$ as $k\to \infty$. On the other hand, since $\hat\rho_t$ is bounded and $\partial a_k\to \partial a$, we also have $\langle\partial a_k,D\rho_t\rangle_{\rho_t}\to \langle \partial a,D\rho_t\rangle_{\rho_t}$ as $k\to\infty$. Moreover,
\begin{align*}
\abs{\langle\partial a_k,D\rho_t\rangle_{\rho_t}}\leq \norm{\rho_t}_\M^{1/2}\E(a_k)^{\frac 1 2}\norm{D\rho_t}_{\rho_t}.
\end{align*}
Since $(\E(a_k))_k$ is bounded and $t\mapsto \norm{\rho_t}_\M$ is bounded on compact intervals, we can apply the dominated convergence theorem to get
\begin{align*}
\tau(a(\rho_t-\rho_s))=\lim_{k\to\infty}\tau(a_k(\rho_t-\rho_s))=\lim_{k\to\infty}\int_s^t \langle \partial a_k,D\rho_r\rangle_{\rho_r}\,dr=\int_s^t \langle \partial a,D\rho_r\rangle_{\rho_r}\,dr.
\end{align*}
From this equality, both the claimed absolute continuity and the identity for the derivative follow easily.
\end{proof}

We are now in the position to introduce the transport metric $\W$ as a length metric with a length functional defined on the class of admissible curves.

Strictly speaking, the map $\W$ will not be a metric since it might be degenerate and take the value infinity. Let us therefore recall the following extended concept of metrics.

\begin{definition}[Extended pseudometric]
Let $X$ be a set. An \emph{extended pseudometric} on $X$ is a map $d\colon X\times X\lra[0,\infty]$ such that
\begin{itemize}
\item $d(x,x)=0$ for $x\in X$,
\item $d(x,y)=d(y,x)$ for $x,y\in X$,
\item $d(x,y)\leq d(x,z)+d(z,y)$ for $x,y,z\in X$.
\end{itemize}
An extended pseudometric $d$ is an \emph{extended metric} if $d(x,y)=0$ implies $x=y$.
\end{definition}

\begin{definition}[Transport metric $\W$]\label{def_W}
The extended pseudometric $\W$ on $\D(\M,\tau)$ is defined by
\begin{align*}
&\W\colon\D(\M,\tau)\times\D(\M,\tau)\lra[0,\infty],\\
&\W(\bar\rho_0,\bar\rho_1)=\inf\left\{\left.\int_0^1\norm{D\rho_t}_{\rho_t}\,dt\,\right\rvert (\rho_t)\text{ admissible}, \rho_0=\bar\rho_0,\,\rho_1=\bar\rho_1\right\}.
\end{align*}
\end{definition}

\begin{remark}
If we endow $\D(\M,\tau)$ with the topology induced by the seminorms $\tau(a\,\cdot\,)$ for $a\in \A_\AM$, then the class of admissible curves together with the map that sends an admissible curve $(\rho_t)_{t\in I}$ to $\int_I\norm{D\rho_t}_{\rho_t}\,dt$ is a length structure in the sense of \cite[Chapter 2]{BBI01} and $\W$ is the associated length metric. The topological condition from their definition is verified in Proposition \ref{W_implies_weak_con}.
\end{remark}

\begin{remark}
Contrary to the Wasserstein metric, but also the metric $\W$ defined for certain jump processes in \cite{Erb14}, we define $\W$ only on densities. This is enough to study gradient flows of the entropy, which is only finite on measures with density anyway, but it would be interesting to see if there is an extension of $\W$ to a larger class of states.
\end{remark}

\begin{remark}
A different approach to noncommutative analogues of the Wasserstein distances, which relies on approximation by commutative subalgebras, has been studied in \cite{Zae15}. Contrary to our construction, if the algebra $\M$ is commutative, the metric $W_2$ defined by Zaev is the usual $L^2$-Was\-ser\-stein distance. In particular, in some examples it coincides and in some examples it is different from the metric constructed here. It is not clear if there is any deeper connection between these two approaches in the noncommutative case.
\end{remark}

\begin{example}
Let $(X,b,m)$ be a weighted graph and $\E^{(N)}$ as in Example \ref{ex_graphs}. Then
\begin{align*}
\norm{\xi}_{\rho}^2=\frac 1 2\sum_{x,y}b(x,y)\theta(\rho(x),\rho(y))\abs{\xi(x,y)}^2
\end{align*}
for $\rho\in \mathcal{P}(X,m)$ and $\xi\in \ell^2(X\times X,\frac 1 2 b)$. In particular, if $X$ is finite, this norm coincides with the one defined in \cite{Maa11}. Consequently, our metric $\W$ coincides with the metric $\W$ defined in \cite{Maa11} for finite graphs.
\end{example}

\begin{example}\label{ex_2_Wasserstein}
If $\E$ is the standard Dirichlet energy on $\IR^n$, then
\begin{align*}
\norm{\partial u}_\rho^2=\int_{\IR^n}\abs{\nabla u}^2\rho\,dx
\end{align*}
and the definition of $\W$ coincides with the Benamou--Brenier formulation \cite{BB00} of the $L^2$-Wasserstein distance.
\end{example}

\begin{example}\label{W_local}
More generally, let $\E$ be a strongly local regular Dirichlet form on $L^2(X,m)$ and assume that $m$ is energy dominant. Then
\begin{align*}
\norm{\partial u}_\rho^2=\int_X \Gamma(u) \rho\,dm.
\end{align*}
In this case $\W$ coincides with the metric $\W_\E$ defined in \cite[Definition 10.4]{AES16}. This in turn was shown in \cite[Theorem 12.5]{AES16} to coincide with the $L^2$-Was\-ser\-stein distance $\W_2$ if $(X,d,m)$ is an $\mathrm{RCD}(K,\infty)$ space and $\E$ is twice the Cheeger energy (see also \cite{AGS14b} for the relevant definitions).
\end{example}

Note that in the last two examples the transport metric $\W$ does not depend on the choice of the mean $\theta$. That is because if $\E$ is strongly local, only the values of $\theta$ on the diagonal matter, and these are already determined by the assumption that $\theta$ can be represented by an operator mean.

If one is only interested in the commutative case, one might want to relax the condition of operator concavity of $\theta$ to mere concavity. In this case, $\theta$ does not necessarily reduce to the identity on the diagonal. Metrics of this type (in the strongly local case) were studied in \cite{DNS09,CLSS10}.

\begin{example}
Let $\M$ be a finite-dimensional von Neumann algebra, $\tr$ the normalized trace on $\M$, and $(P_t)$ a quantum Markov semigroup on $\M$. Under the assumption that $(P_t)$ satisfies the quantum detailed balance condition, Carlen and Maas \cite{CM17a} defined a \emph{Riemannian} metric on the space $\D_+(\M,\tr)$ of  strictly positive density matrices. Let us shortly summarize their construction.

Given $\omega=(\omega_1,\dots,\omega_n)\in\IR^n$ and $c=(c_1,\dots,c_n)\in\IR^n$ (which are canonically associated with $(P_t)$) and a density matrix $\rho$, they define
\begin{align*}
[\rho]_{\omega_j}\colon L^2(\M,\tr)\lra L^2(\M,\tr),\,[\rho]_{\omega_j}=\int_0^1 (e^{-\omega_j/2}L_\rho)^s(e^{\omega_j/2}R_\rho)^{1-s}\,ds,
\end{align*} 
where $L_\rho$ and $R_\rho$ are the left and right multiplication with $\rho$ on $\M$. Further, $[\rho]_\omega=[\rho]_{\omega_1}\oplus\dots\oplus [\rho]_{\omega_n}$.

The norm of a tangent vector $\dot\rho_0$ is defined by
\begin{align*}
g(\dot\rho_0,\dot\rho_0)=\inf_V\sum_{j=1}^n c_j\langle V_j,[\rho_0]_{\omega_j}V_j\rangle_{L^2(\M,\tr)},
\end{align*}
where the infimum is taken over all $V$ satisfying a continuity equation of the form
\begin{align*}
\dot\rho_0=\div([\rho_0]_\omega V).
\end{align*}

In the case when $(P_t)$ is tracially symmetric and $\theta=\mathrm{LM}$, one has $c_j=1$, $\omega_j=0$ for all $j\in \{1,\dots,n\}$, and it is easily checked that $L$, $R$, $\div$ etc. coincide with the operations obtained from the first- order differential calculus. Therefore, $[\rho]_0=\hat\rho$ and the distance function induced by $g$ coincides with $\W$.

However, it should be stressed that the class of quantum Markov semigroups satisfying the detailed balance condition is larger than the class of tracially symmetric ones, so we do not fully recover the construction from \cite{CM17a}. It is an interesting open question how one can generalize the construction of the metric $\W$ to the case of infinite-dimensional quantum Markov semigroups satisfying the detailed balance condition for a non-tracial state or weight.
\end{example}

Next we collect some basic properties of $\W$. The first one is a sufficient condition to make $\W$ non-degenerate.

\begin{proposition}\label{W_implies_weak_con}
If $\rho_0,\rho_1\in \D(\M,\tau)$ and $a\in\A_\AM$, then
\begin{align*}
\abs{\tau(a(\rho_0-\rho_1))}^2\leq \norm{a}_\theta^2\W(\rho_0,\rho_1)^2\leq \norm{a}_{\AM}^2\W(\rho_0,\rho_1)^2.
\end{align*}

In particular, if $\A_\AM$ is $\sigma$-weakly dense in $\M$, then $\W$ is non-degenerate.
\end{proposition}
\begin{proof}
We can assume that $\W(\rho_0,\rho_1)<\infty$, otherwise there is nothing to prove. Let $(\rho_t)_{t\in [0,1]}$ be an admissible curve connecting $\rho_0$ and $\rho_1$. By definition $\norm{\partial a}_{\rho_t}^2\leq \norm{a}_\theta^2$ for all $t\in [0,1]$. Thus
\begin{align*}
\abs{\tau(a(\rho_1-\rho_0))}\leq \int_0^1\abs{\langle \partial a,D\rho_t\rangle_{\rho_t}}\,dt\leq\norm{a}_\theta\int_0^1 \norm{D\rho_t}\,dt.
\end{align*}
Taking the infimum over all admissible curves connecting $\rho_0$ and $\rho_1$ yields the first inequality.

The second inequality follows directly from the first an Lemma \ref{theta_leq_AM}. Finally, the last claim is an immediate consequence of the first inequality.
\end{proof}

\begin{remark}\label{L1-Wasserstein}
According to  \cite[Proposition 10.6]{Cip16}, the seminorm
\begin{align*}
\norm{\cdot}_\mathrm{AM}\colon \A_\mathrm{AM}\lra[0,\infty),\,a\mapsto \left(\frac 1 2(\norm{\Gamma(a)+\Gamma(a^\ast)}_\M)\right)^{1/2}
\end{align*}
is a Lipschitz seminorm in the spirit of \cite{Con89, Rie99}. The induced metric $W_\Gamma$ on $\D(\M,\tau)$ given by
\begin{align*}
W_\Gamma(\rho,\sigma)= \sup\{\abs{\tau(a(\rho-\sigma))}\colon a\in \A_\mathrm{AM},\,\norm{a}_\mathrm{AM}\leq 1\}
\end{align*}
is a noncommutative analog of the $L^1$-Wasserstein distance (depending on the context, it is also called Connes distance or spectral distance).
\end{remark}

The following lemma is standard.
\begin{lemma}\label{constant_speed}
If $(\rho_s)_{s\in[0,1]}$ is an admissible curve with $D\rho_s\neq 0$ for a.e. $s\in[0,1]$, then $(\rho_s)$ can be reparametrized so that the resulting curve $(\sigma_t)_{t\in I}$ has constant speed and 
\begin{align*}
\int_0^1\norm{D\sigma_t}_{\sigma_t}^2\,dt\leq \int_0^1\norm{D\rho_s}_{\rho_s}^2\,ds.
\end{align*}
\end{lemma}
\begin{proof}
By assumption, the map
\begin{align*}
[0,1]\lra[0,1],\,s\mapsto\frac{\int_0^s\norm{D\rho_r}_{\rho_r}\,dr}{\int_0^1 \norm{D\rho_r}_{\rho_r}\,dr}
\end{align*}
is continuous and strictly increasing, hence a homeomorphism. Denote its inverse by $\theta$ and let $\sigma_t=\rho_{\theta(t)}$. It is immediate from the definition that $\sigma$ is admissible and $D\sigma_t=\dot\theta(t)D\rho_{\theta(t)}$ for a.e. $t\in[0,1]$. Note that 
\begin{align*}
\dot\theta(t)=\frac{\int_0^1\norm{D\rho_r}_{\rho_r}\,dr}{\norm{D\rho_{\theta(t)}}_{\rho_{\theta(t)}}}.
\end{align*}
Thus $(\sigma_t)$ has constant speed and
\begin{equation*}
\int_0^1 \norm{D\sigma_t}_{\sigma_t}^2\,dt=\left(\int_0^1\norm{D\rho_r}_{\rho_r}\,dr\right)^2\leq \int_0^1\norm{D\rho_r}_{\rho_r}^2\,dr.\qedhere
\end{equation*}
\end{proof}

\begin{corollary}
The pseudometric $\W$ can alternatively be calculated as
\begin{align*}
\W(\bar\rho_0,\bar\rho_1)^2=\inf\left\lbrace\left.\int_0^1\norm{D\rho_t}_{\rho_t}^2\,dt\,\right|\, (\rho_t)\text{ admissible}, \rho_0=\bar\rho_0,\rho_1=\bar\rho_1\right\rbrace.
\end{align*}
\end{corollary}

\begin{lemma}[Convexity of the squared distance]\label{W_convex}
For $\rho^i_j\in\D(\M,\tau)$, $i,j\in\{0,1\}$, let $\rho^i_t=(1-t)\rho^i_0+t\rho^i_1$ for $i\in\{0,1\}$, $t\in[0,1]$. Then
\begin{align*}
\W^2(\rho^0_t,\rho^1_t)\leq (1-t)\W^2(\rho^0_0,\rho^1_0)+t \W^2(\rho^0_1,\rho^1_1)
\end{align*}
for all $t\in[0,1]$.
\end{lemma}
\begin{proof}
We can assume that $\W(\rho^0_0,\rho^1_0),\W(\rho^0_1,\rho^1_1)<\infty$. For $j\in\{0,1\}$ let $(\rho^s_j)_{s\in[0,1]}$ be admissible curves connecting $\rho^0_j$ and $\rho^1_j$ and let  $\xi^s_j=D_s\rho^s_j$. Define $\rho^s_t=(1-t)\rho^s_0+t\rho^s_1$ for $s,t\in[0,1]$. Obviously, $s\mapsto \tau(a\rho^s_t)$ is locally absolutely continuous for all $a\in \A_\AM$ and $t\in [0,1]$.

We will show that the map
\begin{align*}
\partial\A_\AM\lra \IC,\,\partial a\mapsto \frac{d}{ds}\tau(a \rho^s_t)
\end{align*}
is well-defined and continuous with respect to $\norm{\cdot}_{\rho^s_t}$: Indeed,
\begin{align*}
\frac{\Abs{\frac{d}{ds}\tau(a \rho^s_t)}^2}{\norm{\partial a}_{\rho^s_t}^2}&=\frac{\Abs{(1-t)\frac{d}{ds}\tau(a \rho^s_0)+t\frac{d}{ds}\tau(a\rho^s_1)}^2}{\norm{\partial a}_{\rho^s_t}^2}\\
&=\frac{\abs{(1-t)\langle\partial a,\xi^s_0\rangle_{\rho^s_0}+t\langle\partial a,\xi^s_1\rangle_{\rho^s_1}}^2}{\norm{\partial a}_{\rho^s_t}^2}\\
&\leq \frac{\abs{(1-t)\langle\partial a,\xi^s_0\rangle_{\rho^s_0}+t\langle\partial a,\xi^s_1\rangle_{\rho^s_1}}^2}{(1-t)\norm{\partial a}_{\rho^s_0}^2+t\norm{\partial a}_{\rho^s_1}^2}\\
&\leq (1-t)\frac{\abs{\langle\partial a,\xi_0^s\rangle_{\rho^s_0}}^2}{\norm{\partial a}_{\rho_0^s}^2}+t\frac{\abs{\langle\partial a,\xi_1^s\rangle_{\rho_1^s}}^2}{\norm{\partial a}_{\rho_1^s}^2}\\
&\leq (1-t)\norm{\xi^s_0}_{\rho^s_0}^2+t\norm{\xi^s_1}_{\rho^s_1}^2.
\end{align*}
For the first inequality we used Lemma \ref{logarithmic_mean_concave}, while the second inequality follows from the convexity of the function $(x,y)\mapsto \frac{y^2}{x}$.

Thus, $(\rho^s_t)_{s\in[0,1]}$ is admissible for every $t\in[0,1]$ and $(D_s\rho^s_t)_{s\in[0,1]}$ satisfies
\begin{align*}
\norm{D_s\rho^s_t}_{\rho^s_t}^2=\sup_{a\in\A_\AM}\frac{\abs{\langle \partial a,D_s\rho^s_t\rangle_{\rho^s_t}}^2}{\norm{\partial a}_{\rho^s_t}^2}\leq (1-t)\norm{D_s\rho^s_0}_{\rho^s_0}^2+t\norm{D_s\rho^s_1}_{\rho^s_1}^2.
\end{align*}
Therefore
\begin{align*}
\W^2(\rho^0_t,\rho^1_t)\leq (1-t)\int_0^1 \norm{D_s\rho^s_0}_{\rho^s_0}^2\,ds+t\int_0^1 \norm{D_s\rho^s_1}_{\rho^s_1}^2\,ds.
\end{align*}
Taking the infimum over all admissible curves $(\rho^s_j)_{s\in[0,1]}$ connecting $\rho^0_j$ and $\rho^1_j$ yields the assertion.
\end{proof}

\begin{definition}\label{def_AC}
Let $(X,d)$ be an extended metric space. A curve $\gamma\colon I\lra X$ is called \emph{$p$-locally absolutely continuous} if there exists a positive function $g\in L^p_\loc(I)$ such that
\begin{equation}
d(\gamma_s,\gamma_t)\leq \int_s^t g(r)\,dr\tag{AC$_p$}\label{AC}
\end{equation}
for all $s,t\in I$. We write $\AC_\loc^p(I;(X,d))$ for the space of all $p$-locally absolutely continuous curves in $(X,d)$. If $\gamma\in \AC_\loc^p(I;(X,d))$, then the \emph{metric speed}
\begin{align*}
\abs{\dot\gamma_t}_d:=\lim_{h\to 0}\frac{d(\gamma_{t+h},\gamma_t)}{\abs{h}}
\end{align*}
exists for a.e. $t\in I$ and $\abs{\dot\gamma}_d$ is the minimal $g\in L^p_\loc(I)$ such that (\ref{AC}) holds. 
\end{definition}

It is immediate from the definition that every admissible curve $(\rho_t)_{t\in I}$ belongs to $\AC^2_\loc(I;(\D(\M,\tau),\W))$ and $\abs{\dot\rho_t}_{\W}\leq \norm{D\rho_t}_{\rho_t}$ for a.e. $t\in I$.

\begin{corollary}[Convexity squared metric speed]\label{metric_speed_convex}
Let $(\rho^i_t)_{t\in I}$, $i\in\{0,1\}$, be locally absolutely continuous curves in $(\D(\M,\tau),\W)$ and $\rho^s=(1-s)\rho^0+s \rho^1$ for $s\in[0,1]$. Then $\rho^s$ is locally absolutely continuous and
\begin{align*}
\abs{\dot \rho^s}_{\W}^2\leq (1-s) \abs{\dot\rho^0}_{\W}^2+s\abs{\dot\rho^1}_{\W}^2
\end{align*}
for all $s\in[0,1]$.
\end{corollary}

At the present stage we cannot say much about when the distance $\W$ between two density matrices is finite. However, if $\E$ satisfies some functional inequalities, we get estimates on $\W$.
\begin{proposition}\label{W_bound_Poincare}
Assume that $\tau$ is a state. If $\E$ satisfies the Poincaré inequality with constant $c_P>0$, that is,
\begin{align*}
\norm{a-\tau(a)}_{2}^2\leq c_P^2\E(a)
\end{align*}
for all $a\in D(\E)$, then
\begin{align*}
\W(\rho_0,\rho_1)\leq \frac{c_P}{\lambda}\norm{\rho_1-\rho_0}_{2}
\end{align*}
for all $\rho_0,\rho_1\in \D(\M,\tau)\cap L^2(\M,\tau)$ with $\rho_0,\rho_1\geq \lambda^2>0$
\end{proposition}
\begin{proof}
Let $\rho_t=(1-t)\rho_0+t\rho_1$ and notice that $\rho_t\geq \lambda^2$ implies $\hat\rho_t\geq \lambda^2$. For all $a\in\A_\AM$ we have
\begin{align*}
\abs{\tau(a(\rho_t-\rho_s))}&=\abs{t-s}\abs{\tau((a-\tau(a))(\rho_1-\rho_0))}\\
&\leq c_P\abs{t-s}\norm{\rho_1-\rho_0}_{2}\norm{\partial a}_{\H}\\
&\leq \frac{c_P}{\lambda}\norm{\rho_1-\rho_0}_{2}\int_s^t \norm{\partial a}_{\rho_r}\,dr.
\end{align*}
Hence $(\rho_r)_{r\in[0,1]}$ is admissible with $\norm{D\rho_r}_{\rho_r}\leq \frac{c_P}{\lambda}\norm{\rho_1-\rho_0}_{2}$.
\end{proof}

\begin{remark}
Qualitatively, this result can be rephrased as follows: The form $\E$ satisfies a Poincaré inequality if and only if $\ker \partial$ is spanned by $1$ and $\partial^\ast$ has closed range. In this case, if $(\rho_t)$ is the linear interpolation between two density matrices in $L^2$, then
\begin{align*}
\dot\rho_t=\partial^\ast \eta_t
\end{align*}
has a solution $\eta_t\in \H$. If $\rho_t$ is additionally bounded away from zero, then there is a solution $\xi_t$ to
\begin{align*}
\hat\rho_t\xi_t=\eta_t,
\end{align*}
and $(\rho_t)$ satisfies the continuity equation for the vector field $(\xi_t)$.
\end{remark}

\begin{remark}
One typical problem for length spaces we have not touched upon yet is the existence of geodesics, that is, length-minimizing curves for a given start and endpoint. A length space is called geodesic if any two points with finite distance are joined by a geodesic. This property is one of the advantages of the metric $\W$ in the finite-dimensional case compared to the Wasserstein distance, with the geometry of the geodesics an object of recent attention (see \cite{GLM17,EMW19}).

Unfortunately, $(\D(\M,\tau),\W)$ can fail to be geodesic even in the commutative case, as was pointed out to the author by Erbar. However, we will see in Chapter \ref{heat_gradient_flow} that (under suitable conditions) the subset of all density matrices with finite entropy is indeed geodesic, and this is enough for the study of geodesic convexity of the entropy.
\end{remark}

Next we prove that the action functional appearing in the definition of $\W$ is lower semicontinuous with respect to pointwise weak convergence in $L^1$ and show some first consequences. This property will later be important for several approximation arguments.

\begin{theorem}[Lower semicontinuity of the action]\label{energy_lsc}
If $L^1(\M,\tau)$ is separable, then the action functional
\begin{align*}
E\colon \D(\M,\tau)^{[0,1]}\lra [0,\infty],\,(\rho_t)\mapsto \begin{cases}\int_0^1 \norm{D\rho_t}_{\rho_t}^2\,dt&\text{if } (\rho_t)\text{ is admissible},\\\infty&\text{otherwise}\end{cases}
\end{align*}
is lower semicontinuous with respect to pointwise weak convergence in $L^1(\M,\tau)$.
\end{theorem}
\begin{proof}
Let $(\rho^n)$ be a sequence in $\D(\M,\tau)^{[0,1]}$ and $\rho\colon [0,1]\lra \D(\M,\tau)$ such that $\rho_t^n\to\rho_t$ weakly in $L^1$ for all $t\in[0,1]$. Otherwise passing to subsequence, we may assume that $(E(\rho^n))_n$ is convergent. Moreover, if the limit is infinite, there is nothing to prove, so we assume additionally that $\sup_n E(\rho^n)<\infty$. In particular, the curve $(\rho^n_t)_t$ is admissible for all $n\in\IN$.

Fix $a\in\A_\AM$. Since $\rho^n_t\to \rho_t$ weakly for all $t\in [0,1]$, we have
\begin{align*}
\abs{\tau(a(\rho_t-\rho_s))}= \lim_{n\to\infty}\abs{\tau(a(\rho_t^n-\rho_s^n))}\leq \liminf_{n\to\infty}\int_s^t \norm{D \rho_r^n}_{\rho_r^n}\norm{\partial a}_{\rho_r^n}\,dr.
\end{align*}

Let $c^n\colon [0,1]\lra [0,\infty),\,c^n(r)=\norm{D\rho_r^n}_{\rho_r^n}$. By assumption, $(c^n)$ is bounded in $L^2([0,1])$, hence we may assume that $c^n\to c$ weakly in $L^2([0,1])$.

Note that the separability of $L^1(\M,\tau)$ implies that the weak $L^1$-topology restricted to $\D(\M,\tau)$ is metrizable (see \cite[Theorem V.5.1]{DS88}). Since $\rho\mapsto\norm{\partial a}_\rho$ is upper semicontinuous (Theorem \ref{theta_norm_usc}), there is a decreasing sequence $(G_k)$ of weakly continuous functions on $\D(\M,\tau)$ such that
\begin{align*}
\norm{\partial a}_{\rho}=\inf_{k\in\IN}G_k(\rho)
\end{align*}
for all $\rho\in \D(\M,\tau)$. Moreover, we can assume that $G_k\leq \norm{a}_\theta$ for all $k\in\IN$.

Let $g^n(r)=G_k(\rho_r^n)$ and $g(r)=G_k(\rho_r)$. The dominated convergence theorem gives $g^n\to g$ strongly in $L^2([0,1])$.

Hence
\begin{align*}
\abs{\tau(a(\rho_t-\rho_s))}\leq \lim_{n\to\infty}\int_s^t c^n(r) g^n(r)\,dr=\int_s^t c(r) g(r)\,dr=\int_s^t c(r) G_k(\rho_r)\,dr
\end{align*}
for all $k\in\IN$.

Finally, another application of the dominated convergence theorem yields
\begin{align*}
\abs{\tau(a(\rho_t-\rho_s))}\leq \lim_{k\to\infty}\int_s^t c(r)G_k(\rho_r)\,dr=\int_s^t c(r)\norm{\partial a}_{\rho_r}\,dr.
\end{align*}
This inequality implies that $(\rho_t)$ is admissible and $\int_0^1\norm{ D\rho_t}_{\rho_t}^2\,dt\leq \norm{c}_{L^2([0,1])}^2$. Thus
\begin{equation*}
\int_0^1 \norm{D \rho_t}_{\rho_t}^2\,dt\leq \norm{c}_{L^2([0,1])}^2\leq \liminf_{n\to\infty}\int_0^1 \norm{D \rho_t^n}_{\rho_t^n}^2\,dt.\qedhere
\end{equation*}
\end{proof}

\begin{remark}
The separability of $L^1(\M,\tau)$ is equivalent to each of the following properties:
\begin{itemize}
\item[(i)] The $\sigma$-weak topology on the unit ball of $\M$ is metrizable.
\item[(ii)] $\M$ has a faithful normal representation on a separable Hilbert space.
\end{itemize}
A von Neumann algebra with one of these properties is often called \emph{separable} or \emph{separably acting}.
\end{remark}

\begin{lemma}\label{admissible_measurable}
If $\A_\AM$ is $\sigma$-weakly dense in $\M$ and $L^1(\M,\tau)$ is separable, then every admissible curve is measurable in $L^1(\M,\tau)$.
\end{lemma}
\begin{proof}
Let $A$ be the uniform closure of $\A_\AM$. By Kaplansky's density theorem, $A\cap\M_1$ is $\sigma$-weakly dense in $\M_1$. Since $L^1(\M,\tau)$ is separable, the $\sigma$-weak topology is metrizable on $\M_1$. Thus, for every $a\in \M$ there exists a sequence $(a_k)$ in $A$ such that $a_k\to a$ $\sigma$-weakly.

If $(\rho_t)$ is an admissible curve, then $t\mapsto \tau(\rho_t a_k)$ is continuous for all $k\in\IN$. Therefore $t\mapsto \tau(\rho_t a)$ is measurable as pointwise limit of a sequence of continuous functions. Using once more the separability of $L^1(\M,\tau)$, we conclude that $(\rho_t)$ is measurable in $L^1(\M,\tau)$ due to Pettis' measurability theorem (see \cite[Theorem II.2]{DU77}).
\end{proof}

\begin{lemma}\label{W_diff_curves}
Assume that $\A_\AM$ is $\sigma$-weakly dense in $\M$ and $L^1(\M,\tau)$ is separable. If $(\rho_t)_{t\in[0,1]}$ is an admissible curve in $\D(\M,\tau)$, then there exists a family of admissible curves $(\rho_t^\epsilon)\in C^\infty([0,1];L^1(\M,\tau))$ such that $\rho_0^\epsilon=\rho_0$, $\rho_1^\epsilon=\rho_1$ for $\epsilon>0$, and
\begin{align*}
\limsup_{\epsilon\to 0}\int_0^1 \norm{D\rho_t^\epsilon}_{\rho_t^\epsilon}^2\,dt&\leq \int_0^1 \norm{D\rho_t}_{\rho_t}^2\,dt,\\
\limsup_{\epsilon\to 0}\esssup_{t\in[0,1]}\norm{D\rho_t^\epsilon}_{\rho_t^\epsilon}^2&\leq \esssup_{t\in[0,1]} \norm{D\rho_t}_{\rho_t}^2.
\end{align*}
In particular, the infimum in the definition of $\W$ can alternatively be taken over $L^1$-smooth admissible curves.
\end{lemma}
\begin{proof}
Let $(\rho_t)_{t\in[0,1]}$ be an admissible curve. Extend it to a curve $(\rho_t)_{t\in \IR}$ by setting $\rho_t=\rho_0$ for $t<0$ and $\rho_t=\rho_1$ for $t>1$. Note that the extended curve is still admissible with $D\rho_t=0$ for $t\in (-\infty,0)\cup(1,\infty)$.

Let $(\eta_\epsilon)_{\epsilon>0}$ be a mollifying kernel on $\IR$ with $\operatorname{supp}\eta_\epsilon\subset (-\epsilon,\epsilon)$ and set
\begin{align*}
\rho^\epsilon_t=\int_\IR \eta_\epsilon(s)\rho_{t-s}\,ds,
\end{align*}
where the integral is to be understood as Pettis integral in $L^1(\M,\tau)$. The measurability of $(\rho_t)$ is guaranteed by Lemma \ref{admissible_measurable}.

The curve $(\rho^\epsilon_t)_{t\in\IR}$ is in $C^\infty(\IR;L^1(\M,\tau))$ and satisfies $\rho^\epsilon_{t}=\rho_0$ for $t\leq-\epsilon$ and $\rho^\epsilon_{t}=\rho_1$ for $t\geq 1+\epsilon$. Moreover, if $a\in\A_{\AM}$, then
\begin{align*}
\abs{\tau(a(\rho^\epsilon_t-\rho^\epsilon_s))}&\leq\int_\IR\eta_\epsilon(r)\abs{\tau(a(\rho_{t-r}-\rho_{s-r}))}\,dr\\
&\leq\int_\IR \eta_\epsilon(r)\int_{s-r}^{t-r}\norm{\partial a}_{\rho_u}\norm{D\rho_u}_{\rho_u}\,du \,dr\\
&=\int_s^t \int_\IR \eta_\epsilon(r)\norm{\partial a}_{\rho_{u-r}}\norm{D\rho_{u-r}}_{\rho_{u-r}}\,dr\,du\\
&\leq\int_s^t \left(\int_\IR\eta_\epsilon(r)\norm{\partial a}_{\rho_{u-r}}^2\,dr\right)^{\frac 1 2}\left(\int_\IR\eta_\epsilon(r)\norm{D\rho_{u-r}}_{\rho_{u-r}}^2\,dr\right)^{\frac 1 2}\,du
\end{align*}
Since $\rho\mapsto\norm{\partial a}_\rho^2$ is upper semicontinuous and concave by Lemmas \ref{theta_norm_usc}, \ref{logarithmic_mean_concave}, we can apply the vector-valued version of Jensen's inequality (see \cite[Theorem 3.10]{Per74}) to get
\begin{align*}
\int_\IR\eta_\epsilon(r)\norm{\partial a}_{\rho_{u-r}}^2\,dr\leq \norm{\partial a}_{\rho_u^\epsilon}^2.
\end{align*}
Thus $(\rho^\epsilon_t)_{t\in\IR}$ is admissible and
\begin{align*}
\norm{D\rho^\epsilon_t}_{\rho^\epsilon_t}^2\leq \int_\IR \eta_\epsilon(r)\norm{D\rho_{t-r}}_{\rho_{t-r}}^2\,dr.
\end{align*}
for a.e. $t\in [0,1]$.

Finally one can reparametrize $(\rho_t^\epsilon)$ in such a way that $\rho_0^\epsilon=\rho_0$, $\rho^1_\epsilon=\rho_1$ and the claimed inequalities are retained.
\end{proof}

\begin{remark}
Mollifying in the time variable to restrict minimization problems to smooth curves is a standard argument, but the nonlinearity of $\norm{\xi}_\rho$ in $\rho$ requires some additional care in our case. In particular, the upper semicontinuity of $\rho\mapsto \norm{\partial a}_\rho^2$ is crucial here, because the vector-valued version of Jensen's inequality may fail otherwise (see \cite{Per74}).
\end{remark}

\section{Von Neumann entropy and Fisher information}\label{entropy_information}

In this section we introduce two important quantities for the gradient flow characterization of the heat flow, namely the (von Neumann) entropy and the Fisher information or entropy production, and carefully analyze convexity and continuity properties of these two as well as some related functionals.

The Fisher information appears in quantum information theory as the derivative of the entropy along heat flow curves, hence the name entropy production, but we will see that it also occurs in the metric speed of heat flow trajectories with respect to $\W$. This foreshadows already the close relation between entropy, heat flow and the metric $\W$, which we will exploit for the gradient flow characterization in the next section.

Throughout this section let $(\M,\tau)$ be a tracial von Neumann algebra, $\E$ a quantum Dirichlet form on $L^2(\M,\tau)$, $(\partial,\H,L,R,J)$ the associated first order differential calculus and assume that $\tau$ is energy dominant. We further assume that $\theta$ can be represented by a symmetric operator mean. Denote by $(P_t)_{t\geq 0}$ the quantum Markov semigroup associated with $\E$ and by $\L=\partial^\ast\partial$ its generator.

The \emph{von Neumann entropy} is defined as
\begin{align*}
\Ent\colon \D(\M,\tau)\lra [-\infty,\infty],\,\Ent(\rho)=\begin{cases}\tau(\rho\log \rho)&\text{if } (\rho\log \rho)_+\in L^1(\M,\tau),\\ \infty& \text{otherwise}.\end{cases}
\end{align*}
Its domain of definition is $D(\Ent)=\{\rho\in \D(\M,\tau)\mid \Ent(\rho)\in\IR\}$. Here and in the following, the expression $\rho\log \rho$ is to be understood as $f(\rho)$ for the (continuous) function
\begin{align*}
f\colon [0,\infty)\lra \IR,\,x\mapsto\begin{cases}x\log x&\text{if } x>0,\\0&\text{if } x=0.\end{cases}
\end{align*}

\begin{remark}
If $\tau(1)=1$, an application of Jensen's inequality shows that $\Ent\geq 0$. If $\tau(1)=\infty$, the entropy can be rather ill-behaved already in the commutative case (see e.g. \cite[Example 4.4]{Stu06a}). For that reason, we will from now on concentrate on the finite case. However, we believe that this assumption is not essential and that similar modifications as in \cite{AGS14a} should also work in our setting.
\end{remark}

\begin{remark}
Some authors, especially in the physics community, define the entropy with the opposite sign. We choose the sign in such a way that the entropy is positive if $\tau$ is a state and the semigroup $(P_t)$ is a gradient flow of $\Ent$ instead of $-\Ent$.
\end{remark}

An important property of the entropy is its lower semicontinuity with respect to suitable topologies on $\D(\M,\tau)$, in our case the topology induced by $\W$. To prove it, we first establish a variational formulation of the entropy. In the noncommutative case, it is (along with the idea of proof presented here) originally due to Petz \cite{Pet88}. We just adapt it to be applicable in duality with $\A_\mathrm{LM}$ instead of $\M$.

\begin{proposition}[Variational formula for the entropy]\label{Ent_variational}
Assume $\tau(1)=1$. If $A$ is a $\sigma$-weakly dense $C^\ast$-subalgebra of $\M$, then
\begin{align*}
\Ent(\rho)=\sup\{\tau(a\rho)-\log \tau(e^a)\mid a\in A_+\}
\end{align*}
for all $\rho\in\D(\M,\tau)$.
\end{proposition}
\begin{proof}
\emph{Step 1:} The equality
\begin{align*}
\Ent(\rho)=\max\{\tau(a\rho)-\log\tau(e^a)\mid a\in \M_+\}
\end{align*}
holds for invertible $\rho\in \D(\M,\tau)\cap \M$, and the maximum is attained at $a=\log M\rho$ for all sufficiently large $M$:

By assumption, there are constants $c,C>0$ such that $c\leq \rho\leq C$. Let $a=\log M\rho\in \M_+$ for $M\geq c^{-1}$. Then
\begin{align*}
\tau(a\rho)-\log\tau(e^a)=\tau(\rho\log \rho)+\log M-\log \tau(M\rho)=\tau(\rho\log \rho).
\end{align*}
For the converse inequality let $x=\log \rho$, $y=a-\log \tau(e^a)$. By Klein's inequality (Lemma \ref{Klein_ineq}) we have
\begin{align*}
0&\leq \tau(e^y -e^x-e^x(y-x))\\
&=\tau\left(\rho-\frac{e^a}{\tau(e^a)}-\rho(a-\log\tau(e^a)-\log \rho)\right)\\
&=\tau(\rho\log\rho)-(\tau(a\rho)-\log\tau(e^a)).
\end{align*}

\emph{Step 2: } The equality
\begin{align*}
\Ent(\rho)=\sup\{\tau(a \rho)-\log\tau(e^a)\mid a\in \M_+\}
\end{align*}
holds for all $\rho\in\D(\M,\tau)$:

Let $\rho^{(n)}=(\rho\wedge n)\vee\frac 1 n$, denote by $e$ the spectral measure of $\rho$ and let $\mu=\tau\circ e$. Then
\begin{align*}
\Ent(\rho^{(n)})&=\int_{[0,e^{-1})} \left(\lambda\vee\frac 1 n\right)\log\left(\lambda\vee\frac 1 n\right)\,d\mu(\lambda)+\int_{[e^{-1},\infty)}(\lambda\wedge n)\log(\lambda\wedge n)\,d\mu(\lambda)\\
&\to \int \lambda \log \lambda\,d\mu(\lambda),\,n\to\infty
\end{align*}
by monotone convergence.

Moreover,
\begin{align*}
\tau(\log\rho^{(n)}(\rho^{(n)}-\rho))=\int_{[0,\infty)}\log\lambda^{(n)}(\lambda^{(n)}-\lambda)\,d\mu(\lambda)\leq 0,
\end{align*}
as can be seen by a decomposition of $[0,\infty)$ into $[0,1/n)$, $[1/n,n]$ and $(n,\infty)$.

Let $a_n=\log n\rho^{(n)}$. Using the first step, we get
\begin{align*}
\Ent(\rho)=\lim_{n\to\infty}\Ent(\rho^{(n)})=\lim_{n\to\infty}(\tau(\rho^{(n)}a_n)-\log\tau(e^{a_n}))\leq \sup_{n\in\IN}(\tau(\rho a_n)-\log\tau(e^{a_n})).
\end{align*}
For the converse inequality, first observe that
\begin{align*}
\tau(\abs{\rho-\rho^{(n)}})=\int_{[0,1/n)}\left(\frac 1 n-\lambda\right)\,d\mu(\lambda)+\int_{(n,\infty)}(\lambda-n)\,d\mu(\lambda)\to 0,\,n\to\infty
\end{align*}
by dominated convergence.

Now let $a\in \M_+$. Then
\begin{align*}
\Ent(\rho)=\lim_{n\to\infty}\Ent(\rho^{(n)})\geq \lim_{n\to\infty}\tau(a\rho^{(n)})-\log\tau(e^a)=\tau(a\rho )-\log\tau(e^a).
\end{align*}
\emph{Step 3:} The equality
\begin{align*}
\Ent(\rho)=\sup\{\tau(a\rho)-\log\tau(e^a)\mid a\in A_+\}
\end{align*}
holds for all $\rho\in\D(\M,\tau)$:

Let $\epsilon>0$. By the second step, there is an $a\in \M_+$ such that 
\begin{align*}
\Ent(\rho)\leq \tau(a\rho)-\log \tau(e^a)+\frac{\epsilon} 3.
\end{align*}
Since $A\subset \M$ is $\sigma$-weakly dense, the unit ball of $A$ is strongly dense in the unit ball of $\M$ by Kaplansky's density theorem. Thus there is a net $(a_i)$ in $A_+$ with $\norm{a_i}_\M\leq \norm{a}_\M$ such that $a_i\to a$. By the continuity of the functional calculus, $(e^{a_i})$ converges strongly to $e^a$.

Let $i\in I$ such that $\tau(a\rho )\leq \tau(a_i \rho)+\epsilon/3$, $-\log\tau(e^a)\leq -\log\tau(e^{a_i})+\epsilon/3$. Then
\begin{align*}
\Ent(\rho)\leq \tau(a_i\rho )-\log \tau(e^{a_i})+\epsilon.
\end{align*}
Thus $\Ent(\rho)\leq \sup\{\tau(a \rho)-\log \tau(e^a)\mid a\in A_+\}$. The converse inequality is clear from Step 2.
\end{proof}

For the next corollary recall that a convex function is called \emph{proper} if it is not identically $\infty$.

\begin{corollary}\label{Ent_lsc_convex}
If $\tau(1)=1$ and $\A_\mathrm{LM}\subset \M$ is $\sigma$-weakly dense, then $\Ent$ is a proper lower semicontinuous convex functional on $(\D(\M,\tau),\W)$.
\end{corollary}
\begin{proof}
Denote by $A_\mathrm{LM}$ the uniform closure of $\A_\mathrm{LM}$. Let $(\rho_n)$ be a sequence in $\D(\M,\tau)$ such that $\W(\rho_n,\rho)\to 0$. It follows from Proposition \ref{W_implies_weak_con} that $\tau(a\rho_n)\to \tau(a \rho)$ for all $a\in\A_\theta$, and, since $(\rho_n)$ is bounded in $L^1(\M,\tau)$, indeed for all $a\in A_\theta$.

Combined with Proposition \ref{Ent_variational} we infer that $\Ent$ is the supremum of affine, continuous functions on $(\D(\M,\tau),\W)$, hence lower semicontinuous and convex. Since $\Ent(1)=0$, the entropy is also proper.
\end{proof}

\begin{remark}
As we already used in the proof of Theorem \ref{theta_norm_usc}, functionals of the form $\rho\mapsto \tau(f(\rho))$ are convex for all convex functions $f$.
\end{remark}

For approximating functionals we will also need the following (weaker) lower semicontinuity property, which holds for convex trace functionals.

\begin{proposition}[Lower semicontinuity of convex trace functionals]\label{entropy_lsc}
Assume that $\tau(1)<\infty$. For an interval $I\subset\IR$ let $L^1(\M,\tau)_I=\{a\in L^1(\M,\tau)\mid a=a^\ast,\,\sigma(a)\subset I\}$. If $f\colon I\lra \IR$ is a proper lower semicontinuous convex function,
then
\begin{align*}
F\colon L^1(\M,\tau)_I\lra (-\infty,\infty],\,a\mapsto\tau(f(a))
\end{align*}
is lower semicontinuous.
\end{proposition}
\begin{proof}
Since $f$ is proper, lower semicontinuous and convex, by \cite[Theorem 12.1]{Roc97} there is a sequence $(f_i)$ of affine functions such that $f=\sup_i f_i$. In particular, $F$ is well-defined.

Fix $n\in\IN$. For $j\in\{1,\dots,n\}$ let
\begin{align*}
E_j=\{t\in I\mid \max_{1\leq i<j}f_i(t)<f_j(t),\,\max_{j<i\leq n}f_i(t)\leq f_j(t)\}.
\end{align*}
In particular, $f_j(t)=\max_{1\leq i\leq n}f_i(t)$ for $t\in E_j$, and $I=\bigsqcup_j E_j$.

Let $(a_k)$ be a sequence in $L^1(\M,\tau)_I$ converging to $a\in L^1(\M,\tau)_I$. Then
\begin{align*}
\liminf_{k\to\infty}\tau(f(a_k))&=\liminf_{k\to\infty}\sum_{j=1}^n \tau(\1_{E_j}(a)f(a_k))\\
&\geq \liminf_{k\to\infty}\sum_{j=1}^n \tau(\1_{E_j}(a)f_j(a_k))\\
&=\sum_{j=1}^n\tau(\1_{E_j}(a)f_j(a))\\
&=\sum_{j=1}^n \tau\left(\1_{E_j}(a)\max_{1\leq i\leq n}f_i(a)\right)\\
&=\tau\left(\max_{1\leq i\leq n}f_i(a)\right).
\end{align*}
Since $f_1(a)$ is integrable, $\tau\left(\max_{1\leq i\leq n}f_i(a)\right)\nearrow \tau(f(a))$. This proves the claim.
\end{proof}

\begin{remark}
By the Hahn-Banach theorem, lower semicontinuity and weak lower semicontinuity are equivalent for convex functionals on convex subsets of Banach spaces.
\end{remark}

\begin{lemma}\label{trace_fct_diff}
Assume that $\tau$ is finite. Let $(\rho_t)_{t\in I}$ be an $L^1$-differentiable curve in $\D(\M,\tau)$. Assume that there exists $M>0$ such that $\norm{\rho_t}_\M\leq M$ for all $t\in I$. If $f\in C^1(\IR)$, then the map
\begin{align*}
F\colon I\lra \IR,\,t\mapsto\tau(f(\rho_t))
\end{align*}
is differentiable with derivative $F'(t)=\tau(f'(\rho_t)\dot \rho_t)$.
\end{lemma}
\begin{proof}
By approximation in $C^1$ we can assume that $f$ is a polynomial, and by linearity, even that $f$ is a monomial. Observe that
\begin{align*}
\tau(a^{k+1}-b^{k+1})=\tau\left(\left(\sum_{j=0}^k a^{k-j}b^j\right)(a-b)\right)
\end{align*}
for all $a,b\in \M$ and $k\geq 0$ (of course it is crucial here that $\tau$ is a trace).

Thus
\begin{align*}
\frac 1 h\tau(\rho_{t+h}^{k+1}-\rho_t^{k+1})=\sum_{j=0}^k\tau\left(\rho_{t+h}^{k-j}\rho_t^j\left(\frac{\rho_{t+h}-\rho_t}{h}\right)\right).
\end{align*}

Since $(\rho_t)$ is uniformly bounded and $L^1$-continuous, we have $\rho_{t+h}^j\to \rho_t^j$ as $h\to 0$ in $L^1$ and, by a standard approximation argument, also $\sigma$-weakly for all $j\geq 0$. Hence
\begin{align*}
\Abs{\tau\left(\rho_{t+h}^{k-j}\rho_t^j\left(\frac{\rho_{t+h}-\rho_t}{h}\right)-\rho_t^k \dot\rho_t\right)}&\leq \Abs{\tau\left(\rho_{t+h}^{k-j}\rho_t^j\left(\frac{\rho_{t+h}-\rho_t}{h}-\dot\rho_t\right)\right)}\\
&\quad+\abs{\tau((\rho_{t+h}^{k-j}-\rho_t^{k-j})\rho_t^j \dot\rho_t)}\\
&\leq M^k\tau\left(\Abs{\frac{\rho_{t+h}-\rho_t}{h}-\dot\rho_t}\right)\\
&\quad+\abs{\tau((\rho_{t+h}^{k-j}-\rho_t^{k-j})\rho_t^j \dot\rho_t)}\\
&\to 0,\,h\to 0.
\end{align*}
All put together, we have proven that
\begin{align*}
\frac 1 h(F(t+h)-F(t))=\frac 1 h\tau(\rho_{t+h}^{k+1}-\rho_t^{k+1})\to \tau((k+1)\rho_t^k \dot\rho_t)
\end{align*}
as $h\to 0$.
\end{proof}

\begin{lemma}[Klein's inequality]\label{Klein_ineq}
Assume that $\tau$ is finite. Let $f\colon \IR\lra\IR$ be a convex $C^1$ Lipschitz function. If $\rho_0,\rho_1\in L^1_h(\M,\tau)$, then $f(\rho_0),f(\rho_1)\in L^1(\M,\tau)$ and
\begin{align*}
\tau(f(\rho_1)-f(\rho_0))\leq \tau(f'(\rho_1)(\rho_1-\rho_0)).
\end{align*}
\end{lemma}
\begin{proof}
If $\rho_0,\rho_1\in L^1_h(\M,\tau)\cap\M$, then one can use Lemma \ref{trace_fct_diff} to adapt the proof of the finite-dimensional case (see \cite[Theorem 2.11]{Car10}).

Now let $\rho_0,\rho_1\in L^1_h(\M,\tau)$ be arbitrary. We already know that the inequality holds for $\rho_0,\rho_1$ replaced by $\rho_0\wedge n$, $\rho_1\wedge n$. It remains to show that both sides converge to the correct limit as $n\to\infty$.

Since $f'$ is bounded and increasing, $\norm{f'(\rho_1\wedge n)}_\M\leq\norm{f'}_\infty$ for all $n\in\IN$ and $f'(\rho_1\wedge n)\to f'(\rho_1)$ $\sigma$-weakly as $n\to\infty$. It is elementary that $\rho_0\wedge n\to \rho_0$ and $\rho_1\wedge n\to \rho_1$ in $L^1$. Put together we get
\begin{align*}
\tau(f'(\rho_1\wedge n)(\rho_1\wedge n-\rho_0\wedge n))\to\tau(f'(\rho_1)(\rho_1-\rho_0)).
\end{align*}
For the convergence of the left-hand side one can use that $\IR$ can be decomposed into at most two intervals such that $f$ is decreasing on the first one and decreasing on the second. Using monotone convergence on both parts one gets $\tau(f(\rho_1\wedge n))\to\tau(f(\rho_1))$ and the same for $\rho_0$.
\end{proof}

\begin{corollary}\label{diff_trace_functional}
Assume that $\tau$ is finite. If $f\colon \IR\lra\IR$ is a convex $C^1$ Lipschitz function and $(\rho_t)_{t\in I}$ is an $L^1$-differentiable curve in $L^1_h(\M,\tau)$, then $t\mapsto\tau(f(\rho_t))$ is locally absolutely continuous and
\begin{align*}
\frac{d}{dt}\tau(f(\rho_t))=\tau(f'(\rho_t)\dot\rho_t)
\end{align*}
for a.e. $t\in I$.
\end{corollary}
\begin{proof}
By Klein's inequality we have
\begin{align*}
\tau(f'(\rho_s)(\rho_t-\rho_s))\leq\tau(f(\rho_t)-f(\rho_s))\leq \tau(f'(\rho_t)(\rho_t-\rho_s))
\end{align*}
for all $s,t\in I$. Thus
\begin{align*}
\abs{\tau(f(\rho_t)-f(\rho_s))}\leq \norm{f'}_\infty\norm{\rho_t-\rho_s}_{1}.
\end{align*}
Since $(\rho_t)$ is $L^1$-differentiable, it follows that $t\mapsto \tau(f(\rho_t))$ is locally absolutely continuous.

Now assume that $t\in I$ is a point of differentiability. By Klein's inequality,
\begin{align*}
\frac {d}{dt}\tau(f(\rho_t))&=\lim_{s\nearrow t}\frac 1{t-s}\tau(f(\rho_t)-f(\rho_s))\leq\lim_{s\nearrow t}\tau\left(f'(\rho_t)\frac{\rho_t-\rho_s}{t-s}\right)=\tau(f'(\rho_t)\dot\rho_t),
\end{align*}
and
\begin{align*}
\frac{d}{dt}\tau(f(\rho_t))=\lim_{s\searrow t}\frac 1{t-s}\tau(f(\rho_t)-f(\rho_s))\geq \lim_{s\nearrow t}\tau\left(f'(\rho_t)\frac{\rho_t-\rho_s}{t-s}\right)=\tau(f'(\rho_t)\dot\rho_t).
\end{align*}
This settles the claim.
\end{proof}

\begin{lemma}\label{increasing_Lip}
If $a\in D(\E)_h$ and $C_1,C_2\colon \IR\lra\IR$ are increasing Lipschitz functions with $C_1(0)=C_2(0)=0$, then
\begin{align*}
\E(C_1(a),C_2(a))\geq 0.
\end{align*}
\end{lemma}
\begin{proof}
Since $C_1(s)-C_1(t)$ and $C_2(s)-C_2(t)$ have the same sign,
\begin{align*}
\abs{C_1(s)-C_2(s)-C_1(t)+C_2(t)}\leq \abs{C_1(s)+C_1(t)-C_2(s)-C_2(t)}
\end{align*}
for all $s,t\in\IR$. Hence there exists a $1$-Lipschitz function $C\colon \IR\lra \IR$ with $C(0)=0$ such that $C\circ(C_1+C_2)=C_1-C_2$. Thus
\begin{equation}
\E(C_1(a),C_2(a))=\frac 1 4(\E(C_1(a)+C_2(a))-\E(C_1(a)-C_2(a)))\geq 0.\qedhere
\end{equation}
\end{proof}

\begin{remark}
If $C_1,C_2$ in the previous lemma are continuously differentiable, we can also use the chain rule to get
\begin{align*}
\E(C_1(a),C_2(a))=\langle (\tilde C_1\tilde C_2)(L(a),R(a))\partial a,\partial a\rangle_\H\geq 0.
\end{align*}
\end{remark}

\begin{remark}
The short proof of Lemma \ref{increasing_Lip} was shown to the author by S. Puchert, replacing a quite involved proof arguing by approximation of $\E$.
\end{remark}

\begin{corollary}\label{increasing_contractions}
Let $C,C_1,C_2\colon \IR\lra\IR$ be increasing Lipschitz functions with $C(0)=C_1(0)=C_2(0)=0$ and $\abs{C_1(s)-C_1(t)}\leq \abs{C_2(s)-C_2(t)}$ for all $s,t\in\IR$. Then
\begin{align*}
\E(C(a),C_1(a))\leq \E(C(a),C_2(a))
\end{align*}
for all $a\in D(\E)_h$.
\end{corollary}

\begin{lemma}\label{contractions_L_p_generator}
Let $\E$ be a quantum Dirichlet form on $L^2(\M,\tau)$, denote by $\L^{(p)}$ the generator of the associated semigroup on $L^p(\M,\tau)$ for $p\in [1,\infty)$, and let $a\in D(\L^{(p)})$.

If $C_1,C_2\colon \IR\lra\IR$ are increasing $1$-Lipschitz functions and there exists a constant $\alpha>0$ such that $\abs{C_i(t)}\leq \alpha \abs{t}^{p-1}$ for $t\in \IR$, $i\in\{1,2\}$, then $C_1(a),C_2(a)\in D(\E)$ and
\begin{align*}
\E(C_1(a),C_2(a))\leq \tau(C_1(a)\L^{(p)} a).
\end{align*}
\end{lemma}
\begin{proof}
First note that $C_i(a)\in L^p(\M,\tau)\cap L^q(\M,\tau)\subset L^2(\M,\tau)$, where $q$ is the dual exponent of $p$. To prove $C_i(a)\in D(\E)$, it suffices to show
\begin{align}
\frac 1 t\tau(C_i(a)(C_i(a)-P_t C_i(a))\leq \frac 1t\tau(C_i(a)(a-P_t a))\label{ineq_approx_forms}
\end{align}
for $t>0$, since the right-hand side converges to $\tau(C_i(a)\L^{(p)}a)$ as $t\to 0$.

Since the approximating form
\begin{align*}
L^2(\M,\tau)\lra [0,\infty), x\mapsto \frac 1 t \tau(x(x-P_t x))
\end{align*}
is a quantum Dirichlet form, Equation (\ref{ineq_approx_forms}) holds for $a\in L^2(\M,\tau)$ by Corollary \ref{increasing_contractions}.

In the general case let $(a_k)$ be a sequence in $L^2(\M,\tau)\cap L^p(\M,\tau)$. By \cite[Theorem 3.2]{Tik87} we have $C_i(a_k)\to C_i(a)$ with respect to $\norm\cdot_p$ and $\norm\cdot_q$ (resp. $\sigma$-weakly in the case $p=1$). Using the continuity of $P_t$ with respect to $\norm\cdot_p$ and $\norm\cdot_q$ (and additionally the bound $\norm{C_i(a_k)}_\M\leq \norm{C_i}_\infty$ in the case $p=1$), we see that (\ref{ineq_approx_forms}) continues to hold for arbitrary $a\in L^p(\M,\tau)$.

Finally, by Corollary \ref{increasing_contractions} and the same approximation argument as above, we obtain
\begin{align*}
\frac 1 t \tau(C_1(a)(C_2(a)-P_t C_2(a)))\leq \frac 1t \tau(C_1(a)(a-P_t a)),
\end{align*}
which gives
\begin{align*}
\E(C_1(a),C_2(a))\leq \tau(C_1(a)\L^{(p)}a)
\end{align*}
in the limit $t\to 0$.
\end{proof}

\begin{lemma}
Let $(C_n)$ be a sequence of continuously differentiable, increasing normal contractions on $[0,\infty)$ with
\begin{align*}
\tilde C_n(s,t)\nearrow \widetilde{\log}(s,t)
\end{align*}
for all $s,t\geq 0$ (with the convention that the right-hand side equals $\infty$ whenever $s=0$ or $t=0$).

Then $\lim_{n\to\infty}\E(a,C_n(a))$ exists in $[0,\infty]$ for all $a\in D(\E)_+$ and is independent from the choice of the sequence $(C_n)$.
\end{lemma}
\begin{proof}
We use the chain rule to get
\begin{align*}
\E(a,C_n(a))=\langle \tilde C_n(L(a),R(a))\partial a,\partial a\rangle_\H.
\end{align*}
Denote by $e$ the joint spectral measure of $L(a)$ and $R(a)$. Then
\begin{align*}
\langle \tilde C_n(L(a),R(a))\partial a,\partial a\rangle_\H&=\int_{[0,\infty)^2} \tilde C_n(s,t)\,d\langle e(s,t)\partial a,\partial a\rangle_\H\\
&\nearrow \int_{[0,\infty)^2}\widetilde{\log}(s,t)\,d\langle e(s,t)\partial a,\partial a\rangle_\H,
\end{align*}
where the integrand is interpreted as $\infty$ whenever $s=0$ or $t=0$.
\end{proof}

\begin{definition}[Fisher information]
The \emph{Fisher information} of $a\in D(\E)_+$ is defined as
\begin{align*}
\I(a)=\lim_{n\to\infty}\E(a,C_n(a))\in[0,\infty]
\end{align*}
for some (any) sequence $(C_n)$ of continuously differentiable, increasing normal contractions with $\tilde C_n\nearrow \widetilde{\log}$ pointwise.
\end{definition}

An example of a sequence $(C_n)$ that is admissible in the definition of the Fisher information is given by
\begin{align*}
C_n\colon [0,\infty)\lra[0,\infty),\,C_n(t)=\log(t+e^{-n})+n.
\end{align*}

We will also need the following two different approximation results of the Fisher information:
\begin{lemma}
If $a\in D(\E)_+$ with $\I(a)<\infty$, then 
\begin{align*}
\I(a)=\lim_{n\to\infty}\E(a,((\log\wedge n)\vee(-n)+n)(a)).
\end{align*}
\end{lemma}
\begin{proof}
Let $f_n=(\log\wedge n)\vee(-n)+n$. By Corollary \ref{increasing_contractions}, the sequence $(\E(a,f_n(a)))_n$ is increasing, hence it suffices to show convergence along a subsequence. Let 
\begin{align*}
C_k\colon [0,\infty)\lra[0,\infty),\,C_k(t)=\log(t+e^{-k})+k
\end{align*}
and $n_k\in\IN$ such that $\E((C_k(a)\wedge n_k)\vee(-n_k))\geq \E(C_k(a))-\frac 1 k$. Then
\begin{align*}
\E(a,(C_k(a)\wedge n_k)\vee(-n_k))\to \I(a),\,k\to\infty,
\end{align*}
and
\begin{align*}
\abs{(C_k(s)\wedge n_k)\vee(-n_k)-(C_k(t)\wedge n_k)\vee(-n_k)}\leq \abs{f_{n_k}(s)-f_{n_k}(t)}\leq \abs{\log s-\log t}.
\end{align*}
An application of Corollary \ref{increasing_contractions} yields
\begin{equation}
\E(a,f_{n_k}(a))\to\I(a),\,k\to\infty.\qedhere
\end{equation}
\end{proof}

\begin{lemma}
Let $a\in D(\E)_+$. Then $\I(a)=\sup_n \I(a\wedge n)$.
\end{lemma}
\begin{proof}
By Corollary \ref{increasing_contractions}, the sequence $(\I(a\wedge n))_n$ is increasing and bounded from above by $\I(a)$. For the converse inequality let $(C_m)$ be a sequence as in the definition of $\I$.

Since $C_m(a)=\lim_{n\to\infty} C_m(a\wedge n)$ in $L^2(\M,\tau)$ and $\E(C_m(a\wedge n))\leq \E(C_m(a))$ by Corollary \ref{increasing_contractions}, we have $\E(C_m(a\wedge n))\to \E(C_m(a))$ as $n\to\infty$ by the lower semicontinuity of $\E$. The same argument holds for $\E(a\wedge n)$ so that we get
\begin{align*}
\I(a)\geq \I(a\wedge n)\geq \E(a\wedge n,C_m(a\wedge n))\to\E(a,C_m(a)).
\end{align*}
Taking the supremum over $m\in\IN$, the assertion follows.
\end{proof}

With the aid of the previous lemma, we can extend the Fisher information to $L^1_+(\M,\tau)$ via $\I(\rho)=\sup_n \I(\rho\wedge n)$ if $\rho\wedge n\in D(\E)_+$ for all $n\in\IN$ and $\I(\rho)=\infty$ otherwise.

\begin{lemma}\label{Fisher_info_lsc}
The Fisher information is lower semicontinuous on $L^1_+(\M,\tau)$.
\end{lemma}
\begin{proof}
By monotone approximation it suffices to show that maps of the form
\begin{align*}
L^1_+(\M,\tau)\lra [0,\infty),\,\rho\mapsto \begin{cases}\E(\rho\wedge n,C(\rho\wedge n))&\text{if }\rho\wedge n\in D(\E),\\\infty&\text{otherwise}\end{cases}
\end{align*}
are continuous for $n\in \IN$ and $C\in C^1(\IR)$ with $C'>0$.

Let $(\rho_k)$ be a sequence in $L^1_+(\M,\tau)$ such that $\rho_k\to \rho$ in $L^1(\M,\tau)$ and $\sup_k \E(\rho_k\wedge n,C(\rho_k\wedge n))<\infty$. By \cite[Theorem 3.2]{Tik87} we have $\rho_k\wedge n\to \rho\wedge n$ in $L^1(\M,\tau)$ and $L^2(\M,\tau)$. Moreover, since $C'>0$, the sequence $(\E(\rho_k\wedge n))_k$ is bounded. Thus $\rho\wedge n\in D(\E)$ and $\rho_k\wedge n\to \rho\wedge n$ weakly in $D(\E)$.

By Lemma \ref{theta_bounded_cont} we have $\tilde C(L(\rho_k\wedge n),R(\rho_k\wedge n))\to \tilde C(L(\rho\wedge n),R(\rho\wedge n))$ strongly as $k\to \infty$. If we combine these convergences, we obtain
\begin{align*}
\E(\rho_k\wedge n,C(\rho_k\wedge n))&=\langle \partial(\rho_k\wedge n),\tilde C(L(\rho_k\wedge n),R(\rho_k\wedge n))\partial(\rho_k\wedge n)\rangle_\H\\
&\to \langle\partial(\rho\wedge n),\tilde C(L(\rho\wedge n),R(\rho\wedge n))\partial(\rho\wedge n)\rangle_\H\\
&=\E(\rho\wedge n,C(\rho\wedge n)).\qedhere
\end{align*}
\end{proof}

\begin{proposition}\label{Fisher_info_convex}
The Fisher information is convex and weakly lower semicontinuous on $L^1_+(\M,\tau)$.
\end{proposition}
\begin{proof}
For the convexity it suffices to show that $\I$ is convex on $D(\E)_+\cap \M$. Indeed, if $\rho_0,\rho_1\in L^1_+(\M,\tau)$ and $\lambda\in [0,1]$, then
\begin{align*}
\I((1-\lambda)\rho_0+\lambda \rho_1)\leq\liminf_{n\to\infty}\I((1-\lambda)(\rho_0\wedge n)+\lambda (\rho_1\wedge n))
\end{align*}
by the lower semicontinuity of $\I$, and
\begin{align*}
(1-\lambda)\I(\rho_0\wedge n)+\lambda \I(\rho_1\wedge n)\leq (1-\lambda)\I(\rho_0)+\lambda \I(\rho_1)
\end{align*}
by definition.

For that purpose let
\begin{align*}
C_k\colon [0,\infty)\lra [0,\infty),\,t\mapsto \log(t+e^{-n})+n
\end{align*}
and let $\E_\epsilon$ be the quadratic form generated by $\L(1+\epsilon \L)^{-1}$. Since
\begin{align*}
\I(a)=\lim_{k\to\infty}\lim_{\epsilon\to 0}\E_\epsilon(a,C_k(a))
\end{align*}
for $a\in D(\E)_+\cap \M$, it is sufficient to prove that $a\mapsto \E_\epsilon(a,C_k(a))$ is convex for all $\epsilon>0$ and $k\in\IN$.

Let $(\H_\epsilon,\partial_\epsilon,L_\epsilon,R_\epsilon,J_\epsilon)$ be the first-order differential calculus associated with $\E_\epsilon$. By \cite[Section 10.3]{CS03} there exists $\eta_\epsilon\in \H_\epsilon$ such that $\partial_\epsilon a=(L_\epsilon(a)-R_\epsilon(a))\eta_\epsilon$ for all $a\in D(\E)_+\cap\M$. Thus
\begin{align*}
\E_\epsilon(a,C_k(a))=\langle (L_\epsilon(a)-R_\epsilon(a))\eta_\epsilon,(C_k(L_\epsilon(a))-C_k(R_\epsilon(a)))\eta_\epsilon\rangle_{\H_\epsilon}.
\end{align*}
For $\alpha>0$ let
\begin{align*}
\Phi_\alpha\colon (0,\infty)^2\lra \IR,\,(s,t)\mapsto \frac{s^{\alpha+1}t^\alpha+t-s^\alpha t^{1-\alpha}-t}{\alpha}.
\end{align*}
Note that $\Phi_\alpha(s,t)\to (s-t)(\log s-\log t)$ as $\alpha\searrow 0$. By dominated convergence we have 
\begin{align*}
\Phi_\alpha(L_\epsilon(a)+e^{-k},R_\epsilon(a)+e^{-k})\to (L_\epsilon(a)-R_\epsilon(a))(C_k(L_\epsilon(a))-C_k(R_\epsilon(a))).
\end{align*}
as $\alpha\searrow 0$ in the weak operator topology. Hence it suffices to show that
\begin{align*}
D(\E)_+\cap\M\lra B(\H),\,a\mapsto\Phi_\alpha(L_\epsilon(a)+e^{-k},R_\epsilon(a)+e^{-k})
\end{align*}
is convex for all $k\in \IN$ and $\alpha,\epsilon>0$.

Let $l(a)=L_\epsilon(a)+e^{-k}$ and $r(a)=R_\epsilon(a)+e^{-k}$. Since the images of $L_\epsilon$ and $R_\epsilon$ commute, we have
\begin{align*}
\Phi_\alpha(l(a),r(a))=\frac 1 \alpha(l(a)(l(a)\#_\alpha r(a))^{-1}l(a)+r(a)-l(a)\#_\alpha r(a)-l(a)),
\end{align*}
where $\#_\alpha$ is the operator mean with generating function $t\mapsto t^\alpha$. As mentioned before, the mean $\#_\alpha$ is jointly operator concave. Together with the joint operator convexity of the map $(x,y)\mapsto yx^{-1} y$ this implies the convexity of $a\mapsto l(a)(l(a)\#_\alpha r(a))^{-1}l(a)$. Hence $a\mapsto \Phi_\alpha(l(a),r(a))$ is convex as the sum of convex maps.

Finally, the weak lower semicontinuity follows from the convexity and the strong lower semicontinuity by the Hahn-Banach theorem.
\end{proof}

\begin{proposition}\label{heat_flow_admissible}
Assume that $\theta$ is the logarithmic mean. If $\rho\in\D(\M,\tau)\cap L^2(\M,\tau)$ and
\begin{align*}
\int_s^t \I(P_r\rho)\,dr<\infty
\end{align*}
for all $s,t>0$, then the curve $(P_t \rho)_{t>0}$ is admissible and
\begin{align*}
\norm{D(P_t\rho)}_{P_t\rho}^2\leq \I(P_t \rho).
\end{align*}
\end{proposition}
\begin{proof}
Let 
\begin{align*}
C_n\colon (e^{-n},\infty)\lra[0,\infty),\,t\mapsto \log(t+e^{-n})+n.
\end{align*}Let $\rho_t=P_t \rho$ and $\xi_t^n=-\partial (C_n(\rho_t))$. Since $\rho\in L^2(\M,\tau)$, we have $\rho_t\in D(\L)$. As $C_n$ is continuously differentiable on a neighborhood of $\sigma(\rho_t)$, we have $\partial (C_n(\rho_t))=\tilde C_n(L(\rho_t),R(\rho_t))\partial \rho_t$.

Denote by $e$ the joint spectral measure of $L(\rho_t)$ and $R(\rho_t)$. An application of $0\leq \tilde C_n\cdot\mathrm{LM}\leq 1$ gives
\begin{align*}
\norm{\xi_t^n}_{\rho_t}^2&=\norm{(\mathrm{LM}^{1/2}\tilde C_n)(L(\rho_t),R(\rho_t))\partial\rho_t}_{\H}^2\\
&=\int_{[0,\infty)^2}\mathrm{LM}(s,t)\tilde C_n(s,t)^2\,d\langle e(s,t)\partial\rho_t,\partial \rho_t\rangle_\H\\
&\leq \int_{[0,\infty)^2}\tilde C_n(s,t)\,d\langle e(s,t)\partial\rho_t,\partial\rho_t\rangle_\H\\
&=\E(\rho_t,C_n(\rho_t))\\
&\leq \I(\rho_t).
\end{align*}
On the other hand, $\tilde C_n \mathrm{LM}\nearrow 1$ implies 
\begin{align*}
\langle \xi_t^n,\partial a\rangle_{\rho_t}&=-\int_{[0,\infty)^2}\tilde C_n(s,t)\mathrm{LM}(s,t)\,d\langle e(s,t)\partial\rho_t,\partial a\rangle_\H\\
&\to -\langle \partial\rho_t,\partial a\rangle_\H\\
&=\langle -\L\rho_t,a\rangle_{L^2(\M,\tau)}\\
&=\tau(a \dot\rho_t)
\end{align*}
for all $a\in \A_\theta$.

Let $\tilde\xi_t^n$ be the projection of $\xi_t^n$ onto $\H_{\rho_t}$. From the computations above we conclude that $(\tilde\xi_t^n)_n$ converges weakly to some $\xi_t\in \H_{\rho_t}$ with $\norm{\xi_t}_{\rho_t}^2\leq \I(\rho_t)$ for a.e. $t>0$ and 
\begin{align*}
\langle \xi_t,\partial a\rangle_{\rho_t}=\tau(a\dot\rho_t)
\end{align*}
for all $a\in \A_\theta$.
\end{proof}

\begin{proposition}[Fisher information equals entropy dissipation]\label{entropy_Fisher_decay}
Assume that $\tau$ is finite. If $\rho\in \D(\M,\tau)\cap L^2(\M,\tau)$, then
\begin{align*}
\Ent(\rho)-\Ent(P_t \rho)=\int_0^t \I(P_s \rho)\,ds
\end{align*}
for all $t\geq 0$. In particular, $\Ent$ is decreasing along $(P_t\rho)_{t\geq 0}$.
\end{proposition}
\begin{proof}
Let $C_n=(\log\wedge n)\vee(-n)+n$ and
\begin{align*}
f_n\colon [0,\infty)\lra\IR,\,f_n(t)=-e^{-1}+\int_{e^{-1}}^t (C_n(r)+1-n)\,dr.
\end{align*}
Then $f_n\in C^1((-e^{-n},\infty))$, $f_n(t)\nearrow t\log t$ for all $t\geq 0$, and $f_n'=C_n-n+1$ is bounded and Lipschitz. Define furthermore
\begin{align*}
F_n\colon [0,\infty)\lra \IR,\,s\mapsto \tau(f_n(P_s \rho)).
\end{align*}

Since $s\mapsto P_s\rho$ is $L^1$-differentiable, we can apply Corollary \ref{diff_trace_functional} to see that $F_n$ is locally absolutely continuous and $F_n'(s)=-\tau(f_n'(P_s \rho)\L P_s\rho)$.

Thus
\begin{align*}
F_n(\rho)-F_n(P_t\rho)&=\int_0^t \tau(f_n'(P_s \rho)\L P_s\rho)\,s\\
&=\int_0^t \E(C_n(P_s\rho),P_s\rho)\,ds\\
&\to \int_0^t \I(P_s\rho)\,ds,\,n\to\infty.
\end{align*}
Here we used conservativeness for the fact that $\tau(\L \bar\rho)=0$ for all $\bar\rho\in L^1(\M,\tau)$.

Since $\rho\in \D(\M,\tau)\cap L^2(\M,\tau)\subset D(\Ent)$, the monotone convergence theorem gives the convergence of the left-hand side to $\Ent(\rho)-\Ent(P_t\rho)$.
\end{proof}

\begin{corollary}\label{heat_flow_L1}
Assume that $\theta$ is the logarithmic mean. If $\rho\in D(\Ent)$, then 
\begin{align*}
\int_0^\infty\I(P_r\rho)\,dr\leq \Ent(\rho),
\end{align*}
$(P_t\rho)_{t\geq 0}$ is an admissible curve, $\Ent$ is decreasing along $(P_t \rho)_{t\geq 0}$ and $\W(\rho,P_t\rho)\to 0$ as $t\to 0$.
\end{corollary}
\begin{proof}
Let $\rho_n=\frac{\rho\wedge n}{\tau(\rho\wedge n)}$. Since $\rho_n\to \rho$ in $L^1(\M,\tau)$, one has $P_t \rho_n\to P_t \rho$ in $L^1(\M,\tau)$ for all $t\geq 0$. Moreover,
\begin{align*}
\Ent(\rho_n)=\tau(\rho_n \log \rho_n)=\frac 1{\tau(\rho\wedge n)}\tau((\rho\wedge n)\log(\rho\wedge n))-\log\tau(\rho\wedge n)\to \Ent(\rho).
\end{align*}
It follows from the lower semicontinuity of the entropy and Proposition \ref{entropy_Fisher_decay} that
\begin{align*}
\Ent(P_t \rho)\leq \liminf_{n\to\infty}\Ent(P_t \rho_n)\leq \lim_{n\to\infty}\Ent(\rho_n)=\Ent(\rho)
\end{align*}
for all $t\geq 0$. Thus $\Ent$ is decreasing along $(P_t \rho)_{t\geq 0}$.

The lower semicontinuity of $\I$ (Lemma \ref{Fisher_info_lsc}) and Fatou's lemma imply
\begin{align*}
\int_0^\infty\I(P_r\rho_n)\,dr\leq \int_0^\infty\liminf_{n\to\infty}\I(P_r\rho_n)\,dr\leq \liminf_{n\to\infty}\int_0^\infty\I(P_r \rho_n)\,dr.
\end{align*}
From Proposition \ref{entropy_Fisher_decay} we deduce
\begin{align*}
\liminf_{n\to\infty}\int_0^\infty\I(P_r\rho_n)\,dr\leq\lim_{n\to\infty}\Ent(\rho_n)=\Ent(\rho).
\end{align*}
Moreover, 
\begin{align*}
\liminf_{n\to\infty}\int_0^\infty \norm{D(P_r \rho_n)}_{P_r \rho_n}^2\,dr\leq \liminf_{n\to\infty}\int_0^\infty\I(P_r\rho_n)\,dr
\end{align*}
by Proposition \ref{heat_flow_admissible}. Thus $(P_t\rho)_{t\geq 0}$ is admissible by Theorem \ref{energy_lsc}.

Finally, $\W(\rho,P_t \rho)\to 0$ as $t\to 0$ is a direct consequence of the admissibility of $(P_t \rho)_{t\geq 0}$.
\end{proof}

\section{Gradient flow of the entropy}\label{heat_gradient_flow}

In this section we give the announced characterization of the quantum Markov semigroup as metric gradient flow of the entropy under suitable conditions.

To be more precise, in Subsection \ref{subsec_GE} we first introduce a Bakry--Émery-type gradient estimate for the semigroup $(P_t)$ and prove that it implies a kind of Feller regularization for the semigroup (Proposition \ref{A_LM_regularization}) as well as contraction/expansion estimates with respect to $\W$ (Theorem \ref{W_contraction}).

A central difficulty in the proof of the gradient flow characterization, as already in the case of metric measure spaces, lies in the fact that we are working on $L^1$, so that the heat flow curves may fail to be differentiable and Hilbert space methods are not directly applicable. To overcome this problem, Subsection \ref{mollification} is devoted to a fine analysis of standard semigroup mollification in our setting. In particular, we prove an entropy regularization estimate in Proposition \ref{Ent_regularization}.

In the last Subsection \ref{proof_main} we review the evolution variational inequality (EVI) formulation of gradient flows in metric spaces and complete the proof of the characterization of the Markovian quantum master equation as EVI gradient flow of the entropy (Theorem \ref{EVI_gradient_flow}).

In part our proof strategy is a careful adaptation to the noncommutative setting of the paths taken by Ambrosio et al. (see \cite{AGS15} in the case of infinitesimally Hilbertian metric measure spaces and \cite{AES16} in the case of abstract local Dirichlet forms). However, both proofs rely strongly on duality (either the dual problem of the Monge--Kantorovich or the Benamou--Brenier formulation). Little is known on the dual formulation in the present setting, so we avoid it altogether. In this way, our approach gives a new proof variant even when restricted to the case of infinitesimally Hilbertian metric measure spaces (with finite measure).

As usual, let $(\M,\tau)$ be a tracial von Neumann algebra, $\E$ a quantum Dirichlet form on $L^2(\M,\tau)$ such that $\tau$ is energy dominant, and $(\partial,\H,L,R,J)$ the associated first order differential calculus. We further assume that $\tau$ is a state and $L^1(\M,\tau)$ is separable. In the beginning we only assume that $\theta$ can be represented by a symmetric operator mean, but starting from Subsection \ref{mollification} we take the logarithmic mean for $\theta$. We do not make any density assumptions on $\A_\theta$ -- these follow automatically from the gradient estimate we introduce in the first subsection.

\subsection{The gradient estimate $\GE(K,\infty)$}\label{subsec_GE}

In this subsection we introduce the gradient estimate 
\begin{equation*}
\norm{\partial P_t a}_\rho^2\leq e^{-2Kt}\norm{\partial a}_{P_t\rho}^2,
\end{equation*}
which is inspired by the classical Bakry--Émery estimate (see the original paper \cite{BE85} or the monograph \cite{BGL14})
\begin{equation*}
\Gamma(P_t u)\leq e^{-2Kt}P_t \Gamma(u).
\end{equation*}
In the case of the heat semigroup on a complete Riemannian manifold, the Bakry--Émery gradient estimate is equivalent to a lower bound on the Ricci curvature. Recently, it has been shown that this equivalence between gradient estimates and lower Ricci curvature bounds (in the sense of Lott--Villani--Sturm) extends to a large class of metric measure spaces \cite{AGS14b,AGS15,EKS15}. Thus the gradient estimate from above can be interpreted as the noncommutative version of a lower Ricci curvature bound.

One of the reasons the gradient estimate is important for the gradient flow characterization is that it provides certain regularizing effects of the semigroup. In this subsection we investigate some of these consequences of the gradient estimates, among them an $L^\infty$-to-$\A_\theta$-regularization property of the semigroup (Proposition \ref{A_LM_regularization}) and an exponential contraction (or expansion) bound for the metric $\W$ (Theorem \ref{W_contraction}). 

Throughout this subsection let $(\M,\tau)$ be a tracial von Neumann algebra, $\E$ a quantum Dirichlet form on $L^2(\M,\tau)$, $(\partial,\H,L,R,J)$ the associated first-order differential calculus and assume that $\tau$ is energy dominant. We further assume that $\theta$ can be represented by a symmetric operator mean. Denote by $(P_t)_{t\geq 0}$ the quantum Markov semigroup associated with $\E$ and by $\L=\partial^\ast\partial$ its generator.

\begin{definition}[$\GE(K,\infty)$]\label{def_GE}
Let $K\in\IR$. The quantum Dirichlet form $\E$ satisfies the gradient estimate $\GE(K,\infty)$ (for the mean $\theta$) if
\begin{align*}
\norm{\partial P_t a}_{\rho}^2\leq e^{-2Kt}\norm{\partial a}_{P_t \rho}^2
\end{align*}
for all $a\in D(\E)$, $\rho\in\D(\M,\tau)$ and $t\geq 0$.
\end{definition}

\begin{remark}\label{GE_manifold}
The gradient estimate $\GE(K,\infty)$ is a modification of the classical Bakry--Émery gradient estimate. Indeed, if $\E$ is a strongly local (commutative) Dirichlet form on $L^2(X,m)$, then the gradient estimate $\GE(K,\infty)$ reads
\begin{align*}
\int_X \Gamma(P_t u)\rho\,dm\leq e^{-2Kt}\int_X P_t\Gamma(u)\rho\,dm,
\end{align*}
which is just a weak formulation of the Bakry--Émery gradient estimate 
\begin{align*}
\Gamma(P_t u)\leq e^{-2Kt}P_t\Gamma(u).
\end{align*}
If $\E$ is the standard Dirichlet energy on a complete Riemannian manifold $(M,g)$, then $\E$ satisfies $\mathrm{GE}(K,\infty)$ if and only if the Ricci curvature of $(M,g)$ is bounded below by $K$.
\end{remark}

\begin{remark}
In the classical Bakry--Émery gradient estimate there is an additional dimension parameter $N$, and $\mathrm{BE}(K,\infty)$ corresponds to the case $N=\infty$. That is why we keep the parameter $\infty$ in our notation although we do not introduce any finite-dimensional variant of $\GE(K,\infty)$.
\end{remark}

\begin{remark}\label{GE_graph}
For finite graphs, the gradient estimate $\GE(K,\infty)$ was introduced in \cite{EF18}, where it was shown to be equivalent to $K$-convexity of the entropy along $\W$-geodesics (\cite[Theorem 3.1]{EF18}). The latter was taken as definition for a lower Ricci curvature bound $K$ of a graph in \cite{EM14} (see also \cite{Mie13}).
\end{remark}

Some examples of quantum Dirichlet forms satisfying the gradient estimate will be studied more systematically in upcoming work. Here we just give a simple (commutative) example.

\begin{example}
Endow $\IZ^d$ with the natural graph structure given by
\begin{align*}
b(m,n)=\begin{cases}1&\text{if }\abs{m-n}=1,\\0&\text{otherwise,}\end{cases}
\end{align*}
and the uniform measure $m=1$.

Since $(\IZ^d,b,m)$ has bounded degree, the Dirichlet forms $\E^{(D)}$ and $\E^{(N)}$ introduced in Example \ref{ex_graphs} coincide, and we simply denote them by $\E$. We will show that $\E$ satisfies $\GE(0,\infty)$ for the logarithmic mean. By a simple approximation argument it suffices to prove
\begin{align*}
\norm{\partial P_t u}_\rho^2\leq \norm{\partial u}_{P_t\rho}^2
\end{align*}
for $u\in C_c(\IZ^d)$ and $\rho\in C_c(\IZ^d)_+$.

For this, we will approximate $\IZ^d$ by finite graphs and use known results for finite Markov chains. Let $K_j=\{-j,\dots,j\}^d$ and
\begin{align*}
L_j\colon \ell^2(K_j)\to \ell^2(K_j),\,L_j u(m)=\sum_{n\in K_j}b(m,n)(u(m)-u(n)).
\end{align*}
This is the generator of the Dirichlet form $\E_j$ associated with the weighted graph $(K_j,b|_{K_j\times K_j},1)$. Let $(P_t^j)$ denote the associated semigroup. It follows from \cite[Theorem 4.1]{FM16} in combination with \cite[Theorem 6.2]{EM14} or \cite[Theorem 5.1]{Mie13} that $\E_j$ satisfies $\GE(0,\infty)$ for the logarithmic mean.

We extend $L_j$ to an operator on $\ell^2(\IZ^d)$ by setting it zero on $\ell^2(\IZ^d\setminus K_j)$. Clearly, if $u\in C_c(\IZ^d)$ with $\operatorname{supp} u\subset K_j$, then $L_{j+1}u=L u$, where $L$ is the generator of $\E$. Since the sequence $(L_j)$ is uniformly bounded, it follows that $L_j\to L$ strongly. Thus also $e^{-tL_j}\to e^{-tL}$ strongly.

Since $\E_j$ satisfies $\GE(0,\infty)$, we have
\begin{align*}
\frac1  2\sum_{\substack{m,n\in K_j\\m\sim n}}\hat \rho(m,n)(P_t^j u(m)-P_t^j u(n))^2\leq \frac 1 2\sum_{\substack{m,n\in K_j\\m\sim n}}\widehat{P_t^j \rho}(m,n)(u(m)-u(n))^2
\end{align*}
for all $u\in C_c(\IZ^d)$, $\rho\in C_c(\IZ^d)_+$ and $j\in \IN$ such that $\operatorname{supp} u,\operatorname{supp}\rho\subset K_j$. By Fatou's lemma, the limes inferior as $j\to\infty$ of the left-hand side is bounded below by $\norm{\partial P_t u}_\rho^2$, while the right-hand side converges to $\norm{\partial u}_{P_t\rho}^2$ by the dominated convergence theorem. Hence
\begin{align*}
\norm{\partial P_t u}_{\rho}^2\leq \norm{\partial u}_{P_t\rho}^2.
\end{align*}
\end{example}
More generally, this method deducing a gradient estimate for a graph from gradient estimates for an exhaustion of finite subgraphs can be applied if the Laplacians of the subgraphs converge to the Laplacian of the original graph in the strong resolvent sense. This is the case for instance if the graph is locally finite and stochastically complete (in this case one can easily adapt the arguments from \cite[Proposition 2.7]{KL12}).

\begin{remark}
If $\theta$ is the arithmetic mean, then $\GE(K,\infty)$ reads
\begin{align*}
\Gamma(P_t a)\leq e^{-2Kt}P_t \Gamma(a)
\end{align*}
for $a\in D(\E)$ self-adjoint. This noncommutative form of the Bakry--Émery gradient estimate was used for example in \cite{JZ15} in the study of noncommutative Poincaré inequalities.

In general, the gradient estimate $\GE(K,\infty)$ is not equivalent to the (noncommutative) Bakry--Émery gradient estimate in this form, even in the simplest (commutative) examples. For instance, if $\E$ is the Dirichlet form associated with the weighted graph $(\{0,1\},b,m)$ with $b(0,1)>0$, then one can show that the best possible constant $K$ in $\GE(K,\infty)$ for the logarithmic mean coincides with the best possible constant for the arithmetic mean  if and only if $m(0)=m(1)$.
\end{remark}

\begin{proposition}[Feller property]\label{A_LM_regularization}
If $\E$ satisfies $\GE(K,\infty)$ for some $K\in\IR$, then $P_t$ maps $L^2(\M,\tau)\cap\M$ into $\A_\theta$ for $t>0$.
\end{proposition}
\begin{proof}
Let $a\in L^2(\M,\tau)\cap\M_h$ and $t>0$. Since $P_t$ maps $L^2(\M,\tau)$ into $D(\L)$, we can assume $a\in D(\L)\cap \M$. For $\rho\in D(\L)\cap L^1_+(\M,\tau)\cap\M$ define
\begin{align*}
\phi\colon [0,t]\lra\IR,\,\phi(s)=\int_0^s \norm{\partial P_{t-r}a}_{P_r\rho}^2\,dr.
\end{align*}
Note that $r\mapsto \widehat{P_r\rho}$ is strongly continuous by Lemma \ref{theta_bounded_cont} and $r\mapsto \partial P_{t-r}a$ is continuous in $\H$, so that the integrand is continuous in $r$.

Since $\E$ satisfies $\GE(K,\infty)$, the map $s\mapsto e^{-2Ks}\phi'(s)$ is increasing. It follows from a comparison argument (see \cite[Lemma 2.2 and Equation (2.30)]{AGS15}) that
\begin{align}
I_{2K}(t)\norm{\partial P_t a}_\rho^2=I_{2K}(t)\phi'(0)\leq \int_0^t \phi'(s)\,ds=\int_0^t \norm{\partial P_{t-s} a}_{P_s\rho}^2\,ds,\label{bound_phi_prime}
\end{align}
where $I_{\kappa}(t)=\int_0^t e^{\kappa s}\,ds$.

By Lemma \ref{theta_leq_AM} and a direct calculation we have
\begin{align*}
\norm{\partial P_{t-s} a}_{P_s\rho}^2\leq \tau(\Gamma(P_{t-s}a)P_s \rho)=\frac 1 2\frac {d}{ds}\tau((P_{t-s}a)^2 P_s \rho).
\end{align*}
If we plug this into (\ref{bound_phi_prime}), we get
\begin{align*}
I_{2K}(t)\norm{\partial P_t a}_\rho^2\leq \frac 1 2 \tau((P_t(a^2)-(P_t a)^2)\rho).
\end{align*}
Thus
\begin{align*}
\norm{\partial P_t a}_{\rho}^2\leq \frac 1{2 I_{2K}(t)}\norm{a}_\M^2\norm{\rho}_1.
\end{align*}
In the general case $\rho\in L^1_+(\M,\tau)$, one can use Lemma \ref{theta_increasing} and once more $\GE(K,\infty)$ to also get
\begin{align*}
\norm{\partial P_t a}_\rho^2&=\lim_{m\to\infty}\norm{\partial P_t a}_{\rho\wedge m}^2\\
&\leq \liminf_{m\to\infty}\,\liminf_{\delta\to 0}e^{-2K\delta}\norm{\partial P_{t-\delta}a}_{P_\delta (\rho\wedge m)}^2\\
&\leq \liminf_{m\to\infty}\,\liminf_{\delta\to 0}\frac{e^{-2K\delta}}{2 I_{2K}(t-\delta)}\norm{a}_\M^2 \norm{\rho\wedge m}_1\\
&=\frac 1{2 I_{2K}(t)}\norm{a}_\M^2\norm{\rho}_1.
\end{align*}
Hence $P_t a\in \A_\theta$.
\end{proof}

\begin{remark}
If $\E$ is a strongly local (commutative) Dirichlet form, then Proposition \ref{A_LM_regularization} recovers the $L^\infty$-to-Lipschitz Feller property of the heat flow on $\mathrm{RCD}$ spaces from \cite[Theorem 6.8]{AGS14b} (with essentially the same proof). If $\E$ is not strongly local or not commutative, the correct analogue of the algebra of bounded Lipschitz functions seems to be $\A_{\mathrm{AM}}$ (see Example \ref{ex_AM}), so that Proposition \ref{A_LM_regularization} is in general weaker than (noncommutative) $L^\infty$-to-Lipschitz regularization unless $\theta$ is the arithmetic mean.
\end{remark}

\begin{corollary}\label{A_LM_dense}
If $\E$ satisfies $\GE(K,\infty)$, then $\A_{\theta}\cap\M_1$ is dense in $D(\E)\cap \M_1$ with respect to $\norm\cdot_\E$ and strongly dense in $\M_1$.
\end{corollary}
\begin{proof}
The density of $\A_\theta\cap\M_1$ in $D(\E)\cap \M_1$ follows directly from Proposition \ref{A_LM_regularization}. The strong density in $\M_1$ is then a consequence of Kaplansky's density theorem.
\end{proof}

We now fulfill the promise made in Remark \ref{rmk_admissible_curves} and introduce a regularity property that ensures that the duality in the definition of admissible curves can be extended to $\A_\theta$.

\begin{definition}[Regular mean]
The mean $\theta$ is called \emph{regular} for $\E$ if for all $a\in \A_\theta$ there exists a sequence $(a_n)$ in $\A_{\AM}$ such that $a_n\to a$ $\sigma$-weakly, $(a_n)$ is bounded in $A_\theta$ and $\partial a_n\to \partial a$ weakly in $\tilde \H_\rho$ for all $\rho\in \D(\M,\tau)$.
\end{definition}

The next lemma then follows immediately from this definition.

\begin{lemma}\label{admissibility_A_theta}
If the mean $\theta$ is regular for $\E$, then $\partial \A_\theta\subset \H_\rho$ for all $\rho\in \D(\M,\tau)$ and if $(\rho_t)_{t\in I}$ is an admissible curve in $\D(\M,\tau)$, then
\begin{equation*}
\tau(a(\rho_t-\rho_s))=\int_s^t \langle \partial a, D\rho_r\rangle_{\rho_r}\,dr
\end{equation*}
for all $a\in \A_\theta$ and $s,t\in I$.
\end{lemma}

Of course, the arithmetic mean is regular for every quantum Dirichlet form. For other means it may not be easy to check whether it is regular for a given quantum Dirichlet form. However, the gradient estimate $\GE(K,\infty)$ yields a sufficient condition for regularity of means. More precisely, we have the following result.

\begin{proposition}\label{GE_regular_mean}
Let $\theta$ be a mean and $K\in \IR$. If $\E$ satisfies $\GE(K,\infty)$ for $\theta$ and the arithmetic mean, then $\theta$ is regular for $\E$.
\end{proposition}
\begin{proof}
For $a\in \A_\theta$ let $a_n=P_{t_n}a$ for some null sequence $(t_n)$. Since $\E$ satisfies $\GE(K,\infty)$ for $\AM$, we have $a_n\in \A_{\AM}$ by Proposition \ref{A_LM_regularization}.

The $\sigma$-weak convergence of $(a_n)$ to $a$ is a consequence of the continuity of $(P_t)$. Moreover, $\GE(K,\infty)$ for $\theta$ implies
\begin{align*}
\norm{\partial a_n}_{\rho}\leq e^{-Kt_n}\norm{\partial a}_{P_t\rho}\leq \max\{1,e^{-K\sup_n t_n}\}\norm{a}_\theta
\end{align*}
for all $\rho\in \D(\M,\tau)$, where the norms $\norm{\cdot}_\rho$ and $\norm{\cdot}_{P_t\rho}$ are to be understood with respect to the mean $\theta$. Hence $(a_n)$ is bounded in $\A_\theta$.

Since $\norm{\partial a_n}_\rho+\norm{\partial a_n}_\H$ is bounded and $(D(\hat\rho^{1/2}),\langle\cdot,\cdot\rangle_\H+\langle\cdot,\cdot\rangle_{\rho})$ is complete, every subsequence of $(\partial a_n)$ has a subsequence that converges weakly in $D(\hat\rho^{1/2})$. Moreover, since $\partial a_n\to\partial a$ in $\H$, the weak limit is $\partial a$. Finally, since $D(\hat\rho^{1/2})$ is dense in $\tilde\H_\rho$ and $\norm{\partial a_n}_\rho$ is bounded, we also have $\partial a_n\to\partial a$ weakly in $\tilde\H_\rho$.
\end{proof}

\begin{theorem}[Contraction estimate]\label{W_contraction}
Assume that $\theta$ is regular for $\E$. If $\E$ satisfies $\GE(K,\infty)$ for $\theta$ and $(\rho_t)_{t\in I}$ is an admissible curve in $\D(\M,\tau)$, then $(P_T\rho_t)_{t\in I}$ is an admissible curve and $\norm{DP_T\rho_t}_{P_T\rho_t}\leq e^{-KT}\norm{D\rho_t}_{\rho_t}$ for a.e. $t\in I$.

In particular, $\W(P_T \rho_0,P_T \rho_1)\leq e^{-KT}\W(\rho_0,\rho_1)$ for $\rho_0,\rho_1\in\D(\M,\tau)$ and $T\geq 0$.
\end{theorem}
\begin{proof}
Let $(\rho_t)_{t\in I}$ be an admissible curve in $\D(\M,\tau)$. For all $s,t\in I$ and $a\in\A_\AM$ we have
\begin{align*}
\abs{\tau(a(P_T\rho_t-P_T\rho_s))}&=\abs{\tau(P_T a(\rho_t-\rho_s))}\\
&\leq \int_s ^t\norm{\partial P_T a}_{\rho_r}\norm{D\rho_r}_{\rho_r}\,dr\\
&\leq e^{-KT}\int_s^t \norm{\partial a}_{P_T \rho_r}\norm{D\rho_r}_{\rho_r}\,dr,
\end{align*}
where we used Proposition \ref{A_LM_regularization} and Lemma \ref{admissibility_A_theta} for the first and $\GE(K,\infty)$ for the second inequality.

Thus, $(P_T\rho_r)_{r\in [0,1]}$ is an admissible curve with $\norm{DP_T\rho_r}_{P_T\rho_r}\leq e^{-KT}\norm{D\rho_r}_{\rho_r}$ for a.e. $r\in I$. Minimizing over all admissible curves connecting $\rho_0$ and $\rho_1$ yields the second claim.
\end{proof}

\begin{corollary}
Assume that the logarithmic mean is regular for $\E$ and $\E$ satisfies $\GE(K,\infty)$ for the logarithmic mean. If $\rho\in \overline{D(\Ent)}^{\W}$, then $\W(P_t\rho,\rho)\to 0$ as $t\to 0$.
\end{corollary}
\begin{proof}
In Corollary \ref{heat_flow_L1} we have already seen that this convergence holds for $\rho\in D(\Ent)$. If $\rho\in \overline{D(\Ent)}^\W$, let $(\rho_k)$ be a sequence in $D(\Ent)$ such that $\W(\rho_k,\rho)\to 0$. By Theorem\ref{W_contraction} we have
\begin{align*}
\W(P_t\rho,\rho)\leq (1+e^{-Kt})\W(\rho,\rho_k)+\W(P_t\rho_k,\rho_k).
\end{align*}
Letting first $t\to 0$ and then $k\to \infty$ yields the claimed convergence.
\end{proof}

\subsection{Mollification}\label{mollification}

A central difficulty in the proof of the gradient flow characterization, as already in the case of metric measure spaces, lies in the fact that we are working on $L^1$, so that the heat flow curves may fail to be differentiable and Hilbert space methods are not directly applicable. To overcome this problem, Section~\ref{mollification} is devoted to a fine analysis of semigroup mollification techniques in our setting. In particular, we prove an entropy regularization estimate in Proposition~\ref{Ent_regularization} that will be crucial in the proof of the main theorem.

As usual, let $(\M,\tau)$ be a tracial von Neumann algebra, $\E$ a quantum Dirichlet form on $L^2(\M,\tau)$ such that $\tau$ is energy dominant, and $(\partial,\H,L,R,J)$ the associated first-order differential calculus. We further assume that $\tau$ is a state, $L^1(\M,\tau)$ is separable and $\theta$ is the logarithmic mean.

We first introduce a mollified version of $(P_t)$, which is a standard tool in the theory of operator semigroups. Let $\kappa\in C_c^\infty((0,\infty))$ be a positive function with support in $(1,2)$ and $\int_0^\infty \kappa(r)\,dr=1$. For $\epsilon>0$ and $p\in [1,\infty]$ define
\begin{align*}
\p^\epsilon\colon L^p(\M,\tau)\lra L^p(\M,\tau),\,\p^\epsilon a=\frac 1 \epsilon\int_0^\infty \kappa\left(\frac r \epsilon\right)\,P_r a\,dr,
\end{align*}
where the integral is to be understood as Bochner integral if $p<\infty$ and as Pettis integral for the $\sigma$-weak topology if $p=\infty$. It is clear that $\p^\epsilon$ is positive and contractive on all $L^p$ spaces.

The following Lemma is standard, see for example the proof of \cite[Proposition 1.8]{EN00}.

\begin{lemma}\label{regularity_p_epsilon}
Let $p\in[1,\infty]$ and $a\in L^p(\M,\tau)$. For all $\epsilon>0$ one has $\p^\epsilon a\in D(\L^{(p)})$ and
\begin{align*}
\L^{(p)}\p^\epsilon a=\frac 1{\epsilon^2}\int_0^\infty \kappa'\left(\frac r \epsilon\right) P_r a\,dr.
\end{align*}
If $p<\infty$, then $\p^\epsilon a \to a$ in $L^p(\M,\tau)$, and if $p=\infty$, then $\p^\epsilon a\to a$ $\sigma$-weakly as $\epsilon\to 0$.

If $a\in D(\E)$, then $\p^\epsilon a\in D(\E)$ and $\p^\epsilon a\to a$ with respect to $\norm\cdot_\E$ as $\epsilon\to 0$.
\end{lemma}

\begin{lemma}\label{p_epsilon_Feller_reg}
Assume that $\E$ satisfies $\GE(K,\infty)$ for some $K\in\IR$. If $a\in \M$, then $\p^\epsilon a\in \A_\mathrm{LM}$ for $\epsilon>0$. Moreover, if $a\in \A_\mathrm{AM}$, then 
\begin{align*}
\norm{\partial\p^\epsilon a}_\rho^2\leq \norm{\partial a}_\rho^2\sup_{r\in[0,2\epsilon]} e^{-2Kr}
\end{align*}
and $\p^\epsilon a\to a$ in $\tilde\H_\rho$ for all $\rho\in \D(\M,\tau)$.
\end{lemma}
\begin{proof}
Let $\rho\in \D(\M,\tau)\cap \M$ be invertible. By $\GE(K,\infty)$ we have
\begin{align*}
\norm{\partial P_r a}_\rho^2\leq e^{-2K(r-\epsilon)}\norm{\partial P_\epsilon a}_{P_{r-\epsilon}\rho}^2
\end{align*}
whenever $r\geq \epsilon$. The right-hand side is bounded above by $e^{-2K(r-\epsilon)}\norm{P_\epsilon a}_\mathrm{LM}^2$, which is finite by Proposition~\ref{A_LM_regularization}. By Lemma~\ref{rho_norm_lsc} the map $a\mapsto \norm{\partial a}_\rho^2$ is lower semicontinuous. Thus we can apply Jensen's inequality (compare the proof of Lemma~ \ref{W_diff_curves}) to see that
\begin{align*}
\norm{\partial \p^\epsilon a}_\rho^2\leq \frac 1 \epsilon \int_{\epsilon}^{2\epsilon} \kappa\left(\frac r \epsilon\right)\norm{\partial P_r a}_\rho^2\,dr\leq \frac {\norm{P_\epsilon a}_\mathrm{LM}^2}\epsilon \int_0^\infty e^{-2K(r-\epsilon)}\kappa\left(\frac r \epsilon\right)\,dr.
\end{align*}
The right-hand side is clearly bounded independently of $\rho$. Hence $\p^\epsilon a\in \A_\mathrm{LM}$.

Now let $\rho\in\D(\M,\tau)$ be arbitrary and assume that $a\in\A_\mathrm{AM}$. By $\GE(K,\infty)$ we have 
\begin{align*}
\norm{\partial\p^\epsilon a}_{\rho}^2\leq \frac 1 \epsilon \int_0^\infty \kappa\left(\frac r \epsilon\right)\norm{\partial P_r a}_{\rho}^2\,dr\leq \frac 1\epsilon\int_0^\infty \kappa\left(\frac r\epsilon\right)e^{-2Kr}\norm{\partial a}_{P_r\rho}^2\,dr.
\end{align*}
It follows from the upper semicontinuity and concavity of $\rho\mapsto \norm{\partial a}_\rho^2$ by an application of Jensen's inequality that
\begin{align*}
\norm{\partial \p^\epsilon a}_\rho^2\leq \left(\sup_{r\in[0,2\epsilon]}e^{-2Kr}\right)\frac 1 \epsilon\int_0^\infty\kappa\left(\frac r \epsilon\right)\norm{\partial a}_{P_r\rho}^2\,dr\leq \norm{\partial a}_{\p^\epsilon\rho}^2\sup_{r\in [0,2\epsilon]}e^{-2Kr}.
\end{align*}
Hence 
\begin{align}
\limsup_{\epsilon\to 0}\norm{\partial \p^\epsilon  a}_\rho^2\leq \limsup_{\epsilon\to 0}\norm{\partial a}_{\p^\epsilon \rho}^2\leq \norm{\partial a}_{\rho}^2
\label{bound_rho_norm_p_epsilon}
\end{align}
by Theorem~\ref{theta_norm_usc}.

Thus $(\partial \p^\epsilon a)_{\epsilon>0}$ is a bounded net in the Hilbert space $D(\hat\rho^{1/2})$ with inner product
\begin{align*}
\langle\cdot,\cdot\rangle_\H+\langle\hat\rho^{1/2}\,\cdot\,,\hat\rho^{1/2}\,\cdot\,\rangle_\H.
\end{align*}
On the other hand, since $\partial\p^\epsilon a\to \partial a$ in $\H$, every weak limit point of $(\p^\epsilon a)_{\epsilon>0}$ in $D(\hat\rho^{1/2})$ coincides with $\partial a$. Thus $\partial \p^\epsilon a\to \partial a$ also weakly in $\tilde\H_\rho$. Finally, (\ref{bound_rho_norm_p_epsilon}) implies that the convergence is indeed strong.
\end{proof}

\begin{lemma}\label{p_epsilon_Ent_decreasing}
Let $\rho\in \D(\M,\tau)$. For all $\epsilon>0$ one has $\Ent(\p^\epsilon(\rho))\leq \Ent(P_\epsilon\rho)$. If $\rho\in D(\Ent)$, then $\p^\epsilon(\rho)\in D(\Ent)$ and $\W(\rho,\p^\epsilon(\rho))\to 0$ as $\epsilon\to 0$.
\end{lemma}
\begin{proof}
Since the entropy is a convex lower semicontinuous functional and decreasing along heat flow trajectories, Jensen's inequality implies 
\begin{align*}
\Ent(\p^\epsilon(\rho))\leq \frac 1 \epsilon\int_\epsilon^{2\epsilon}\kappa\left(\frac r \epsilon\right)\Ent(P_r \rho)\,dr\leq\Ent(P_\epsilon\rho).
\end{align*}
If $\rho\in D(\Ent)$, then $(P_t \rho)_{t\geq 0}$ is admissible by Corollary~\ref{heat_flow_L1}.

Let
\begin{align*}
\sigma\colon [0,1]\lra \D(\M,\tau),\,\sigma_t=\frac 1 \epsilon\int_0^\infty \kappa\left(\frac r \epsilon\right)P_{rt}\rho\,dr
\end{align*}
and
\begin{align*}
\sigma^n\colon [0,1]\lra\D(\M,\tau),\,\sigma^n_t=\sum_{k=1}^{2n} P_{\frac{k-1}n\epsilon t}\rho\cdot\int_ {\frac{k-1}n }^{\frac{k} {n}}\kappa(r)\,dr.
\end{align*}
Since $(P_{rt}\rho)_{r\geq 0}$ is $L^1$ continuous, we have $\sigma^n_t\to\sigma_t$ in $L^1$ for all $t\in[0,1]$. As $\sigma^n$ is a convex combination of admissible curves, it is itself admissible and
\begin{align*}
\norm{D_t\sigma^n_t}_{\sigma^n_t}^2\leq \sum_{k=1}^{2n}\norm{D_t P_{\frac{k-1}n\epsilon t}\rho}^2\cdot\int_{\frac{k-1}{n}}^{\frac{k}{n}}\kappa(r)\,dr
\end{align*}
by Lemma~\ref{W_convex}.

It follows from Theorem~\ref{energy_lsc} and Corollary~\ref{heat_flow_L1} that $\sigma$ is an admissible curve connecting $\rho$ and $\p^\epsilon(\rho)$ with
\begin{align*}
\int_0^1\norm{D\sigma_t}_{\sigma_t}^2\,dt&\leq \liminf_{n\to\infty}\sum_{k=1}^{2n}\int_0^1 \norm{D_t P_{\frac{k-1}{n}\epsilon t}}^2\,dt\cdot\int_{\frac{k-1}{n}}^{\frac{k}{n}}\kappa(r)\,dr\\
&\leq2\epsilon \int_0^{2\epsilon}\norm{DP_t \rho}_{P_t\rho}^2\,dt\cdot\int_0^2 \kappa(r)\,dr\\
&\to 0
\end{align*}
as $\epsilon\to 0$.
\end{proof}

\begin{corollary}
The space $D(\L^{(1)})\cap D(\Ent)$ is dense in $(D(\Ent),\W)$.
\end{corollary}

\begin{lemma}\label{regularization_admissible}
Assume that the logarithmic mean is regular for $\E$ and $\E$ satisfies $\GE(K,\infty)$. If $(\rho_t)_{t\in [0,1]}$ is an admissible curve in $\D(\M,\tau)$, then $(\p^\epsilon\rho_t)_{t\in [0,1]}$ is admissible for all $\epsilon>0$ and
\begin{align*}
\limsup_{\epsilon\to 0}\int_0^1 \norm{D_t\p^\epsilon(\rho_t)}_{\p^\epsilon(\rho_t)}^2\,dr&\leq \int_0^1 \norm{D\rho_t}_{\rho_t}^2\,dr,\\
\limsup_{\epsilon\to 0}\esssup_{t\in[0,1]}\norm{D_t \p^\epsilon \rho_t}_{\p_\epsilon\rho_t}^2&\leq \esssup_{t\in[0,1]}\norm{D\rho_t}_{\rho_t}.
\end{align*}

If additionally $(\rho_t)\in C^1([0,1];L^1(\M,\tau))$, then $(\p^\epsilon\rho_t)_t\in C^1([0,1];L^1(\M,\tau))$ and $\frac{d}{dt}\p^\epsilon\rho_t=\p^\epsilon\dot \rho_t$.
\end{lemma}
\begin{proof}
Let $\rho^\epsilon_s=\p^\epsilon\rho_s$. For $a\in \A_\mathrm{AM}$ and $s,t\in [0,1]$ we have
\begin{align*}
\abs{\tau(a(\rho_t^\epsilon-\rho_s^\epsilon))}&\leq \frac 1 \epsilon\int_0^\infty \kappa\left(\frac r\epsilon\right)\abs{\tau(P_r a(\rho_t-\rho_s))}\,dr\\
&\leq \frac 1 \epsilon\int_0^{2\epsilon}\kappa\left(\frac r \epsilon\right)\int_s^t \norm{\partial P_r a}_{\rho_u}\norm{D \rho_u}_{\rho_u}\,du\,dr\\
&\leq \int_s^t \norm{D\rho_u}_{\rho_u} \frac 1 {2\epsilon}\int_0^\epsilon e^{-Kr}\kappa\left(\frac r \epsilon\right) \norm{\partial a}_{P_r \rho_u}\,dr\,du,
\end{align*}
where we used Proposition~\ref{A_LM_regularization} and Lemma~\ref{admissibility_A_theta} in the second and $\GE(K,\infty)$ in the third inequality. Let $C(\epsilon)=\sup_{r\in [0,2\epsilon]}e^{-Kr}$ and note that $C(\epsilon)\to 1$ as $\epsilon\to 0$. An application of Jensen's inequality yields
\begin{align*}
\frac 1 \epsilon \int_0^\infty \kappa\left(\frac r \epsilon\right)\norm{\partial a}_{P_r \rho_u}\,dr\leq \left(\frac 1 \epsilon \int_0^\infty\kappa\left(\frac r\epsilon\right)\norm{\partial a}_{P_r \rho_u}^2\,dr\right)^{1/2} \leq\norm{\partial a}_{\rho_u^\epsilon}.
\end{align*}
Hence
\begin{align*}
\abs{\tau(a(\rho_t^\epsilon-\rho_s^\epsilon))}\leq C(\epsilon)\int_s^t \norm{D\rho_u}_{\rho_u}\norm{\partial a}_{\rho_u^\epsilon}\,du.
\end{align*}
Thus $(\rho^\epsilon_t)_{t\in [0,1]}$ is admissible with $\norm{D_t\rho^\epsilon_t}_{\rho^\epsilon_t}\leq C(\epsilon)\norm{D\rho_t}_{\rho_t}$ for a.e. $t\in[0,1]$. This settles both of the claimed inequalities.

Finally, the claim concerning the differentiability follows easily from an application of the dominated convergence theorem.
\end{proof}

\begin{lemma}\label{regular_curves}
Assume that the logarithmic mean is regular for $\E$ and $\E$ satisfies $\GE(K,\infty)$ for some $K\in\IR$. For every admissible curve $(\rho_s)\in C^1([0,1];L^1(\M,\tau))$ there exists a sequence $\epsilon_n\searrow 0$ such that the curves $(\rho_s^n)_{s\in[0,1]}$ defined by $\rho_s^n=(1+1/n)^{-1}\p^{\epsilon_n}(\rho_s+1/n)$ satisfy
\begin{itemize}
\item[(a)]$(\rho_s^n)_{s\in [0,1]}\in C^1([0,1];L^1(\M,\tau))$ for $n\in\IN$,
\item[(b)]$\rho_s^n\in D(\L^{(1)})$ for $s\in [0,1]$, $n\in\IN$,
\item[(c)]$\rho_s^n\geq \frac 1 {2n}$ for $s\in [0,1]$, $n\in\IN$,
\item[(d)]$\rho_s^n\to \rho_s$ in $L^1(\M,\tau)$ for $s\in[0,1]$,
\item[(e)]$\Ent(\rho_0^n)\leq \Ent(\rho_0)$, $\Ent(\rho_1^n)\leq \Ent(\rho_1)$ for $n\in\IN$,
\item[(f)]$\limsup_{n\to\infty}\int_0^1 \norm{D\rho_s^n}_{\rho_s^n}^2\,ds\leq \int_0^1\norm{D\rho_s}_{\rho_s}^2\,ds$,\\ $\limsup_{n\to\infty}\esssup_{s\in[0,1]}\norm{D\rho_s^n}_{\rho_s^n}^2\leq \esssup_{s\in[0,1]}\norm{D\rho_s}_{\rho_s}^2.$
\end{itemize}
\end{lemma}
\begin{proof}
Let $\tilde \rho^n_s=(1+1/n)^{-1}(\rho_s+1/n)$. Clearly the curves $(\tilde \rho_s^n)_s$ satisfy (a), (c) and (d). Property (e) follows from the convexity of $\Ent$. Moreover,
\begin{align*}
\tau(a(\tilde\rho_t^n-\tilde \rho_s^n))=(1+1/n)^{-1}\tau(a(\rho_t-\rho_s))
\end{align*}
implies $\norm{D\tilde \rho_s^n}_{\tilde \rho_s^n}= (1+1/n)^{-1}\norm{D\rho_s}_{\rho_s}$.

Now let $\rho_s^n=\p^{\epsilon_n}\tilde \rho_s^n$ for a strictly positive null sequence $(\epsilon_n)$. The curve $(\rho_s^n)$ satisfies (a) by Lemma~\ref{regularization_admissible}, (b) by Lemma~\ref{regularity_p_epsilon} and (c) as a direct consequence of the positivity of $(P_t)$. Moreover,
\begin{align*}
\norm{\rho_s^n-\rho_s}_1\leq \norm{\p^{\epsilon_n}\tilde\rho_s^n-\p^{\epsilon_n}\rho_s}_1+\norm{\p^{\epsilon_n}\rho_s-\rho_s}_1\leq \norm{\tilde\rho_s^n-\rho_s}_1+\norm{\p^{\epsilon_n}\rho_s-\rho_s}_1\to 0.
\end{align*}
Thus (d) is satisfied. Property (e) follows from Lemma~\ref{p_epsilon_Ent_decreasing}. Finally, by Lemma~\ref{regularization_admissible} we can achieve (f) if we choose $(\epsilon_n)$ appropriately.
\end{proof}

Lemma~\ref{regular_curves} gives entropy estimates at the endpoints of the connecting curves, but we will need entropy estimates for the entire curve. In \cite{AGS15} these are established via a logarithmic Harnack inequality. Since the proof relies on the second-order chain rule for the Laplacian, there seems to be little hope to generalize it beyond the local setting. Indeed, obtaining Harnack inequalities from Bakry--Émery-type Ricci curvature bounds has turned out to be exceptionally challenging in the non-local case. For the gradient estimate used here, there seem to be no results in this direction even in the case of finite graphs (see however \cite{CLY14,BHLLMY15,DKZ17,Mun18} for Harnack inequalities on graphs under related assumptions).

Instead we adopt a different approach. The kind of entropy estimate we need is a consequence of the EVI gradient flow characterization (see \cite[Theorem 3.1]{DS08}) and it turns out that one can run the portion of the proof needed to show only this consequence with the weaker regularity estimates already established in Lemma~\ref{regular_curves}. This is done in the next proposition.

\begin{proposition}[Entropy regularization]\label{Ent_regularization}
Assume that $\E$ satisfies $\GE(K,\infty)$. If $\rho_0,\rho_1\in \D(\M,\tau)$, then
\begin{align*}
\Ent(P_t \rho_1)\leq \Ent(\rho_0)+\frac{1}{2t}\left(\int_0^1 e^{-2Kst}\,ds\right) \W^2(\rho_0,\rho_1)
\end{align*}
for $t>0$.
\end{proposition}
\begin{proof}
Let 
\begin{align*}
C(K,t)=\int_0^1 e^{-2Kst}\,ds.
\end{align*}
We can assume that $\Ent(\rho_0)<\infty$ and $\W(\rho_0,\rho_1)<\infty$. By Lemmas~\ref{constant_speed} and \ref{W_diff_curves}, for every $\epsilon>0$ there exists an admissible curve $(\rho_s)\in C^1([0,1];L^1(\M,\tau))$ such that 
\begin{align*}
\esssup_{s\in [0,1]}\norm{D\rho_s}_{\rho_s}^2\leq \W(\rho_0,\rho_1)^2+\epsilon
\end{align*}
Let $(\rho_s^n)_{s\in[0,1]}$, $n\in\IN$, be curves defined in Lemma~\ref{regular_curves}. If we can show
\begin{align}
\Ent(P_t\rho_1^n)\leq \Ent(\rho_0^n)+\frac{C(K,t)}{2t}\esssup_{s\in[0,1]} \norm{D\rho_s^n}_{\rho_s^n}^2,\label{approx_Ent_regularization}
\end{align}
then the claim of the proposition follows by taking the limit $n\to\infty$ and then the limit $\epsilon\to 0$.

Let $(C_k)$ be an increasing sequence in $C^1((0,\infty))$ such that each $C_k$ is increasing, $1$-Lipschitz, $C_k(s)= s$ if $s\leq k-1$ and $C_k(s)= k$ if $s\geq n$. Let $f_k=C_k\circ \log$ and 
\begin{align*}
F_k\colon (0,\infty)\to \IR,\,t\mapsto \int_0^t (f_k(s)+1)\,ds.
\end{align*}
Note that $f_k(s)\nearrow \log s$ and by monotone convergence also $F_k(t)\nearrow t\log t$.

Thus, in order to prove (\ref{approx_Ent_regularization}), it suffices to show
\begin{align}
\tau(F_k(P_t\rho_1))\leq \tau(F_k(\rho_0))+\frac{C(K,t)}{2t}\int_0^1 \norm{D\rho_s^n}_{\rho_s^n}^2\,ds.\label{double_approx_Ent_regularization}
\end{align}

Let $\sigma_s=P_{st}\rho_s^n$. Since $\rho_s^n\in D(\L^{(1)})$ for all $s\in [0,1]$, the curve $(\sigma_s)$ is $L^1$-differentiable with derivative
\begin{align*}
\dot\sigma_s=P_{st}\dot\rho_{s}^n-t\L^{(1)}\sigma_s.
\end{align*}

Hence, by Lemma~\ref{diff_trace_functional},
\begin{align}
\tau(F_k(P_t \rho_1)-F_k(\rho_0))=\int_0^1 \tau(f_k(\sigma_s)(P_{st}\dot\rho_{s}^n-t\L^{(1)}\sigma_s))\,ds.\label{derivative_approx_Ent}
\end{align}
Since $(P_{st})$ maps $\M$ into $\A_\mathrm{LM}$ by Proposition~\ref{A_LM_regularization}, we have
\begin{align}
\begin{split}
\abs{\tau(f_k(\sigma_s) P_{st}\dot\rho_{s}^n)}&\leq \norm{\partial P_{st} f_k(\sigma_s)}_{\rho_{s}^n}\norm{D\rho_{s}^n}_{\rho_{s}^n}\\
&\leq e^{-Kst}\norm{\partial f_k(\sigma_s)}_{\sigma_s}\norm{D\rho_{s}^n}_{\rho_{s}^n}\\
&\leq \frac{e^{-2Kst}}{2t}\norm{D\rho_{s}^n}_{\rho_{s}^n}^2+\frac t {2}\norm{\partial f_k(\sigma_s)}_{\sigma_s}^2,\label{bound_admissible_summand}
\end{split}
\end{align}
where we used the admissibility of $(\rho_s^n)_s$ and Lemma~\ref{admissibility_A_theta} for the first inequality, $\GE(K,\infty)$ for the second and Young's inequality for the third.

We are now going to estimate the second summand. Note that by definition $0\leq f_k'(s)\leq 1/s$ and for each fixed $k\in\IN$ there exists $l\in\IN$ such that $f_k(s)=f_k(s\wedge l)$ for all $s>0$.

Since $\sigma_s\in D(\L^{(1)})$, we have $\sigma_s\wedge m\in D(\E)$ for all $m\in\IN$ by Lemma~\ref{contractions_L_p_generator}. Thus, writing $e$ for the joint spectral measure of $L(\sigma_s\wedge m)$ and $R(\sigma_s\wedge m)$,
\begin{align*}
\norm{\partial f_k(\sigma_s)}_{\sigma_s\wedge m}^2&=\norm{\partial f_k(\sigma_s\wedge m)}_{\sigma_s\wedge m}^2\\
&=\int_{(0,\infty)^2}\tilde f_k(s,t)^2\widehat{\log}(s,t)\,d\langle e(s,t)\partial (\sigma_s\wedge m),\partial(\sigma_s\wedge m)\rangle_\H\\
&\leq \langle \tilde f_k(L(\sigma_s\wedge m),R(\sigma_s\wedge m))\partial (\sigma_s\wedge m),\partial (\sigma_s\wedge m)\rangle_\H\\
&=\langle \partial f_k(\sigma_s),\partial (\sigma_s\wedge m)\rangle_\H\\
&\leq \tau(f_k(\sigma_s)\L^{(1)}\sigma_s),
\end{align*}
where we used $\tilde f_k\leq \widetilde{\log}$ and Lemma~\ref{contractions_L_p_generator}. Now we can let $m$ go to infinity to obtain
\begin{align*}
\norm{\partial f_k(\sigma_s)}_{\sigma_s}^2\leq \tau(f_k(\sigma_s)\L^{(1)}\sigma_s).
\end{align*}
If we plug this inequality into (\ref{bound_admissible_summand}), we get
\begin{align*}
\abs{\tau(f_k(\sigma_s) P_{st}\dot\rho_{s}^n)}\leq \frac{e^{-2Kst}}{2t}\norm{D\rho_{s}^n}_{\rho_{s}^n}^2+\frac t 2 \tau(f_k(\sigma_s)\L^{(1)}\sigma_s),
\end{align*}
which we can then apply to (\ref{derivative_approx_Ent}) to obtain
\begin{align}
\begin{split}
\tau(F_k(P_t\rho_1)-F_k(\rho_0))&\leq \int_0^1 \frac{e^{-2Kst}}{2t}\norm{D\rho_s^n}_{\rho_s^n}^2\,ds\\
&\leq \frac{C(K,t)}{2t}\esssup_{s\in[0,1]} \norm{D\rho_s^n}_{\rho_s^n}^2.
\label{suboptimal_entropy_bound}
\end{split}
\end{align}
This settles (\ref{double_approx_Ent_regularization}).
\end{proof}

\begin{corollary}\label{regular_curves_improv}
Assume that $\E$ satisfies $\GE(K,\infty)$ for some $K\in\IR$. For every admissible curve $(\rho_s)_{s\in[0,1]}\in C^1([0,1];L^1(\M,\tau))$ with $\Ent(\rho_0)<\infty$ the curves $(\rho_s^n)_{s\in[0,1]}$ defined in Lemma~\ref{regular_curves} satisfy
\begin{itemize}
\item[(g)] $\sup_{s\in [0,1]} \Ent(\rho_s^n)<\infty$ for $n\in \IN$,
\item[(h)] $\sup_{t\geq 0}\sup_{s\in[0,1]}\I(P_t\rho_s^n)<\infty$ for $n\in\IN$.
\end{itemize}
\end{corollary}
\begin{proof}
As in the proof of Lemma~\ref{regular_curves} let $\tilde \rho_s^n=(1+1/n)^{-1}(\rho_s+1/n)$ and $\rho_s^n=\p^{\epsilon_n}\tilde\rho_s^n$ for a suitably chosen strictly positive null sequence $(\epsilon_n)$. By Lemma~\ref{p_epsilon_Ent_decreasing} we have $\Ent(\rho_s^n)\leq \Ent(P_{\epsilon_n}\tilde \rho_s^n)$. We can apply Theorem~\ref{Ent_regularization} to the right-hand side to get
\begin{align*}
\Ent(P_{\epsilon_n}\tilde\rho_s^n)&\leq \Ent(\tilde\rho_0^n)+C(K,\epsilon_n)\W(\tilde\rho_0^n,\tilde\rho_s^n)^2\\
&\leq \frac{n}{n+1}\left(\Ent(\rho_0)+C(K,\epsilon_n)\int_0^1 \norm{D\rho_r}_{\rho_r}^2\,dr\right).
\end{align*}
The latter is clearly bounded independently of $s\in[0,1]$. This proves (g).

To establish (h), we use Jensen's inequality (which is applicable according to Lemma~\ref{Fisher_info_lsc} and Proposition~\ref{Fisher_info_convex}) to see that
\begin{align*}
\I(P_t \rho_s^n)\leq \frac 1 {\epsilon_n} \int_{\epsilon_n}^\infty\kappa\left(\frac r{\epsilon_n}\right)\I(P_{t+r}\tilde\rho_s^n)\,dr\leq \frac{\norm{\kappa}_\infty}{\epsilon_n}\int_{\epsilon_n}^\infty\I(P_r P_t\tilde\rho_s^n)\,dr.
\end{align*}
By Corollary~\ref{heat_flow_L1}, we have
\begin{align*}
\int_{\epsilon_n}^\infty\I(P_r P_t\tilde\rho_s^n)\,dr\leq \Ent(P_{t+\epsilon_n}\tilde\rho_s^n)\leq \Ent(P_{\epsilon_n}\tilde \rho_s^n).
\end{align*}
We have already seen in the first part that the right-hand side is bounded independently of $s\in[0,1]$.
\end{proof}

\subsection{Proof of the gradient flow characterization}\label{proof_main}

With preparations done in the previous subsection, we can now launch the proof of the gradient flow characterization. We start with some basic facts from the theory of gradient flows in metric spaces.

Recall that we mean by an extended metric space a pair $(X,d)$ that satisfies all axioms of a metric space, except that $d$ may take the value $\infty$. Let $S\colon X\lra(-\infty,\infty]$ be a proper lower semicontinuous functional and let $D(S)$ denote its proper domain, that is,
\begin{align*}
D(S)=\{x\in X\mid S(x)<\infty\}.
\end{align*}
The \emph{descending} and \emph{ascending slope} of $S$ are defined by
\begin{align*}
\abs{D^+S}(x)&=\limsup_{y\to x}\frac{(S(y)-S(x))_-}{d(x,y)},\\
\abs{D^-S}(x)&=\limsup_{y\to x}\frac{(S(y)-S(x))_+}{d(x,y)}
\end{align*}
if $x\in D(S)$ is not isolated. For isolated points $x\in D(S)$ one sets 
\begin{equation*}
\abs{D^+S}(x)=\abs{D^-S}(x)=0,
\end{equation*}
and furthermore $\abs{D^+S}=\abs{D^-S}=\infty$ on $X\setminus D(S)$.

As a further piece of notation we need the \emph{upper right derivative} (or upper Dini derivative) $\frac{d^+}{dt}$ of a function $f$ on a right-open interval $I$, which is defined by
\begin{align*}
\frac{d^+}{dt}f(t)=\limsup_{h\searrow 0}\frac{f(t+h)-f(t)}{h}
\end{align*}
for $t\in I$.

\begin{definition}[$\mathrm{EDE}$ and $\EVI$ gradient flow curves]
Let $(X,d)$ be an extended metric space and $S\colon X\lra (-\infty,\infty]$ a proper lower semicontinuous functional.
A locally absolutely continuous curve $(\gamma_t)_{t\geq 0}$ in $X$ is called \emph{EDE gradient flow of $S$} if it satisfies the \emph{energy dissipation equality}
\begin{equation*}
S(\gamma_0)=S(\gamma_t)+\frac 1 2\int_0^t\abs{\dot\gamma_s}^2\,ds+\frac 1 2\int_0^t \abs{D^-S}^2(\gamma_s)\,ds\tag{$\mathrm{EDE}$}
\end{equation*}
for all $t\geq 0$.

Let $K\in\IR$. The curve $\gamma$  is called \emph{$\EVI_K$ gradient flow curve} of $S$ if it satisfies the \emph{evolution variational inequality}
\begin{align*}
\frac 1 2 \frac{d^+}{dt}d(\gamma_t,x)^2+\frac{K}{2}d(\gamma_t,x)^2+S(\gamma_t)\leq S(x)\tag{$\EVI_K$}\label{EVI}
\end{align*}
for all $t\geq 0$ and $x\in X$ with $d(x,\gamma_0)<\infty$.

A semigroup of continuous maps $T_t \colon D(S)\lra D(S)$, $t\geq 0$, is called \emph{$\EVI_K$ gradient flow} of $S$ if
\begin{itemize}
\item[(F1)] $d(T_t x,x)\to 0$ as $t\to 0$ for all $x\in X$,
\item[(F2)] $S$ is decreasing along $(T_t x)_{t\geq 0}$ for all $t\geq 0$,
\item[(F3)] $(T_t x)_{t\geq 0}$ is an $\EVI_K$ gradient flow for all $x\in X$.
\end{itemize}
\end{definition}

\begin{remark}If $\IR^d$ is endowed with the Euclidean metric and $S\in C^1(\IR^d)$, then a $C^1$-curve $\gamma$ is an EDE gradient flow curve of $S$ if and only if it satisfies the classical gradient flow equality
\begin{equation*}
\dot\gamma_t=-\nabla S(\gamma_t).
\end{equation*}
If moreover $S-\frac{K}{2}\abs{\cdot}^2$ is convex, then the notion of EDE gradient flow curve, $\EVI_K$  gradient flow curve and classical gradient flow curve all coincide for $C^1$-curves. Conversely, if $S$ admits an $\EVI_K$ gradient flow, then $S-\frac{K}2\abs{\cdot}^2$ is convex.
\end{remark}

While the existence of $\EVI_K$ gradient flow curves for a given functional is not guaranteed, the uniqueness is a consequence of the defining property (see e.g. \cite[Proposition 3.1]{DS08}):

\begin{lemma}\label{EVI_metric_contraction}
Let $(X,d)$ be an extended metric space and $S\colon X\lra(-\infty,\infty]$ a proper lower semicontinuous functional. If $\gamma$, $\tilde\gamma$ are $\EVI_K$ gradient flow curves of $S$ starting in $\gamma_0$, $\tilde\gamma_0$ respectively, then
\begin{align*}
d(\gamma_t,\tilde \gamma_t)\leq e^{-Kt}d(\gamma_0,\tilde\gamma_0)
\end{align*}
for all $t\geq 0$.

In particular, there is at most one $\EVI_K$ gradient flow curve with a given starting point.
\end{lemma}

\begin{theorem}\label{EVI_gradient_flow}
Assume that $\tau$ is finite, $L^1(\M,\tau)$ is separable and the logarithmic mean is regular with respect to $\E$. If $\E$ satisfies $\GE(K,\infty)$, then $(P_t)$ is an $\EVI_K$ gradient flow of $\Ent$.
\end{theorem}
\begin{proof}
The continuity of $(P_t)$ with respect to $\W$ is a consequence of Theorem~\ref{W_contraction}. Moreover, it was proven in Corollary~\ref{heat_flow_L1} that $(P_t)_{t\geq 0}$ is strongly continuous with respect to $\W$ and that $\Ent$ is decreasing along $(P_t \rho)$. It remains to show that 
\begin{align}
\frac 1 2 \frac{d^+}{dt}\W(P_t \rho_1,\rho_0)^2+\frac K 2\W(\rho_0,\rho_1)^2+\Ent(P_t \rho_1)\leq\Ent(\rho_0)\label{EVI_Ent}
\end{align}
for $\rho_0,\rho_1\in D(\Ent)$ with $\W(\rho_0,\rho_1)<\infty$ and $t\geq 0$. Since $(P_t)$ is a semigroup, it suffices to check (\ref{EVI_Ent}) at $t=0$.

If we can prove
\begin{align}
\W(P_t\rho_1,\rho_0)^2\leq \left(\int_0^1 e^{-2Kst}\right)\W(\rho_0,\rho_1)^2-2t(\Ent(P_t\rho_1)-\Ent(\rho_0)),\label{EVI_Ent_integrated}
\end{align}
then
\begin{align*}
\frac{\W(P_t\rho_1)^2-\W(\rho_0,\rho_1)^2}{2t}\leq \Ent(\rho_0)-\Ent(P_t \rho_1)+\frac {\int_0^1 e^{-2Kst}\,ds-1}{2t}\W(\rho_0,\rho_1)^2,
\end{align*}
from which (\ref{EVI_Ent}) at $t=0$ follows in the limit $t\to 0$.

In order to prove (\ref{EVI_Ent_integrated}), let $\epsilon>0$ and let $(\rho_s)\in C^1([0,1];L^1(\M,\tau)$ be an admissible curve with
\begin{align*}
\esssup_{s\in[0,1]}\norm{D\rho_s}_{\rho_s}\leq \W(\rho_0,\rho_1)+\epsilon.
\end{align*}
Let $(\rho_s^n)_{s\in[0,1]}$ be the curves defined in Lemma~\ref{regular_curves} and let $\sigma_{s,t}=P_{st}\rho_s^n$. Since $\rho_s^n\in D(\L^{(1)})$, for each $t\geq 0$ the curve $(\sigma_{s,t})_{s\in[0,1]}$ is $L^1$-differentiable with derivative
\begin{align}
\frac{d}{ds}\sigma_{s,t}=P_{st}\dot\rho_s^n-t\L^{(1)}\sigma_{s,t}.\label{derivative_sigma_st}
\end{align}
We will now show that the curve $(\sigma_{s,t})_{s\in[0,1]}$ is admissible by evaluating both summands separately.

Since $\E$ satisfies $\GE(K,\infty)$, the logarithmic mean is regular for $\E$ and $(\rho_s^n)$ is admissible, Proposition~\ref{A_LM_regularization} and Lemma~\ref{admissibility_A_theta} imply
\begin{align*}
\abs{\tau(a P_{st}\dot\rho_s^n)}=\abs{\tau(\dot\rho_s^n P_{st} a)}\leq\norm{ \partial P_{st}a}_{\rho_s^n}\norm{D\rho_s^n}_{\rho_s^n}\leq e^{-Kst}\norm{\partial a}_{\sigma_{s,t}}\norm{D\rho_s^n}_{\rho_s^n}
\end{align*}
for all $a\in D(\E)\cap\M$ with $\partial a\in D(\hat\sigma_{s,t}^{1/2})$. Thus there exists a unique $\xi_s$ in $\tilde\H_{\sigma_{s,t}}$  such that 
\begin{align}
\tau(a P_{st}\dot\rho_s^n)=\langle \partial a,\xi_s\rangle_{\sigma_{s,t}}\label{vertical_part_admissible}
\end{align}
for all $a\in D(\E)\cap \M$ with $\partial a\in D(\hat\sigma_{s,t}^{1/2})$. Moreover, $\norm{\xi_s}_{\sigma_s}\leq e^{-Kst}\norm{D\rho_s^n}_{\rho_s^n}$.

Some more work is necessary for the second summand in (\ref{derivative_sigma_st}). If $\sigma_{s,t}\in D(\L^{(2)})$, then
\begin{align*}
\tau(a\L^{(1)}\sigma_{s,t})=\langle\partial a,\partial\sigma_{s,t}\rangle_\H=\langle \partial a,\partial\log\sigma_{s,t}\rangle_{\sigma_{s,t}}.
\end{align*}
To show the equality of the left- and right-hand side in the general case (the middle is of course not well-defined), we argue by approximation.

Let $\tilde \rho_s^n=(1+1/n)^{-1}(\rho_s+1/n)$ and recall that $\rho_s^n=\p^{\epsilon_n}\tilde\rho_s^n$. Moreover, let $\sigma_{s,t}^N=P_{st}\p^{\epsilon_n}(\tilde\rho_s^n\wedge N)$. If $a\in D(\E)\cap \M$ with $\partial a\in D(\hat\sigma_{s,t}^{1/2})$, then
\begin{align}
\label{horizontal_part_approx_derivative}
\begin{split}
\tau(a \L^{(1)}\sigma_{s,t})&=\lim_{N\to\infty}\tau((\L^{(2)}\p^{\epsilon_n} a)P_{st}(\tilde\rho_s^n\wedge N))\\
&=\lim_{N\to\infty}\langle\partial \p^{\epsilon_n}a,\partial P_{st}(\tilde\rho_s^n\wedge N)\rangle_\H\\
&=\lim_{N\to\infty}\langle \partial a,\partial \log \sigma_{s,t}^N\rangle_{\sigma_{s,t}^N}.
\end{split}
\end{align}
Since $\sigma_{s,t}^N\leq \sigma_{s,t}$ and $\sigma_{s,t}^N\to \sigma_{s,t}$ in $L^1(\M,\tau)$ as $N\to\infty$, we have 
\begin{equation}
\label{strong_conv_sigma}
\widehat{\sigma_{s,t}^N}^{1/2}\partial a\to\hat\sigma_{s,t}^{1/2}\partial a
\end{equation}
strongly in $\H$ as $N\to \infty$ by Lemma~\ref{monotone_conv_rho_norm}.

We will now show that $\widehat{\sigma^N_{s,t}}^{1/2}\partial\log\sigma_{s,t}^N\to \hat\sigma_{s,t}^{1/2}\partial\log\sigma_{s,t}$ weakly in $\H$. By Jensen's inequality and Corollary~\ref{heat_flow_L1},
\begin{align}\label{Fisher_bound_cut-offs}
\begin{split}
\I\left(\frac{\sigma_{s,t}^N}{\tau(\tilde\rho_s^n\wedge N)}\right)&\leq \frac{\norm{\kappa}_\infty}{\epsilon_n}\int_{\epsilon_n}^\infty \I\left(P_r P_{st}\frac{\tilde\rho_s^n\wedge N}{\tau(\tilde\rho_s^n\wedge N)}\right))\,dr\\
&\leq \frac{\norm{\kappa}_\infty}{\epsilon_n}\Ent\left(P_{\epsilon_n}\frac{\tilde\rho_s^n\wedge N}{\tau(\tilde\rho_s^n\wedge N)}\right)\\
&\leq \frac{\norm{\kappa}_\infty}{\epsilon_n}\tau((P_{\epsilon_n}(\tilde\rho_s^n\wedge N)\log P_{\epsilon_n}(\tilde\rho_s^n\wedge N))_+)\\
&\quad-\frac{\norm{\kappa}_\infty}{\epsilon_n}\log\tau(\tilde\rho_s^n\wedge N).
\end{split}
\end{align}
For the first summand observe that $t\mapsto (t\log t)_+$ is increasing, which implies by \cite[Lemma 4]{BK90} that
\begin{align*}
\tau((P_{\epsilon_n}(\tilde\rho_s^n\wedge N)\log P_{\epsilon_n}(\tilde\rho_s^n\wedge N))_+)\leq \tau((P_{\epsilon_n}\tilde\rho_s^n\log P_{\epsilon_n}\tilde\rho_s^n)_+).
\end{align*}
The right-hand side is finite by Proposition~\ref{Ent_regularization}. Since $\tau(\tilde\rho_s^n)\to 1$, we infer from (\ref{Fisher_bound_cut-offs}) that $\sup_N\I(\sigma_{s,t}^N)<\infty$. The lower bound $\sigma_{s,t}^N\geq 1/n$ then implies 
\begin{equation*}
\sup_N\E(\log\sigma_{s,t}^N)\leq n\sup_N \I(\sigma_{s,t}^N)<\infty.
\end{equation*}

From $\sigma_{s,t}^N\to\sigma_{s,t}$ in $L^1(\M,\tau)$ we infer $\log\sigma_{s,t}^N\to\log\sigma_{s,t}$ in $L^2(\M,\tau)$ by \cite[Theorem 3.2]{Tik87}. Together with the bound on the energy this implies $\log\sigma_{s,t}^N\to\log \sigma_{s,t}$ weakly in $(D(\E),\langle\cdot,\cdot\rangle_\E)$.

If $\xi\in D(\hat\sigma_{s,t}^{1/2})$, then
\begin{align*}
\langle\widehat{\sigma_{s,t}^N}^{1/2}\partial\log\sigma_{s,t}^N,\xi\rangle_\H=\langle \partial \log\sigma_{s,t}^N,\widehat{\sigma^N_{s,t}}^{1/2}\xi\rangle_\H\to \langle\partial \log \sigma_{s,t},\hat\sigma_{s,t}^{1/2}\xi\rangle_\H
\end{align*}
as $N\to\infty$ by Lemma~\ref{monotone_conv_rho_norm}. Since $D(\hat\sigma_{s,t}^{1/2})$ is dense in $\H$ and $\sup_N\I(\sigma_{s,t}^N)<\infty$, this implies
\begin{align*}
\widehat{\sigma^N_{s,t}}^{1/2}\partial \log \sigma_{s,t}^N\to\hat\sigma_{s,t}^{1/2}\partial\log\sigma_{s,t}
\end{align*}
weakly in $\H$ as $N\to\infty$.

If we combine this convergence with (\ref{strong_conv_sigma}), then we can deduce from (\ref{horizontal_part_approx_derivative}) that
\begin{align}
\label{horizontal_part_admissible}
\tau(a\L^{(1)}\sigma_{s,t})=\langle \partial a,\partial\log \sigma_{s,t}\rangle_{\sigma_{s,t}}
\end{align}
for all $a\in D(\E)\cap\M$ with $\partial a\in D(\hat\sigma_{s,t}^{1/2})$.

Let $\eta_{s,t}=\partial\log\sigma_{s,t}$. If we combine the results (\ref{vertical_part_admissible}) and (\ref{horizontal_part_admissible}), we see that $(\sigma_{s,t})_{s\in [0,1]}$ is admissible and
\begin{align*}
\norm{D_s\sigma_{s,t}}_{\sigma_{s,t}}^2\leq \norm{\xi_{s,t}-t\eta_{s,t}}_{\sigma_{s,t}}^2.
\end{align*}
Thus
\begin{align}
\begin{split}
\int_0^1 \norm{D_s\sigma_{s,t}}_{\sigma_{s,t}}^2\,ds&\leq\int_0^1 \norm{\xi_{s,t}-t\eta_{s,t}}_{\sigma_{s,t}}^2\,ds\\
&=\int_0^1 (\norm{\xi_{s,t}}_{\sigma_{s,t}}^2-2t\langle\eta_{s,t},\xi_{s,t}-t\eta_{s,t}\rangle_{\sigma_{s,t}}-t^2\norm{\eta_{s,t}}_{\sigma_{s,t}}^2)\,ds\\
&\leq \int_0^1 (\norm{\xi_{s,t}}_{\sigma_{s,t}}^2-2t\langle\eta_{s,t},\xi_{s,t}-t\eta_{s,t}\rangle_{\sigma_{s,t}})\,ds\\
&\leq \left(\int_0^1 e^{-2Kst}\,ds\right)\esssup_{s\in [0,1]}\norm{D\rho_s^n}_{\rho_s^n}^2\\
&\qquad-2t\int_0^1 \langle \eta_{s,t},\xi_{s,t}-t\eta_{s,t}\rangle_{\sigma_{s,t}}\,ds.
\label{bound_speed_P_st}
\end{split}
\end{align}
Let $(C_k)$ be an increasing sequence in $C^1((0,\infty))$ such that each $C_k$ is increasing, $1$-Lipschitz, $C_k(s)= s$ if $s\leq k-1$ and $C_k(s)= k$ if $s\geq n$. Let $f_k=C_k\circ \log$ and 
\begin{align*}
F_k\colon (0,\infty)\to \IR,\,t\mapsto \int_0^t (f_k(s)+1)\,ds.
\end{align*}
As in the proof of Proposition~\ref{Ent_regularization} one can show 
\begin{align*}
\tau(F_k(P_t\rho_1^n)-F_k(\rho_0^n))&=\int_0^1 \tau(f_k(\sigma_{s,t})(P_{s,t}\rho_s^n-t\L^{(1)}\sigma_{s,t}))\,ds.
\end{align*}
Since $f_k(\sigma_{s,t})\in D(\E)\cap \M$ by Lemma~\ref{contractions_L_p_generator} and $\norm{\partial f_k(\sigma_{s,t})}_{\sigma_{s,t}}^2\leq \I(\sigma_{s,t})$, we can apply (\ref{vertical_part_admissible}) and (\ref{horizontal_part_admissible}) to get
\begin{align}
\tau(F_k(P_t\rho_1^n)-F_k(\rho_0^n))=\int_0^1\langle\partial f_k(\sigma_{s,t}),\xi_{s,t}-t\eta_{s,t}\rangle_{\sigma_{s,t}}\,ds
\label{eq_approx_Ent}
\end{align}
Since $f_k(\sigma_{s,t})\to \log\sigma_{s,t}$ in $L^2(\M,\tau)$ as $k\to\infty$ and $\norm{\partial f_k (\sigma_{s,t})}_{\sigma_{s,t}}^2\leq \I(\sigma_{s,t})$, we have $\partial f_k(\sigma_{s,t})\to \partial \log\sigma_{s,t}=\eta_{s,t}$ weakly in $\tilde\H_{\sigma_{s,t}}$ and the integrand on the right-hand side of (\ref{eq_approx_Ent}) is pointwise bounded by $\I(\sigma_{s,t})^{1/2}\norm{\xi_{s,t}-t\eta_{s,t}}_{\sigma_{s,t}}$.

By the dominated convergence theorem,
\begin{align*}
\lim_{k\to\infty}\int_0^1 \langle\partial f_k(\sigma_{s,t}),\xi_{s,t}-t\eta_{s,t}\rangle_{\sigma_{s,t}}\,ds=\int_0^1 \langle\eta_{s,t},\xi_{s,t}-t\eta_{s,t}\rangle_{\sigma_{s,t}}\,ds.
\end{align*}
On the other hand, the left-hand side of (\ref{eq_approx_Ent}) converges to $\Ent(P_t\rho_1)-\Ent(\rho_0)$ by the monotone convergence theorem.

Then (\ref{bound_speed_P_st}) becomes
\begin{align}
\begin{split}
\int_0^1 \norm{D_s\sigma_{s,t}}_{\sigma_{s,t}}^2\,ds&\leq\left(\int_0^1 e^{-2Kst}\,ds\right)\esssup_{s\in [0,1]}\norm{D\rho_s^n}_{\rho_s^n}^2\\
&\quad-2t(\Ent(P_t \rho_1^n)-\Ent(\rho_0^n)).
\label{bound_speed_P_st_Ent}
\end{split}
\end{align}
Since $\sigma_{s,t}= P_{st}\rho_s^n\to P_{st}\rho_s$ in $L^1(\M,\tau)$ as $n\to\infty$, Theorem~\ref{energy_lsc} can be used to see that the curve $(P_{st}\rho_s)_s$ is admissible and
\begin{align*}
\W(\rho_0,P_t \rho_1)^2\leq\int_0^1 \norm{D_s(P_{st}\rho_s)}_{P_{st\rho_s}}^2\leq \liminf_{n\to\infty}\int_0^1 \norm{D_s(P_{st}\rho_s^n)}_{P_{st}\rho_s^n}^2\,ds.
\end{align*}
By Lemma~\ref{regular_curves} we have
\begin{align*}
\limsup_{n\to\infty}\esssup_{s\in[0,1]}\norm{D\rho_s^n}_{\rho_s^n}^2\leq \esssup_{s\in [0,1]}\norm{D\rho_s}_{\rho_s}^2\leq (\W(\rho_0,\rho_1)+\epsilon)^2.
\end{align*}
Furthermore, using the lower semicontinuity of the entropy and Lemma~\ref{regular_curves}, we obtain $\Ent(P_t\rho_1^n)\to \Ent(P_t\rho_1)$, $\Ent(\rho_0^n)\to\Ent(\rho_0)$.

These inequalities allow to pass to the limit $n\to\infty$ in (\ref{bound_speed_P_st_Ent}) to get
\begin{align*}
\W(\rho_0,P_t\rho_1)^2\leq \left(\int_0^1 e^{-2Kst}\,ds\right)(\W(\rho_0,\rho_1)+\epsilon)^2-2t(\Ent(P_t\rho_1)-\Ent(\rho_0)),
\end{align*}
which yields (\ref{EVI_Ent_integrated}) as $\epsilon\searrow 0$.
\end{proof}

As a consequence of the uniqueness of $\EVI_K$ gradient flow curves (Lemma \ref{EVI_metric_contraction}) we note the following corollary.

\begin{corollary}
A curve $(\rho_t)_{t\geq 0}$ with $\rho_0\in D(\Ent)$ is an $\EVI_K$ gradient flow curve of $\Ent$ if and only if $\rho_t=P_t \rho_0$ for all $t\geq 0$.
\end{corollary}

\section{Geodesic convexity}\label{geodesic_convex}
In this section we will study an important consequence of the gradient flow characterization, namely the (semi-) convexity of the entropy along geodesics in $(\D(\M,\tau),\W)$. This property served as definition for synthetic Ricci curvature bounds by Lott--Villani \cite{LV09} and Sturm \cite{Stu06a,Stu06b} and could therefore also be an entrance gate to the study of Ricci curvature in noncommutative geometry.

As discussed in a previous section, even the existence of $\W$-geodesics is not clear in general. The situation is much better if $(P_t)$ satisfies the gradient estimate $\GE(K,\infty)$ and we restrict our attention to the domain of the entropy (Theorem \ref{D(Ent)_geodesic}). This is due to two ingredients, which we will study next: First, the sublevel sets of the entropy are compact in the weak $L^1$-topology (Lemma \ref{sublevel_ent_cpt}). Together with the lower semicontinuity of the the action functional with respect to pointwise weak convergence in $L^1$, this can be employed for the standard existence proof of minimizers for a variational problem, provided one can always find a minimizing sequence with uniformly bounded entropy. As we will see, the latter is a essentially a consequence of the evolution variational inequality (Proposition \ref{connecting_curves_bounded_entropy}).

Once the existence of geodesics is proven, the semi-convexity of the entropy along them follows from abstract results on gradient flows (Theorem \ref{Ent_semi-convex}). Finally, we summarize the relations between the gradient estimate $\GE(K,\infty)$, the evolution variational inequality $\EVI_K$ and $K$-convexity of the entropy in Theorem \ref{main_theorem_summary}.

As usual, let $(\M,\tau)$ be a tracial von Neumann algebra, $\E$ a quantum Dirichlet form on $L^2(\M,\tau)$ such that $\tau$ is energy dominant, and $(\partial,\H,L,R,J)$ the associated first order differential calculus. We further assume that $\tau$ is a state, $L^1(\M,\tau)$ is separable, $\theta$ is the logarithmic mean and that $\theta$ is a regular mean for $\E$.

\begin{lemma}\label{sublevel_ent_cpt}
If $F\colon [0,\infty)\lra \IR$ is a lower semicontinuous function such that $f(t)/t\to\infty$ as $t\to\infty$, then the sublevel sets of
\begin{align*}
F\colon \D(\M,\tau)\lra (-\infty,\infty],\,F(\rho)=\tau(f(\rho))
\end{align*}
are compact in the weak $L^1$-topology.
\end{lemma}
\begin{proof}
The sublevel sets are closed since $F$ is lower semicontinuous by Lemma \ref{entropy_lsc}, so it suffices to show that they are relatively weakly compact. We assume that $f\geq 0$; otherwise one can replace it by $f_+$ and use that $f_+\leq f-\inf f$.

The proof of relative weak compactness is a noncommutative version of the Vallée Poussin theorem (see \cite[Theorem 4.5.9]{Bog07}). By the noncommutative version of the Dunford-Pettis theorem (\cite[Theorem III.5.4]{Tak02}), it suffices to show that $\tau(p_n \rho)\to 0$ uniformly in $\rho\in\mathcal{F}$ whenever $(p_n)$ is a decreasing sequence of projections in $\M$ such that $p_n\searrow 0$.

Let $C=\sup_{\rho\in \mathcal{F}}\tau(f(\rho))$. For $\epsilon>0$ let $M=\frac {2C}\epsilon$. By assumption there exists $T>0$ such that $f(t)\geq Mt$ for all $t\geq T$. Moreover, since $\tau$ is normal, we can choose $N\in\mathbb{N}$ such that $\tau(p_n)<\frac\epsilon {2T}$ for $n\geq N$.

It follows that
\begin{align*}
\tau(p_n \rho)=\tau(p_n \rho(\1_{[0,T)}(\rho)+\1_{[T,\infty)}(\rho)))\leq T\tau(p_n)+\frac{\epsilon}{2C}\tau(p_n f(\rho))<\epsilon
\end{align*}
for all $\rho\in\mathcal{F}$ and $n\geq N$.
\end{proof}

\begin{remark}
By the Eberlein-\v{S}mulian theorem (\cite[Theorem V.6.1]{DS88}), weak compactness and weak sequential compactness are equivalent for weakly closed subsets of a Banach space.
\end{remark}

\begin{proposition}\label{connecting_curves_bounded_entropy}
For all $K,\alpha, D>0$ there exists a constant $C(K,\alpha,D)>0$ such that the following folds:

If the form $\E$ satisfies $\GE(K,\infty)$, $\rho_0,\rho_1\in D(\Ent)$ with $\Ent(\rho_0),\Ent(\rho_1)\leq \alpha$, and $\W(\rho_0,\rho_1)\leq D$, then there is a sequence of admissible curves $(\rho^n_t)_{t\in[0,1]}$ connecting $\rho_0$ and $\rho_1$ such that
\begin{align*}
\sup_{n\in\IN}\sup_{t\in [0,1]}\Ent(\rho^n_t)\leq C(K,\alpha,D)
\end{align*}
and
\begin{align*}
\int_0^1 \norm{D\rho^n_t}_{\rho^n_t}^2\,dt\to \W(\rho_0,\rho_1)^2.
\end{align*}
\end{proposition}
\begin{proof}
For $n\in\IN$ let $(\sigma^n_t)_{t\in[0,1]}$ be an admissible $L_n$-Lipschitz curve in $(\D(\M,\tau),\W)$ connecting $\rho_0$ and $\rho_1$ such that $L_n^2\leq \W(\rho_0,\rho_1)^2+\frac 1 {n^2}$.

Let $\tilde \sigma_t^n=P_{1/n}\sigma^n_t$. Since $(P_t)$ is an $\EVI_K$ gradient flow of $\Ent$ by Theorem \ref{EVI_gradient_flow}, Theorem 3.2 of \cite{DS08} asserts
\begin{align*}
\Ent(\tilde \sigma^n_t)\leq (1-t)\Ent(\rho_0)+t\Ent(\rho_1)-\frac K 2t(1-t)\W(\rho_0,\rho_1)^2+\frac 1{2n^2 I_K(1/n)},
\end{align*}
where $I_K(t)=\int_0^t e^{Kr}\,dr$.

As $n^2 I_K(1/n)\to \infty$ as $n\to\infty$, the supremum
\begin{align*}
c(K)=\sup_{n\in\IN}\frac 1{2n^2 I_K(1/n)}
\end{align*}
is finite. Thus
\begin{align*}
\sup_{n\in\IN}\sup_{t\in[0,1]}\Ent(\tilde \sigma^n_t)\leq \alpha+\frac{\abs{K}}2 D^2+c(K).
\end{align*}
Furthermore, Theorem \ref{W_contraction} implies $\norm{D\tilde\sigma^n_t}_{\tilde\sigma^n_t}\leq e^{-K/n}\norm{D\sigma^n_t}_{\sigma^n_t}$ for a.e. $t\in [0,1]$.

Moreover, $\Ent(P_s\rho_0)\leq \Ent(\rho_0)\leq\alpha$ and $(P_s \rho_0)_{s\geq 0}$ is admissible by Corollary \ref{heat_flow_L1}, hence
\begin{align*}
\int_0^{1/n}\norm{D_s P_s \rho_0}_{P_s \rho_0}^2\,ds\to 0
\end{align*}
as $n\to \infty$. Of course, the same holds for $\rho_0$ replaced by $\rho_1$.

Hence one can concatenate the curves $(P_t\rho_0)_{t\in [0,1/n]}$, $(\tilde \sigma^n_t)_{t\in [0,1]}$ and $(P_{\frac 1 n-t}\rho_1)_{t\in [0,1/n]}$ to get a curve $(\rho^n_t)$ with the desired properties.
\end{proof}

\begin{definition}
We say that the entropy has \emph{regular sublevel sets} if every curve $(\rho_t)\in \AC^2_\loc(I;(\D(\M,\tau),\W))$ with uniformly bounded entropy is admissible and $\norm{D\rho_t}_{\rho_t}=\abs{\dot\rho_t}_\W$ for a.e. $t\in I$.
\end{definition}

\begin{proposition}\label{prop_regular_sublevels}
If $\E$ satisfies $\GE(K,\infty)$, then the entropy has regular sublevel sets.
\end{proposition}
\begin{proof}

First assume that $(\rho_t)\in \AC^2([0,1];(\D(\M,\tau),\W))$. Since $(\rho_t)$ is continuous on a compact interval, it is uniformly continuous. Thus, for every $\epsilon>0$ there exists a partition $0=t_0<t_1<\dots<t_n=1$ of $[0,1]$ such that $\W(\rho_{t_{k-1}},\rho_t)<\epsilon$ for all $t\in[t_{k-1},t_k]$, $1\leq k\leq n$.

For $k\in\{1,\dots,n\}$ let $\sigma^{k,\epsilon}\colon [t_{k-1},t_k]\lra \D(\M,\tau)$ be an admissible curve with $\sigma^{k,\epsilon}_{t_{k-1}}=\rho_{t_{k-1}}$, $\sigma^{k,\epsilon}_{t_k}=\rho_{t_k}$ and
\begin{align*}
\int_{t_{k-1}}^{t_k}\norm{D\sigma^{k,\epsilon}_r}_{\sigma^{k,\epsilon}_r}^2\,dr\leq\frac{\W(\rho_{t_{k-1}},\rho_{t_k})^2}{t_k-t_{k-1}}+\frac \epsilon n.
\end{align*}
Moreover, by Proposition \ref{connecting_curves_bounded_entropy}, the curves $\sigma^{k,\epsilon}$ can be chosen such that
\begin{align}\label{bound_entropy}
\sup_{\epsilon\in (0,1)}\sup_{k\in\IN}\sup_{t\in [t_{k-1},t_k]}\Ent(\sigma^{k,\epsilon}_t)<\infty.
\end{align}

Denote by $\rho^\epsilon$ the concatenation of $\sigma^{1,\epsilon},\dots,\sigma^{n,\epsilon}$. Then
\begin{align*}
\int_0^1\norm{D\rho^\epsilon_r}_{\rho^\epsilon_r}^2\,dr&=\sum_{k=1}^n\int_{t_{k-1}}^{t_k}
\norm{D\sigma^{k,\epsilon}_r}_{\sigma^{k,\epsilon}_r}^2\,dr\\
&\leq \epsilon+\sum_{k=1}^n \frac{\W(\rho_{t_{k-1}},\rho_{t_k})^2}{t_k-t_{k-1}}\\
&\leq\epsilon+\sum_{k=1}^n \int_{t_{k-1}}^{t_k}\abs{\dot\rho_r}_\W^2\,dr\\
&=\epsilon+\int_0^1 \abs{\dot\rho_r}_\W^2\,dr.
\end{align*}
Moreover, for every $t\in[0,1]$ there is a $k\in\{1,\dots,n\}$ such that $\W(\rho_{t_k},\rho_t)<\epsilon$, hence
\begin{align*}
\W(\rho^\epsilon_t,\rho_t)\leq \W(\rho^\epsilon_t,\rho^\epsilon_{t_k})+\W(\rho^\epsilon_{t_k},\rho_{t_k})+\W(\rho_{t_k},\rho_t)<3\epsilon.
\end{align*}
Thus, $\W(\rho^\epsilon_t,\rho_t)\to 0$ as $\epsilon\searrow 0$.

By Lemma \ref{sublevel_ent_cpt} and the uniform bound on the entropy (\ref {bound_entropy}), for every $t\in[0,1]$ and every sequence $(\epsilon_n)$ converging to $0$ there is a subsequence $(\epsilon_{n(k)})$ and $\tilde \rho_t\in \D(\M,\tau)$ such that $\rho^{\epsilon_{n(k)}}_t\to \tilde \rho_t$ weakly in $L^1$ as $k\to\infty$. In particular, $\tau(\rho^{\epsilon_{n(k)}}_t a)\to \tau(\tilde \rho_t a)$ for all $a\in \A_\mathrm{LM}$.

On the other hand, $\W(\rho^{\epsilon_{n(k)}}_t,\rho_t)\to 0$ implies $\tau(\rho^{\epsilon_{n(k)}}_t a)\to \tau(\rho_t a)$ for all $a\in \A_\mathrm{LM}$ by Proposition \ref{W_implies_weak_con}. Since $\A_\theta\subset \M$ is $\sigma$-weakly dense by Corollary \ref{A_LM_dense}, it follows that $\tilde\rho_t=\rho_t$ for all $t\in [0,1]$. Therefore, $\rho_t^\epsilon\to \rho_t$ weakly in $L^1$ as $\epsilon\searrow 0$ for all $t\in [0,1]$.

By Theorem \ref{energy_lsc}, the curve $(\rho_t)$ is admissible and
\begin{align*}
\int_0^1 \norm{D\rho_r}_{\rho_r}^2\,dr\leq \liminf_{\epsilon\searrow 0}\int_0^1 \norm{D\rho^\epsilon_r}_{\pi^\epsilon_r}^2\,dr\leq \int_0^1 \abs{\dot\rho_r}_\W^2\,dr.
\end{align*}
As the reverse inequality is obvious, we conclude $\norm{D\rho_r}_{\rho_r}=\abs{\dot\rho_r}_\W$ for a.e. $r\in [0,1]$.

In the general case $(\rho_t)\in \AC_\loc^2(I;(\D(\M,\tau),\W))$ one can simply partition $I$ into countably many compact intervals to obtain the same result.
\end{proof}

Let $(X,d)$ be an extended metric space. A curve $(\gamma_t)_{t\in [0,1]}$ in $X$ is called (constant speed) \emph{geodesic} if $d(\gamma_0,\gamma_1)<\infty$ and $d(\gamma_s,\gamma_t)=\abs{s-t}d(\gamma_0,\gamma_1)$ for all $s,t\in [0,1]$. The extended metric space $(X,d)$ is called \emph{geodesic space} if any two $x,y\in X$ with $d(x,y)<\infty$ can be joined by a geodesic.

\begin{lemma}\label{Arzela_Ascoli_entropy}
For $\alpha>0$ let $S_\alpha=\{\rho\in \D(\M,\tau)\mid \Ent(\rho)\leq \alpha\}$. If $L>0$ and $((\rho^n_t)_{t\in[0,1]})_n$ is a sequence of admissible curves in $S_\alpha$ such that
\begin{align*}
\int_s^t \norm{D\rho^n_r}_{\rho^n_r}^2\,dr\leq L^2\abs{t-s}
\end{align*}
for all $s,t\in [0,1]$ and $n\in\IN$, then there exists an admissible curve $(\rho_t)$ in $S_\alpha$ and a subsequence $(\rho^{n_k})_k$ of $(\rho^n)$ such that
\begin{align*}
\rho^{n_k}_t\to \rho_t
\end{align*}
weakly in $L^1$ for all $t\in [0,1]$, and
\begin{align*}
\int_0^1 \norm{D\rho_t}_{\rho_t}^2\leq \liminf_{n\to\infty}\int_0^1 \norm{D\rho^n_t}_{\rho^n_t}^2\,dt.
\end{align*}
\end{lemma}
\begin{proof}
Otherwise passing to a subsequence, we can assume that $\int_0^1 \norm{D\rho^n_t}_{\rho^n_t}^2\,dt$ converges. If $a\in \A_\mathrm{LM}$, then
\begin{align*}
\abs{\tau((\rho^n_t-\rho^n_s)a)}\leq\int_s^t \norm{\partial a}_{\rho^n_r}\norm{D\rho^n_r}_{\rho^n_r}\,dr\leq L\norm{a}_\mathrm{LM}^2\abs{t-s}
\end{align*}
for all $s,t\in [0,1]$ and $n\in\IN$. Thus $(\rho^n)$ is uniformly equicontinuous with respect to the metric 
\begin{align*}
d\colon \D(\M,\tau)\times \D(\M,\tau)\lra [0,\infty),\,d(\rho,\sigma)=\sup_{\norm{a}_\mathrm{LM}\leq 1}\abs{\tau(a(\rho-\sigma))}.
\end{align*}

By Lemma \ref{sublevel_ent_cpt}, the set $S_\alpha$ is sequentially compact with respect to the weak topology on $L^1$. Hence we can apply \cite[Proposition 3.3.1]{AGS08} (the lower semicontinuity property of $d$ is obvious, while the completeness of $(S_\alpha,d)$ follows from the weak compactness and the lower semicontinuity) to get a subsequence $(\rho^{n_k})$ and a curve $(\rho_t)$ in $S_\alpha$ such that $\rho^{n_k}_t\to \rho_t$ weakly in $L^1$ for all $t\in [0,1]$.

The remaining assertions follow from Theorem \ref{energy_lsc}.
\end{proof}

\begin{theorem}\label{D(Ent)_geodesic}
If the entropy has regular sublevel sets and $(P_t)$ is an $\EVI_K$ gradient flow of $\Ent$, then for all $\rho_0,\rho_1\in D(\Ent)$ with $\W(\rho_0,\rho_1)<\infty$ there exists a geodesic $(\rho_t)_{t\in[0,1]}$ with $\sup_{t\in[0,1]}\Ent(\rho_t)<\infty$. In particular, $(D(\Ent),\W)$ is a geodesic space.
\end{theorem}
\begin{proof}
Using the contraction estimate from Lemma \ref{EVI_metric_contraction}, one can proceed exactly as in the proof of Proposition \ref{connecting_curves_bounded_entropy} to see that for all $\rho_0,\rho_1\in D(\Ent)$ with $\W(\rho_0,\rho_1)<\infty$ and all $n\in\IN$ there exists an $L_n$-Lipschitz curve $(\rho^n_t)_{t\in [0,1]}$ connecting $\rho_0$ and $\rho_1$ such that $L_n^2\leq e^{-2K/n}(\W(\rho_0,\rho_1)^2+\frac 1 {n^2})$ and
\begin{align*}
\sup_{n\in\IN}\sup_{t\in [0,1]}\Ent(\rho^n_t)<\infty.
\end{align*}
Since the entropy has regular sublevel sets, the curves $(\rho^n_t)$ are admissible and
\begin{align*}
\norm{D\rho^n_t}_{\rho^n_t}=\abs{\dot\rho^n_t}_\W\leq L_n.
\end{align*}
As $(L_n)$ is bounded, we can apply Lemma \ref{Arzela_Ascoli_entropy} to get an admissible curve $(\rho_t)_{t\in [0,1]}$ with uniformly bounded entropy connecting $\rho_0$ and $\rho_1$ such that
\begin{align*}
\int_0^1 \norm{D\rho_t}_{\rho_t}^2\,dt&\leq\liminf_{n\to\infty}\int_0^1 \norm{D\rho^n_t}_{\rho^n_t}^2\,dt\\
&\leq \liminf_{n\to\infty}e^{-2K/n}\left(\W(\rho_0,\rho_1)^2+\frac 1 {n^2}\right)\\
&=\W(\rho_0,\rho_1)^2.
\end{align*}
Hence $(\rho_t)$ is a geodesic connecting $\rho_0$ and $\rho_1$.
\end{proof}

\begin{corollary}
If the entropy has regular sublevel sets and $(P_t)$ is an $\EVI_K$ gradient flow of $\Ent$, then the metric $\W$ is lower semicontinuous with respect to weak $L^1$-convergence on sublevel sets of the entropy.
\end{corollary}

Although the Theorem \ref{D(Ent)_geodesic} guarantees the existence of geodesics connecting density matrices with finite entropy provided their distance is finite, it does not rule out the possibility that density matrices with finite entropy have infinite distance. In the next proposition we will see that this cannot happen if $\E$ is ergodic and satisfies $\GE(K,\infty)$ for strictly positive $K$.

Here $\E$ is called \emph{ergodic} if for all $a\in L^2(\M,\tau)$ one has $\lim_{t\to\infty}P_t(a)=\tau(a)$ in $L^2(\M,\tau)$. Equivalently, there are no projections $p\in \M$ other than $0$ and $1$ such that $P_t p\leq p$ for all $t\geq 0$.

\begin{proposition}\label{MLSI}
Assume that $\tau(1)=1$. If $\E$ is ergodic and satisfies $\GE(K,\infty)$ for some $K>0$, then the Talagrand inequality
\begin{align*}
\W(\rho,1)^2\leq \frac {2}{K}\Ent(\rho)
\end{align*}
holds for all $\rho\in D(\Ent)$. In particular, $\W$ is finite on $D(\Ent)\times D(\Ent)$.
\end{proposition}
\begin{proof}
First we show that for $\rho\in \D(\M,\tau)\cap L^2(\M,\tau$) we have
\begin{align}
\label{log_Sobolev}
\Ent(P_t \rho)\leq \frac 1 {2K}\I(P_t \rho)
\end{align}
for a.e. $t\geq 0$.

Since $(P_t \rho)_{t\geq 0}$ is an admissible curve with $\norm{D P_t \rho}_{P_t \rho}^2\leq \I(P_t \rho)$ for a.e. $t\geq 0$, we have
\begin{align*}
\limsup_{h\searrow 0}\frac 1 h \W(P_{t+h}\rho,P_t\rho)\leq \limsup_{h\searrow 0}\frac 1 h\int_t^{t+h}\norm{D P_r\rho}_{P_r\rho}\,dr\leq \I(P_t \rho)^{1/2}
\end{align*}
for a.e. $t\geq 0$.

Thus
\begin{align*}
-\frac 1 2\frac{d^+}{dt}\W(P_t \rho,P_s\rho)^2&=\limsup_{h\to 0}\frac 1 {2h}(\W(P_t \rho,P_s\rho)^2-\W(P_{t+h}\rho,P_s\rho)^2)\\
&\leq \limsup_{h\searrow 0}\frac 1{2h}(\W(P_{t+h}\rho,P_t \rho)^2+2\W(P_{t+h}\rho,P_t \rho)\W(P_{t+h}\rho,P_s\rho))\\
&\leq \I(P_t \rho)^{1/2}\W(P_t \rho,P_s\rho)
\end{align*}
for a.e. $t\geq 0$ and all $s>0$.

The evolution variational inequality from Theorem \ref{EVI_gradient_flow} implies
\begin{align*}
\Ent(P_t \rho)&\leq -\frac 1 2\frac{d^+}{dt}\W(P_t \rho,\rho)^2-\frac K 2\W(P_t \rho,\rho)^2+\Ent(P_s\rho)\\
&\leq \I(P_t\rho)^{1/2}\W(P_t \rho,\rho)-\frac K 2\W(P_t \rho,\rho)^2+\Ent(P_s\rho)\\
&\leq \frac 1{2K}\I(P_t \rho)+\Ent(P_s \rho)
\end{align*}
for a.e. $t\geq 0$ and all $s\geq 0$.

Since $\E$ is ergodic and convergence in $L^2$ implies convergence of the entropy, we have $\Ent(P_s \rho)\to 0$ as $s\to \infty$ and (\ref{log_Sobolev}) follows.

Since $t\mapsto \Ent(P_t \rho)$ is a locally absolutely continuous function with derivative a.e. equal to $-\I(P_t \rho)$ by Proposition \ref{entropy_Fisher_decay}, Grönwall's lemma asserts
\begin{align*}
\Ent(P_t \rho)\leq e^{-2Kt}\Ent(\rho).
\end{align*}
By the same arguments used in the proof of Corollary \ref{heat_flow_L1}, this inequality remains true if we only assume $\rho\in D(\Ent)$. In particular, $\Ent(P_t \rho)\to 0$ as $t\to \infty$.

If $\tilde \rho$ is any limit point in the weak $L^1$-topology of $(P_t \rho)$ as $t\to \infty$, then
\begin{align*}
\Ent(\tilde \rho)\leq \limsup_{t\to\infty}\Ent(P_t \rho)=0,
\end{align*}
which implies $\tilde \rho=1$. On the other hand, since the sublevel sets of the entropy are compact in the weak $L^1$-topology by Lemma \ref{sublevel_ent_cpt}, there exist limit points in the weak $L^1$-topology of $(P_t \rho)$ as $t\to\infty$. Both facts combined give $P_t \rho\to 1$ weakly as $t\to\infty$.

Once again assume that $\rho \in \D(\M,\tau)\cap L^2(\M,\tau)$ and let $\vartheta\colon [0,1)\lra [0,\infty)$ be a strictly increasing differentiable function with $\lim_{t\to1}\vartheta(t)=\infty$. Let $\rho_t=P_{\vartheta(t)}\rho$. Since $(P_t \rho)_{t\geq 0}$ is admissible and
\begin{align*}
\int_0^\infty\norm{D P_t \rho}_{P_t\rho}^2\,dt\leq \int_0^\infty\I(P_t \rho)\,dt \leq \Ent(\rho)<\infty
\end{align*}
by Propositions \ref{heat_flow_admissible} and \ref{entropy_Fisher_decay}, the reparametrized curve $(\rho_t)_{t\in [0,1]}$ is also admissible.

Since
\begin{align*}
-2\Ent(P_t \rho)^{1/2}\frac{d}{dt}\Ent(P_t \rho)^{1/2}=\I(P_t\rho)
\end{align*}
for a.e. $t\geq 0$ such that $P_t \rho\neq 1$, we can use the inequality $\Ent(P_t\rho)\leq \frac 1 {2K}\I(P_t \rho)$ proven above to see that
\begin{align*}
\I(P_t \rho)^{1/2}\leq -\left(\frac 2 K\right)^{1/2}\frac{d}{dt}\Ent(P_t \rho)^{1/2}
\end{align*}
for a.e. $t\geq 0$ (clearly the inequality holds if $P_t\rho=1$, since $\I(1)=0$).

Thus
\begin{align*}
\W(\rho,1)&\leq\int_0^1\norm{D\rho_t}_{\rho_t}\,dt\\
&=\int_0^\infty \norm{DP_t \rho}_{P_t \rho}\,dt\\
&\leq\int_0^\infty\I(P_t \rho)^{1/2}\,dt\\
&\leq -\left(\frac 2 K\right)^{1/2}\int_0^\infty\frac{d}{dt}\Ent(\rho_t)^{1/2}\,dt\\
&=\left(\frac 2 K\right)^{1/2}\Ent(\rho)^{1/2}.
\end{align*}
In the general case $\rho\in D(\Ent)$, we can argue by approximation. Let $\rho^n=\frac{\rho\wedge n}{\tau(\rho\wedge n)}$. Since $\Ent(\rho)<\infty$, it is easy to see that the sequence $(\rho^n)$ has bounded entropy. By Proposition \ref{connecting_curves_bounded_entropy} we can choose admissible curves $(\rho^n_t)_{t\in [0,1]}$ connecting $\rho^n$ and $1$ with uniformly bounded entropy such that
\begin{align*}
\int_s^t \norm{D\rho^n_r}_{\rho^n_r}^2\,dr\leq \left(\W(\rho^n,1)+\frac 1 {n^2}\right)\abs{t-s}
\end{align*}
for all $s,t\in [0,1]$ and $n\in\IN$.

It follows from Lemma \ref{Arzela_Ascoli_entropy} that there exists a strictly increasing sequence $(n_k)$ in $\IN$ and an admissible curve $(\rho_t)_{t\in [0,1]}$ such that $\rho^{n_k}_t \to \rho_t$ weakly in $L^1(\M,\tau)$ for all $t\in [0,1]$ and
\begin{align*}
\W(\rho,1)^2&\leq \int_0^1\norm{D\rho_t}_{\rho_t}^2\,dt\\
&\leq \liminf_{n\to\infty}\int_0^1 \norm{D\rho^n_t}_{\rho^n_t}^2\,dt\\
&\leq \liminf_{n\to\infty}\W(\rho^n,1)^2\\
&\leq \frac 2 K \lim_{n\to\infty}\Ent(\rho^n)\\
&=\frac 2 K \Ent(\rho).\qedhere
\end{align*}
\end{proof}

The last property discussed in this section is $K$-convexity of the entropy. Let $(X,d)$ be an extended metric space. A functional $S\colon X\lra(-\infty,\infty]$ is called \emph{$K$-convex} along the geodesic $(\gamma_t)_{t\in [0,1]}$ in $(D(S),d)$ if
\begin{align*}
S(\gamma_t)\leq (1-t)S(\gamma_0)+tS(\gamma_1)-\frac K 2t(1-t)d(\gamma_0,\gamma_1)^2
\end{align*}
for all $t\in [0,1]$.

The functional $S$ is called \emph{strongly geodesically $K$-convex} if it is $K$-convex along every geodesic in $(D(S),d)$. It is called \emph{geodesically $K$-convex} if every pair $x_0,x_1\in D(S)$ can be joined by a geodesic $(\gamma_t)$ such that $S$ is $K$-convex along $(\gamma_t)$.

\begin{remark}
If $(D(S),d)$ is a geodesic space, then every strongly geodesically $K$-convex functional is geodesically $K$-convex. If $(D(S),d)$ is not a geodesic space, it does not make too much sense to talk about geodesic convexity at all.
\end{remark}

If $(P_t)$ is an $\EVI_K$ gradient flow of the entropy and the sublevel sets of the entropy are regular, the strong $K$-convexity follows from abstract results on gradient flows in metric spaces.

\begin{theorem}\label{Ent_semi-convex}
If $(P_t)$ is an $\EVI_K$ gradient flow of $\Ent$ and $\Ent$ has regular sublevel sets, then $(D(\Ent),\W)$ is a geodesic space and $\Ent$ is strongly geodesically $K$-convex.
\end{theorem}
\begin{proof}
By Theorem \ref{D(Ent)_geodesic}, the space $(D(\Ent),\W)$ is geodesic, and by \cite[Proposition 2.23]{AGS14b} the entropy is $K$-convex along all geodesics in $(D(\Ent),\W)$.
\end{proof}

Let us summarize the results of this section.
\begin{theorem}\label{main_theorem_summary}
Assume that $\tau$ is a state, $L^1(\M,\tau)$ is separable and $\theta$ is the logarithmic mean. For $K\in\IR$ consider the following properties.
\begin{itemize}
\item[(i)]The semigroup $(P_t)$ satisfies the gradient estimate $\GE(K,\infty)$.
\item[(ii)]The semigroup $(P_t)$ is an $\EVI_K$ gradient flow of $\Ent$, the sublevel sets of $\Ent$ are regular and $\W$ is non-degenerate.
\item[(iii)]The pseudo metric $\W$ is non-degenerate, $(D(\Ent),\W)$ is geodesic and $\Ent$ is strongly geodesically $K$-convex.
\end{itemize}
Then (i)$\implies$(ii)$\implies$(iii).
\end{theorem}

\begin{remark}
The properties (i), (ii) and (iii) can all be understood as lower Ricci curvature bounds for the geometry determined by $\E$. This approach has been studied intensively for metric measure spaces (see e.q. \cite{AGS14b,LV09,Stu06a,Stu06b}) and, more recently, also for graphs (see e.g. \cite{EM12,EHMT17}). We hope that the present framework allows to address the highly interesting question of introducing a concept of Ricci curvature (bounds) in noncommutative geometry. First steps in this direction was already taken by Hornshaw \cite{Hor18}.
\end{remark}

\begin{remark}
The properties (i), (ii) and (iii) are equivalent for the Cheeger energy on infinitesimally Hilbertian length metric measure spaces (\cite{AGS15}, Theorem 1.1) and the Dirichlet form associated with a finite graph (\cite[Theorem 4.5]{EM12} and \cite[Theorem 3.1]{EF18}). It would be interesting to know if this is still true in this more general setting.
\end{remark}

\bibliography{mybib}{}
\bibliographystyle{alpha}

\end{document}